\documentclass[11pt]{article}

\usepackage{amsmath, amssymb, amsthm, bm, graphicx, xcolor, physics, newtxtext, newtxmath, tikz, eso-pic, enumitem, float}

\usepackage{geometry}
\usepackage{hyperref}
\hypersetup{breaklinks, colorlinks=true, linkcolor=blue, urlcolor=cyan, citecolor=magenta}

\usepackage[lastpage, user]{zref}
\usepackage{fancyhdr}
\fancypagestyle{plain}{}

\usepackage[numbered, framed]{matlab-prettifier}
\numberwithin{equation}{section}

\theoremstyle{plain}
\newtheorem{theorem}{Theorem}[section]

\theoremstyle{definition}
\newtheorem{definition}[theorem]{Definition}
\newtheorem{proposition}[theorem]{Proposition}
\newtheorem{remark}[theorem]{Remark}
\newtheorem{corollary}[theorem]{Corollary}
\newtheorem{lemma}[theorem]{Lemma}

\newtheorem{conjecture}[theorem]{Conjecture}

\renewenvironment{proof}{{\noindent \bf Proof.}}{\qed}

\newcommand*\CircleAroundChar[2][\small]{\tikz[baseline=(char.base)]{\node[shape=circle, draw, inner sep=1pt](char){#1#2};}}

\makeatletter
\newcommand{\newparallel}{\mathrel{\mathpalette\new@parallel\relax}}
\newcommand{\new@parallel}[2]{%
  \begingroup
  \sbox\z@{$#1T$}
  \resizebox{!}{\ht\z@}{$\m@th#1/\mkern-5mu/$}%
  \endgroup
}
\makeatother

\usepackage{xstring,catoptions} 

\newcommand{\monomer}[2][]
{\ifx&#1&%
{$#2$}%
\else
{\color{#1}$#2$}%
\fi
}

\newcommand{\quark}[1]{\raisebox{-5pt}{\tikz\draw ellipse(0.3 and 0.21)node{#1};}}

\definecolor{blue A}{RGB}{0,102,255}
\definecolor{red B}{RGB}{220,20,60}
\definecolor{yellow C}{RGB}{255,140,0}

\definecolor{BW black}{RGB}{0,0,0}
\definecolor{BW orange}{RGB}{230,159,0}
\definecolor{BW sky blue}{RGB}{86,180,233}
\definecolor{BW bluish green}{RGB}{0,158,115}
\definecolor{BW yellow}{RGB}{240,228,66}
\definecolor{BW blue}{RGB}{0,114,178}
\definecolor{BW vermillion}{RGB}{213,94,0}
\definecolor{BW reddish purple}{RGB}{204,121,167}

\definecolor{spring}{rgb}{0.7, 1.0, 0.2}
\definecolor{cyan}{rgb}{0.6, 1.0, 1.0}
\definecolor{purple}{rgb}{0.9, 0.1, 1.0}

\usepackage{fancyhdr}

\begin{document}

\author{Zirui Xu\thanks{Department of Applied Physics and Applied
Mathematics, Columbia University, New York, NY
10027, USA. Email: zx2250@columbia.edu.} \;(corresponding author)
 \quad and \quad Qiang Du\thanks{Department of Applied Physics and Applied
Mathematics, and Data Science Institute,
Columbia University, New York, NY
10027, USA. Email: qd2125@columbia.edu.}
}

\title{On the Ternary Ohta\textendash Kawasaki Free Energy and Its One-dimensional Global Minimizers\thanks{This work is supported in part by
the National Science Foundation DMS-2012562 and DMS-1937254.
}} 

\date{}

\maketitle

\section*{Abstract}

\thispagestyle{empty}

We study the ternary Ohta\textendash Kawasaki free energy that has been used to model triblock copolymer systems. Its one-dimensional global minimizers are conjectured to have cyclic patterns. However, some physical experiments and computer simulations found triblock copolymers forming noncyclic lamellar patterns. In this work, by comparing the free energies of the cyclic pattern and some noncyclic candidates, we show that the conjecture does not hold for some choices of parameters. Our results suggest that even in one dimension, the global minimizers may take on very different patterns in different parameter regimes. To unify the existing choices of the long range coefficient matrix, we present a reformulation of the long range term using a generalized charge interpretation, and thereby propose conditions on the matrix in order for the global minimizers to reproduce physically relevant nanostructures of block copolymers.

\section{Introduction}
\label{introduction}

For the past few decades, block copolymers have attracted broad attention because of their ability to self-assemble into various fine mesoscopic structures. Those nanostructures can be used as scaffolds in a wide range of applications \cite{mai2012self,chang2020abcs,reddy2021block}. A number of approaches have been developed to model the self-assembling behaviors of block copolymers, such as molecular dynamics \cite{lyubimov2018molecular}, dissipative particle dynamics \cite{huang2019dissipative}, self-consistent field theory \cite{sides2003parallel}, etc. As an approximation to the self-consistent field theory, Ohta\textendash Kawasaki (O\textendash K) free energy \cite{ohta1986equilibrium} is effective and capable of reproducing microphase separation found in block copolymers.

The ternary O\textendash K free energy \cite{ren2003triblock1} is a phase field model for triblock copolymers. A triblock copolymer molecule, denoted by \monomer{A}\monomer{B}\monomer{C}, is a linear chain obtained by joining three subchains of monomers of types \monomer{A}, \monomer{B} and \monomer{C}, respectively, via covalent bonds. Let $u_1$, $u_2$ and $u_3$ denote the density fields of monomers \monomer{A}, \monomer{B} and \monomer{C}, respectively, with $\omega_i$ being the spatial average of $u_i$ (i.e., the overall volume fraction) and satisfying $\sum_i\omega_i=1,\;\omega_i>0$, then according to Ren and Wei \cite[Equation (2.1)]{ren2003triblock2}, the free energy of the triblock copolymer system is
\begin{equation}
\label{definition of J_epsilon}
\begin{aligned}
J_\epsilon(\vec u) = &\sum_{i=1}^3\frac{\epsilon}{\omega_i}\int_\Omega\,\big|\nabla u_i(\vec x)\big|^2\dd{\vec x}+\int_\Omega\frac{W(\vec u)}\epsilon\dd{\vec x}+\null\\
&\sum_{i=1}^3\sum_{j=1}^3\gamma_{ij}\int_\Omega\int_\Omega\big(u_i(\vec x)\!-\!\omega_i\big)G(\vec x,\vec y)\big(u_j(\vec y)\!-\!\omega_j\big)\dd{\vec x}\dd{\vec y},
\end{aligned}
\end{equation}
where $\Omega$ is the entire domain under consideration, and $G$ is the Green's function of $-\Delta$ (the negative Laplacian) on $\Omega$, subject to suitable boundary conditions (for illustration, we focus on periodic boundary conditions). The parameter $\epsilon$ is proportional to the interfacial thickness, $\vec u$ denotes $[u_1,u_2,u_3]^{\rm T}$ subject to the incompressibility condition $u_1+u_2+u_3=1$, and $W$ is a triple-well potential whose three wells are $[1,0,0]^{\rm T}$, $[0,1,0]^{\rm T}$ and $[0,0,1]^{\rm T}$, corresponding to pure \monomer{A}, \monomer{B} and \monomer{C} domains, respectively. The matrix $[\gamma_{ij}]$ is given in \eqref{Ren's matrix}. Note that \eqref{definition of J_epsilon} was first derived by Nakazawa and Ohta \cite{nakazawa1993microphase} for the same triblock copolymer system, but with a different $[\gamma_{ij}]$, as discussed later.

In this work, we focus on the sharp interface limit (also known as the strong segregation limit \cite[Figure 3]{bates1990block}) of the 
ternary  O\textendash K model. In such a limit, different types of monomers are well separated by sharp interfaces, and the domains $\Omega_1,\,\Omega_2$ and $\Omega_3$ occupied by monomers \monomer{A}, \monomer{B} and \monomer{C}, respectively, partition the entire domain $\Omega$, with the volume constraint $|\Omega_i|=\omega_i|\Omega|$. For $\Omega=[0,1]$, \cite[Section 3]{ren2003triblock2} has established the $\Gamma$-convergence of $J_\epsilon$, as $\epsilon\rightarrow0$, to $J$ defined by
\begin{equation}
\label{definition of J}
\begin{aligned}
J\big(\{\Omega_i\}\big)=&\sum_{i=1}^2\sum_{j=i+1}^3 c_{ij}\vert\partial\Omega_i\cap\partial\Omega_j\vert+\null\\
&\sum_{i=1}^3\sum_{j=1}^3\gamma_{ij}\int_{\Omega}\int_{\Omega}\big(\bm1_{\Omega_i}(\vec x)\!-\!\omega_i\big)\,G(\vec x,\vec y)\,\big(\bm1_{\Omega_j}(\vec y)\!-\!\omega_j\big)\dd{\vec x}\dd{\vec y},
\end{aligned}
\end{equation}
where the first double summation, dubbed the short range term, is the weighted sum of interfacial sizes, with $c_{ij}$ being the interfacial tensions given by \cite[Definition 3.3]{ren2003triblock2} and satisfying triangle inequalities \cite[Equation (9)]{van2008copolymer}. The second one is dubbed the long range term, and $\bm1_{\Omega_i}$ is the indicator function of $\Omega_i$. For $\Omega=[0,1]$, the Green's function $G$ is given by \eqref{expression of G on 1-D periodic cell}. We expect such $\Gamma$-convergence to hold not only for $\Omega=[0,1]$, but also for bounded and smooth $\Omega$ in any dimension, in the light of the analogous results for binary systems (see Section \ref{background on binary systems}).

To date, mathematical studies on the minimizers of \eqref{definition of J} remain incomplete. In 1-D, Ren and Wei \cite{ren2003triblock2} found some local minimizers, but the global minimizers remain to be an open question. Later on, some stationary points were found in 2-D in the vanishing volume limit as $\omega_1,\omega_2\rightarrow0$, including clusters of tiny single bubbles \cite{ren2019stationary}, tiny double bubbles \cite{ren2015double}, and tiny core-shells \cite{ren2017stationary}, depending on the choices of $[\gamma_{ij}]$. Note that in those 2-D works, $[\gamma_{ij}]$ were chosen to be some general matrices instead of the matrix \eqref{Ren's matrix}. In the same vanishing volume limit, the global minimizers in 2-D were recently found \cite{alama2019periodic} to be clusters of tiny single and double bubbles for some choices of $[\gamma_{ij}]$. Without the vanishing volume assumption, the global minimizers were only found in 1-D for a degenerate case \eqref{Blend's matrix} \cite{van2008copolymer} (such a degenerate ternary system is used to model mixtures of diblock copolymers and homopolymers). The present study is the first attempt towards a systematic characterization of the global minimizers in non-degenerate cases for general compositions (i.e., volume fractions). We begin by numerically and exhaustively searching for the global minimizers in a 1-D periodic cell, which is computationally feasible when there are not many interfaces. Based on our numerical results, we select a number of representative lamellar candidates and analytically calculate their free energies. Given any choice of parameters, we choose the candidate that has the lowest free energy. By repeating this procedure for various choices of parameters, we construct plausible phase diagrams that may offer us a rough picture of the energy landscape. We note that some of our lamellar candidates have not been seen in the literature of triblock copolymers. It is unclear whether those candidates correspond to more complicated multiblock terpolymers (terpolymers are copolymers made up of three types monomers), or the parameters corresponding to those candidates in the phase diagrams are impractical at least in 1-D. Our findings thus indicate that the parameters in the O\textendash K model have to be carefully chosen if one wishes to model triblock copolymers. Although we focus on 1-D global minimizers, our results may also shed light on 2-D and 3-D cases, as discussed later in Section \ref{discussion}.

We note that the parameter $[\gamma_{ij}]$ affects the phase diagrams mentioned above. Unfortunately, there have been different choices of $[\gamma_{ij}]$ in the literature (see Section \ref{existing choices of [gamma_ij]} for more details). In the present work, instead of examining the derivation of $[\gamma_{ij}]$ from statistical physics, we present a reformulation of the long range term, which allows us to draw a mathematical analogy to charged immiscible fluids, with different types of charges uniformly distributed within different types of fluids. With the charge analogy, the meaning of each entry of $[\gamma_{ij}]$ then becomes clear: it describes the interaction between the type $i$ charge and the type $j$ charge. (Such an analogy is straightforward and customary for binary systems, but to our knowledge, there have not been thorough demonstrations for ternary systems. Note that "immiscible fluids" is a term we adopted from \cite{lawlor2014double}, and "charge" is a term we borrowed from the liquid drop model \cite{choksi2017old} to mimic similar characteristics. It might be more appropriate for us to use "phase charge" to avoid ambiguity, or to borrow "color charge" from the strong interaction context, but for brevity we use "charge" throughout this article.) We then impose some admissibility conditions on $[\gamma_{ij}]$ from a purely mathematical perspective, in order to ensure that the global minimizers of \eqref{definition of J} would be characterized by what we call "charge neutrality" (i.e., the homogeneous mixture of \monomer{A}, \monomer{B} and \monomer{C} in the ternary system), if we could neglect the interfacial energy introduced by the short range term. On the other hand, if we only consider the interfacial energy, different types of fluids would separate into a double bubble \cite{lawlor2014double} (or two intervals in the 1-D case). Consequently, we expect fine structures to arise as an outcome of the competition between the long and short range interactions. In this regard, the admissibility conditions on $[\gamma_{ij}]$ are necessary for the global minimizers of \eqref{definition of J} to reproduce the nanostructures of triblock copolymers observed in physical experiments, and this is consistent with recent numerical studies in \cite{wang2019bubble,ren2019stationary}. More importantly, we can see that the free energy originally derived for triblock copolymers can actually describe the universal competition between the interfacial tension and the principle of charge neutrality. By "universal" we mean that it may be representative in many different physical contexts, as discussed later in Remark \ref{charge interpretation}. Our charge analogy also offers an intuitive understanding of the computed phase diagrams. Lastly, we prove that the existing choice \eqref{General matrix} represents all the matrices satisfying the admissibility conditions. It is mathematically interesting (and perhaps practically significant) to explore the entire range of admissible $[\gamma_{ij}]$, which in the block copolymer context may correspond to different block sequences or architectures of multiblock terpolymers.

The rest of the paper is structured as follows: in Section \ref{Background and related studies} we provide more background on related studies; in Section \ref{Numerical inspiration: a counterexample} we present examples of 1-D global minimizers obtained numerically, which in part motivated this work; in Section \ref{physical analogy} we draw an analogy between the free energy and the system of charged immiscible fluids; in Section \ref{conditions on [gamma_{ij}]} based on our analogy we propose some admissibility conditions on $[\gamma_{ij}]$ and make comparison with the existing choices of $[\gamma_{ij}]$ in the literature; in Section \ref{Phase diagrams: most favored repetends} we present relatively comprehensive comparisons among 1-D candidates and provide some intuitive understanding of the comparison results; in Section \ref{discussion} we conclude with some remarks; in Appendix \ref{underlying mechanism} we present a possible underlying mechanism of the interactions between charges; in Appendix \ref{Optimal arrangement of charged balls in 1-D} we discuss a related discrete problem; in Appendices \ref{Calculation of the free energy J in 1-D} and \ref{Calculating the free energy of periodic patterns} we provide some computational details; in Appendix \ref{Alternative derivation} we present an alternative derivation of the admissibility conditions.

\section{Background and related studies}
\label{Background and related studies}

In this section we present some more background of our study. First we recall the binary systems which are simpler and can be illuminating, then we discuss the ternary systems with the focus on 1-D cases.

\subsection{Binary systems}
\label{background on binary systems}

The original O\textendash K free energy \cite{ohta1986equilibrium} is proposed to model the \monomer{A}\monomer{B} diblock copolymer, which is a chain obtained by joining two subchains via a covalent bond, with one subchain made up of monomers of type \monomer{A} and the other of type \monomer{B}. Let $u$ be the difference in the volume fraction between monomers \monomer{A} and \monomer{B} under the incompressibility condition, then the free energy takes on the form \cite{ren2000multiplicity}
\begin{equation}
\label{definition of I_epsilon}
\begin{aligned}
I_\epsilon(u) = &\int_\Omega\Big(\epsilon\,\big|\nabla u(\vec x)\big|^2+\epsilon^{-1}\big(1\!-\!u(\vec x)^2\big)^2\Big)\dd{\vec x}+\null\\
&\gamma\int_\Omega\int_\Omega\big(u(\vec x)\!-\!\omega\big)G(\vec x,\vec y)\big(u(\vec y)\!-\!\omega\big)\dd{\vec x}\dd{\vec y},
\end{aligned}
\end{equation}
with the spatial average of $u$ being a prescribed constant $\omega\in(-1,1)$, and $\gamma$ is related to the total chain length. We call \eqref{definition of I_epsilon} a diffuse interface model, where $\epsilon$ controls the thickness of interfaces separating monomers \monomer{A} and \monomer{B}.

As $\epsilon\rightarrow0$, the functional $I_\epsilon$ will $\Gamma$-converge to a sharp interface limit $I$ \cite[Section 2]{ren2000multiplicity},
\begin{equation}
\label{definition of I}
\begin{aligned}
I(u) =\null&\frac83\,\text{Per}_{\Omega}\big(\{u=1\}\big)+\null\\
&\gamma\int_\Omega\int_\Omega\big(u(\vec x)\!-\!\omega\big)G(\vec x,\vec y)\big(u(\vec y)\!-\!\omega\big)\dd{\vec x}\dd{\vec y},
\end{aligned}
\end{equation}
with the image of $u$ being $\{-1,1\}$. The perimeter term $\text{Per}_{\Omega}$ has its origins in short range interactions and favors phase separation between \monomer{A} and \monomer{B}. However, because of the chemical bond between \monomer{A} and \monomer{B} subchains, \monomer{A} and \monomer{B}-rich domains cannot expand to the macroscopic level. In fact, even an attempt to stretch those domains is entropically unfavorable, since the polymer chains would have to straighten, thus reducing the number of possible molecular configurations. This fact is reflected in the long range term, which would be very large in the case of macrophase separation. As a result of the competition between those two terms, microphase separation occurs and leads to fine structures.

As $\omega$ decreases from $0$ to $-1$ (i.e., the overall volume fraction of \monomer{A} drops from $1/2$ to $0$), roughly speaking, the global minimizer of the free energy is expected to undergo transitions from lamellae to gyroid to cylinders to spheres, see, e.g., \cite[Figure 3]{bates2000block}. When $\omega\approx0$, domains \monomer{A} and \monomer{B} have roughly the same size and take on the lamellar shape; when $\omega\approx-1$, domain \monomer{A} occupies little space and resembles tiny droplets. In the former case, the problem can be reduced to 1-D, and all the 1-D local minimizers of \eqref{definition of I} have been found in \cite{ren2000multiplicity}. The latter case is reminiscent of Wigner crystallization (i.e., due to Coulomb repulsion, confined electrons may form a triangular lattice in 2-D or a body-centered cubic lattice in 3-D). In fact, in the vanishing volume limit, \eqref{definition of I} does converge to some crystallization problem in 2-D or 3-D (see, e.g., \cite[Page 1367]{choksi2010small}). For more studies on 2-D and 3-D cases, see, e.g., \cite{knupfer2016low,morini2014cascade,ren2014double,topaloglu2013nonlocal,ren2011toroidal,sternberg2011global,kang2010pattern,kang2009ring,ren2009oval,ren2008spherical,ren2007many,ren2007single}.

Some variants of \eqref{definition of I} have also been studied in the literature for various applications. A well-known example is Gamow's liquid drop model for atomic nuclei, where $\Omega=\mathbb R^3$ with no boundary conditions (hence $\omega=-1$). For a small volume (i.e., $\text{Vol}\big(\{u=1\}\big)\ll1$), the global minimizer is a perfect ball; for a large volume, the global minimizer does not exist, indicating nuclear fission (see \cite{frank2019non,choksi2017old} and references therein by Kn\"upfer, Muratov and Novaga as well as Choksi and Peletier, see also \cite{carazzato2021minimality} for a generalized result). Another example is the case when $G$ is replaced by the screened Coulomb kernel with $\gamma\gg1$ and $1+\omega\ll1$ but without the volume constraint. This screened version was derived by Muratov \cite{muratov2010droplet} as the $\Gamma$-equivalent of \eqref{definition of I_epsilon} in a regime where $\gamma=\epsilon^{-1}$ and $\omega=\epsilon^{2/3}|\ln\epsilon|^{1/3}\bar\delta-1$ for some fixed $\bar\delta>0$ (whereas \eqref{definition of I} is derived from \eqref{definition of I_epsilon} in a regime where $\gamma$ and $\omega$ are fixed). For this screened version, Muratov et al. showed the 2-D global minimizer to be tiny disks (nearly perfect) on a triangular lattice \cite{goldman2013gamma,goldman2014gamma}. One more example originates from an Ising model with competing interactions \cite{giuliani2016periodic}. In this variant, the perimeter is measured by the 1-norm instead of the Euclidean norm, $\omega$ equals $0$ (although this volume constraint is an outcome of optimization, instead of an assumption), and $G$ is chosen to be either a power law decay function or a screened Coulomb potential. That is, denoting $d$ the spatial dimension, $G(\vec x,\vec y)$ roughly equals $\big(\Vert\vec x\!-\!\vec y\Vert^p\!+\!1\big)^{-1}$ with $p>2d$ \cite{goldman2019optimality}, $\big(\Vert\vec x\!-\!\vec y\Vert_1\!+\!1\big)^{-p}$ with $p>p^*$ for some $p^*\in(d\!+\!1,d\!+\!2)$ \cite{daneri2019exact,kerschbaum2021striped}, $\exp{-\Vert\vec x\!-\!\vec y\Vert_1}/\Vert\vec x\!-\!\vec y\Vert_1^{d-2}$ or $\exp{-\Vert\vec x\!-\!\vec y\Vert}/\Vert\vec x\!-\!\vec y\Vert^{d-2}$ \cite{daneri2020pattern}. In any of those cases, the global minimizer is shown to be equispaced lamellae under some conditions. For $G(\vec x,\vec y)$ roughly equal to $\big(\Vert\vec x\!-\!\vec y\Vert_1\!+\!1\big)^{-p}$ with $p\geq d+2$, analogous results have recently been obtained in the diffuse interface setting \cite{daneri2019one,daneri2022one}, and also in the volume-constrained case for any $\omega\in(-1,1)$ \cite{daneri2021exact}. The word "roughly" is used above to indicate that $G$ needs some suitable modifications since periodic boundary conditions were chosen in those works. Although the screened Coulomb kernel $\exp{-\Vert\vec r\Vert}/\Vert\vec r\Vert$ does have some physical origins, the existing results do not include other physically interesting cases such as $p=d\!+\!1$, $p=d$, and $p=d\!-\!2$, corresponding to thin magnetic films, 3-D micromagnetics, and Coulomb potential, respectively \cite[Page 2533]{daneri2020pattern}.

The results in the binary case may shed light on the ternary case. For example, the tiny droplets found in binary systems when $\omega\approx-1$, correspond to the tiny bubbles found in ternary systems in the vanishing volume limit $\omega_1,\omega_2\rightarrow0$. Furthermore, the periodic lamellar pattern in binary systems motivates us to study its analogues in ternary systems.

\subsection{Ternary systems}
\label{sec:ternary}
Compared to binary systems, mathematical studies on ternary systems are still in the early stages. For $\Omega=[0,1]$ with periodic boundary conditions (think of $\Omega$ as a circle), Ren and Wei proved \cite[Section 4]{ren2003triblock2} that all the cyclic patterns \monomer{A}\monomer{B}\monomer{C} $\cdots\;$\monomer{A}\monomer{B}\monomer{C} with fine periodicity, such as the one depicted in Figure \ref{Ren vs. Mogi}, are local minimizers of \eqref{definition of J}. But they pointed out \cite[Page 190]{ren2003triblock2} that the global minimizer is hard to find even in 1-D, because one has to consider all the possible patterns, not just the cyclic ones. Note that the binary system \eqref{definition of I} has no such issue in 1-D, because the patterns can only be \monomer{A}\monomer{B}\monomer{A}\monomer{B} $\cdots$ without \monomer{C}. Nonetheless, Ren and Wei conjectured \cite[Conjecture 4.10]{ren2003triblock2} that the global minimizer has a cyclic pattern. This conjecture might have been partly motivated by the lamellar phase of triblock copolymers found in the experiment by Mogi et al. \cite{mogi1994superlattice}. However, upon close scrutiny, we note that the lamellar phase found by Mogi et al. \cite{mogi1994superlattice} (and also some other experimental or numerical studies such as \cite{mogi1992preparation,zheng1995morphology,matsen1998gyroid,bailey2001morphological,bailey2002noncubic,hardy2002model,tang2004morphology,jiang2005effect,xia2005microphase,sun2008morphology,liu2012theoretical}), is not the cyclic \monomer{A}\monomer{B}\monomer{C} phase but the noncyclic \monomer{A}\monomer{B}\monomer{C}\monomer{B} phase, as shown in Figure \ref{Ren vs. Mogi}. In fact, the cyclic \monomer{A}\monomer{B}\monomer{C} phase is noncentrosymmetric, while most structures formed by \monomer{A}\monomer{B}\monomer{C} triblock copolymers are centrosymmetric \cite[Page 6487]{wickham2001noncentrosymmetric}.

\begin{figure}[H]
\centering
\raisebox{28.7pt}{\includegraphics[width=200pt]{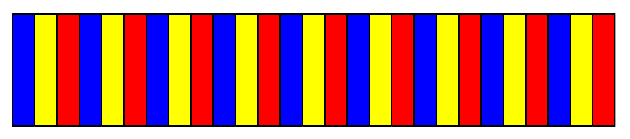}}
\hspace{10pt}
\includegraphics[width=109pt]{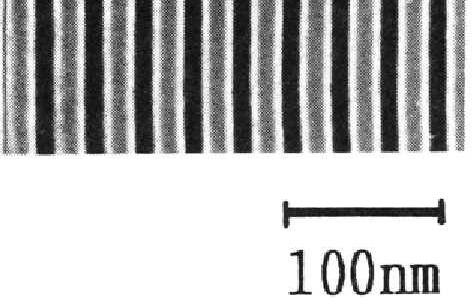}
\caption{(Left) cyclic \monomer{A}\monomer{B}\monomer{C} lamellar phase proposed by Ren and Wei. (Figure from \cite[Figure 1]{ren2014asymmetric}, reprinted with permission. Copyright \copyright{} 2014 Society for Industrial and Applied Mathematics. All rights reserved.) (Right) noncyclic \monomer{A}\monomer{B}\monomer{C}\monomer{B} phase found by Mogi et al., with black, white, and gray indicating isoprene, styrene, and 2-vinylpyridine domains, respectively. (Figure from \cite[Figure 2-a]{mogi1992preparation}, reprinted with permission. Copyright \copyright{} 1992 American Chemical Society. All rights reserved.)}
\label{Ren vs. Mogi}
\end{figure}

After careful calculations (see Sections \ref{Numerical inspiration: a counterexample} and \ref{Phase diagrams: most favored repetends}), we found that Ren\textendash Wei conjecture holds for some parameters, while for some other parameters the pattern \monomer{A}\monomer{B}\monomer{C}\monomer{B} $\cdots\;$\monomer{A}\monomer{B}\monomer{C}\monomer{B} has lower free energy than \monomer{A}\monomer{B}\monomer{C} $\cdots\;$\monomer{A}\monomer{B}\monomer{C}. Interestingly, both the original work \cite{nakazawa1993microphase} by Ohta et al. and an ensuing numerical study \cite{zheng1995morphology} by Zheng et al. chose the \monomer{A}\monomer{B}\monomer{C}\monomer{B} phase instead of the \monomer{A}\monomer{B}\monomer{C} phase as the representative of the lamellar phase when studying the ternary O\textendash K free energy, although such choices in those early works were largely empirical. Under suitable conditions, it is believed that the optimal patterns should be periodic in the vanishing $c_{ij}$ limit, but a rigorous proof is still missing.

It is noteworthy that Ren\textendash Wei conjecture was proposed for the $[\gamma_{ij}]$ given in \eqref{Ren's matrix}. However, in the literature there have been other choices of $[\gamma_{ij}]$. In the original work, Nakazawa and Ohta \cite{nakazawa1993microphase} derived the matrix \eqref{Ohta's matrix} for triblock copolymers from statistical physics. Later Ren and Wei \cite{ren2003triblock1} derived a different matrix \eqref{Ren's matrix} which was used in their 1-D study \cite{ren2003triblock2}. In their subsequent 2-D works (e.g., \cite[Equation (1.1)]{ren2015double}), Ren et al. used some general matrices \eqref{General matrix} instead. A degenerate case \eqref{Blend's matrix}, which was derived \cite{choksi2005diblock} for mixtures of \monomer{A}\monomer{B} diblock copolymers and \monomer{C} homopolymers, has also been studied in the literature (e.g., \cite{ohta1995dynamics,ito1998domain,van2009stability}), and the patterns \monomer{A}\monomer{B}\monomer{A}\monomer{B} $\cdots\;$\monomer{C} are proved to be local minimizers \cite[Section 4]{choksi2005diblock} and also global minimizers \cite[Theorems 4 and 5]{van2008copolymer}. Multiscale phase separation occurs in such a degenerate case: macrophase separation between copolymers and homopolymers, and microphase separation between \monomer{A} and \monomer{B} within the copolymers. In 1-D, this phenomenon is reflected in the patterns where \monomer{A} and \monomer{B} appear many times while \monomer{C} appears only once. The choices of $[\gamma_{ij}]$ are discussed in detail later: in Section \ref{physical analogy} we give interpretation to $[\gamma_{ij}]$ by analogy with charged immiscible fluids; in Section \ref{conditions on [gamma_{ij}]} we propose admissibility conditions on $[\gamma_{ij}]$ based on this interpretation, and show a one-to-one correspondence between the matrices satisfying the admissibility conditions and the general matrices \eqref{General matrix} used by Ren et al.

\section{Global minimizers in 1-D}
\label{Numerical inspiration: a counterexample}

We now present studies on the global minimizers of \eqref{definition of J} for $\Omega=[0,1]$ with periodic boundary conditions. In this 1-D setting, it is computationally feasible to exhaust all the possible patterns of reasonable lengths (e.g., up to 20) to find the global minimizer, until further increases in the pattern length no longer lower the free energy. Note that the pattern length equals the number of interfaces \cite[Definition 4.1]{ren2003triblock2} and we use the $[\gamma_{ij}]$ given by \eqref{Ren's matrix}. Computational details of our numerical experiments can be found in Appendix \ref{Calculation of the free energy J in 1-D}.

\begin{figure}[bpht]
\centering
\includegraphics[width=335pt]{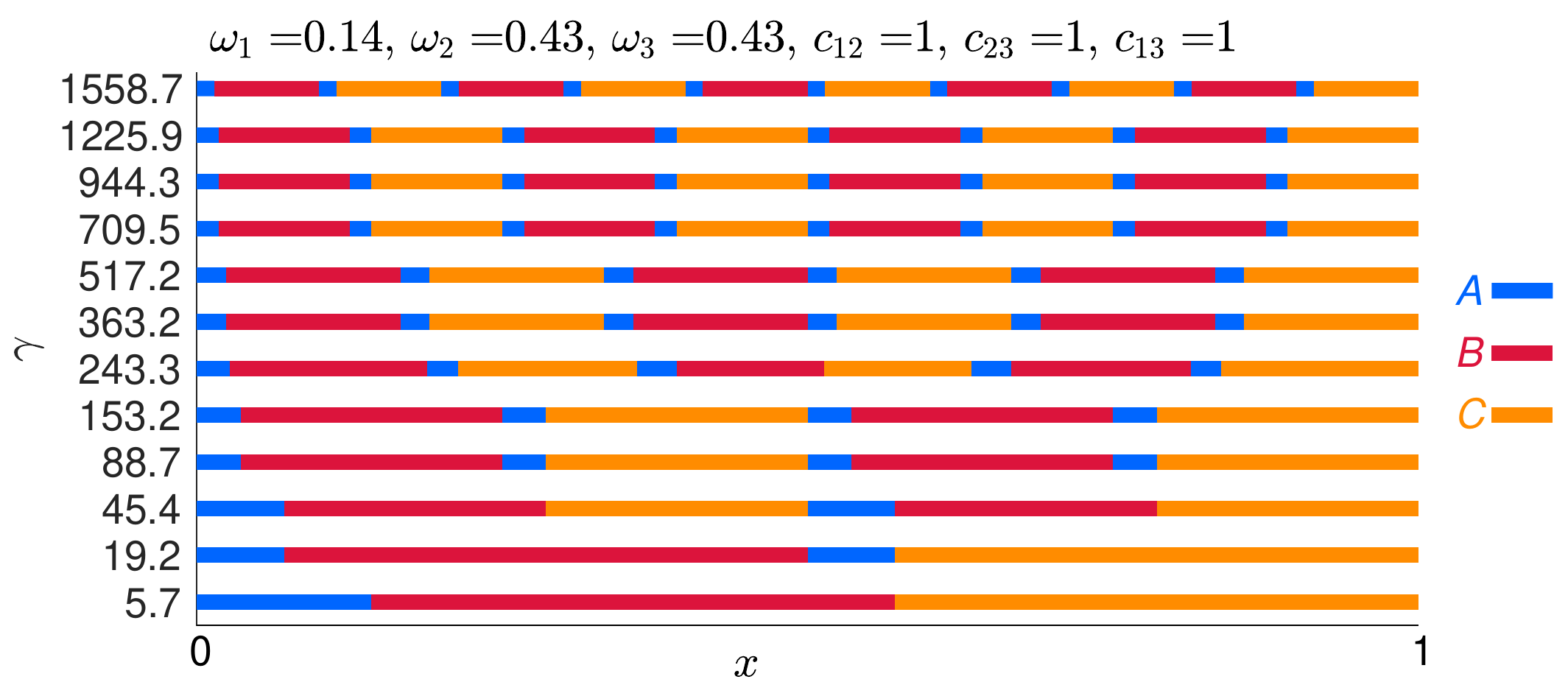}
\caption{Numerical global minimizers obtained by the exhaustive search in a 1-D periodic cell. The horizontal $x$-axis indicates the spatial position (note that 0 and 1 represent the same point under periodic boundary conditions), and the vertical $\gamma$-axis is the overall coefficient used in \eqref{Ren's matrix}.}
\label{133.pdf}
\end{figure}

A typical example of our numerical results is shown in Figure \ref{133.pdf}, where we choose $\omega_1<\omega_2=\omega_3$ and $c_{12}=c_{13}=c_{23}$, and the regions labelled by \monomer{A}, \monomer{B} and \monomer{C} belong to $\Omega_1$, $\Omega_2$ and $\Omega_3$, respectively. When $\gamma$ is very small (e.g., $\gamma=5.7$ at the bottom), the interfacial energy dominates the free energy, and our numerical results indicate that the global minimizer has only three layers. As $\gamma$ increases, more and more layers emerge, because the contribution of the interfacial energy becomes less and less dominant. For large $\gamma$ (e.g., $\gamma=1558.7$ at the top), the numerical global minimizer tends to develop patterns like \monomer{A}\monomer{B}\monomer{A}\monomer{C} $\cdots\,$\monomer{A}\monomer{B}\monomer{A}\monomer{C}, which are repetitions (or periodic extensions) of the repetend \monomer{A}\monomer{B}\monomer{A}\monomer{C}. This numerical result is a counterexample to Ren\textendash Wei conjecture \cite[Conjecture 4.10]{ren2003triblock2}, that for any $\gamma>0$, any global minimizer of \eqref{definition of J} has a cyclic pattern \monomer{A}\monomer{B}\monomer{C} $\cdots\,$\monomer{A}\monomer{B}\monomer{C} (or \monomer{A}\monomer{C}\monomer{B} $\cdots\,$\monomer{A}\monomer{C}\monomer{B}), whose repetend is \monomer{A}\monomer{B}\monomer{C} (or \monomer{A}\monomer{C}\monomer{B}). However, for some other parameters (e.g., $\omega_1=\omega_2=\omega_3$ and $c_{12}=c_{13}=c_{23}$), the numerical global minimizer is indeed of a cyclic pattern for all the $\gamma$ that we have tested. To summarize, our numerical results suggest that the pattern of the global minimizer depends on the parameters.

Motivated by the above numerical results, we calculate and compare the analytic free energies of \monomer{A}\monomer{B}\monomer{C} $\cdots\,$\monomer{A}\monomer{B}\monomer{C} and \monomer{A}\monomer{B}\monomer{A}\monomer{C} $\cdots\,$\monomer{A}\monomer{B}\monomer{A}\monomer{C}. For the cyclic pattern with \monomer{A}\monomer{B}\monomer{C} identically repeating for $n$ times, the free energy has already been derived in \cite[Equation (4.25)]{ren2003triblock2}:
\begin{equation}
\label{free energy of ABC}
J=(c_{12}+c_{13}+c_{23})\,n+\frac{\gamma}{16\hspace{0.5pt}n^2}\Big(5-\frac{9abc}{ab+bc+ca}\Big),
\end{equation}
where $a=\omega_1$, $b=\omega_2$ and $c=\omega_3$. In Appendix \ref{Calculating the free energy of periodic patterns} we outline their derivation as well as our analogous derivation for the pattern with \monomer{A}\monomer{B}\monomer{A}\monomer{C} identically repeating for $n$ times. The latter yields (for convenience we assume that all the \monomer{A} layers have the same layer width, consistent with our numerical observations under various parameters)
\begin{equation}
\label{free energy of ABAC}
J = 2(c_{12}+c_{13})\,n+\frac{\gamma}{16\hspace{0.5pt}n^2}\Big(2+\frac{3 a^2}{a b+a c+b c}\Big),
\end{equation}
which is independent of $c_{23}$ since there is no interfaces between \monomer{B} and \monomer{C}.

The following proposition rigorously demonstrates that for some parameters the repetend \monomer{A}\monomer{B}\monomer{A}\monomer{C} can indeed be energetically more favorable than the repetend \monomer{A}\monomer{B}\monomer{C}, and thus gives explanations to the numerical results in Figure \ref{133.pdf}.

\begin{proposition}
\label{ABAC is sometimes better than ABC}
There exists a set of parameters $a,b,c>0$ ($a+b+c=1$) and $\{c_{ij}\}$ (satisfying triangle inequalities) such that for $\gamma$ large enough, $\min\limits_{n\in\mathbb N}J$ is greater in \eqref{free energy of ABC} than \eqref{free energy of ABAC}.
\end{proposition}

\begin{proof}
For a fixed large $\gamma$, in \eqref{free energy of ABC} and \eqref{free energy of ABAC} we can regard $J$ as a continuous function of $n$ (for minimization purposes), since the optimal $n$ is of the order of $\gamma^{1/3}$. Therefore for \eqref{free energy of ABC} we have
\begin{equation*}
\min_{n\in\mathbb N}J=\frac34\sqrt[3]{(c_{12}+c_{13}+c_{23})^2\Big(5-\frac{9abc}{ab+bc+ca}\Big)\gamma}+O(\gamma^{-1/3}),
\end{equation*}
and for \eqref{free energy of ABAC} we have
\begin{equation*}
\min_{n\in\mathbb N}J=3\sqrt[3]{\frac{(c_{12}+c_{13})^2}{16}\Big(2+\frac{3 a^2}{a b+a c+b c}\Big)\gamma}+O(\gamma^{-1/3}),
\end{equation*}
where the coefficient of $\sqrt[\raisebox{1.5pt}{\fontsize{7pt}{0pt}\selectfont\text{$3$}}]\gamma$ is greater for \eqref{free energy of ABC} than \eqref{free energy of ABAC} under some choices of parameters. For example, in Figure \ref{133.pdf} where $a=0.14$, $b=0.43$, $c=0.43$ and $c_{12}=c_{23}=c_{13}=1$, the former is 2.52, and the latter is 2.46.
\end{proof}

The numerical results for other choices of $\{\omega_i\}$ may be slightly different. For example, if \monomer{B} instead of \monomer{A} is the minority species, i.e., $\omega_2\ll \omega_1\approx \omega_3$, then the repetend \monomer{B}\monomer{A}\monomer{B}\monomer{C} tends to be favored. Our observation is that the minority species tends to appear twice in each period, separating the other two species so that they do not come into contact. When there are two evenly matched minority species, i.e., $\omega_1\approx \omega_2\ll \omega_3$, the repetend \monomer{A}\monomer{B}\monomer{A}\monomer{C}\monomer{B}\monomer{A}\monomer{B}\monomer{C} tends to be favored. If all species are evenly matched, i.e., $\omega_1\approx \omega_2\approx \omega_3$, then the repetend \monomer{A}\monomer{B}\monomer{C} does seem to be favored. Other patterns may arise when $\{c_{ij}\}$ vary, for example, if the interfaces between \monomer{A} and \monomer{B} are barely penalized, i.e., $c_{12}\ll c_{23}\approx c_{13}$, then the global minimizer tends to have \monomer{A}\monomer{B}\monomer{A}\monomer{B} $\cdots\;$\monomer{A}\monomer{B}\monomer{C} or \monomer{A}\monomer{B}\monomer{A}\monomer{B} $\cdots\;$\monomer{A}\monomer{C} as its repetend, and in this respect, with hindsight Ren\textendash Wei conjecture cannot be true for certain parameters. In Section \ref{Phase diagrams: most favored repetends} we will include all those patterns as candidates and present relatively comprehensive comparison results.

If we switch from periodic boundary conditions to the case where there is no boundary condition (essentially by choosing $G(x,y)=-|x\!-\!y|/2$), the numerical results remain similar, except that the layers near the boundaries (0 and 1) become thinner. Therefore, for illustrative purposes it is sufficient for us to only present the case of periodic boundary conditions.

\section{Analogy for ternary O\textendash K free energy}
\label{physical analogy}

\subsection{Charge Interpretation}
\label{charge interpretation}

We now reformulate the long range term of \eqref{definition of J_epsilon} or \eqref{definition of J} in terms of generalized charge densities, and then interpret the coefficient matrix $[\gamma_{ij}]$ using the interactions between generalized charges. This interpretation is useful in subsequent sections as we impose conditions on $[\gamma_{ij}]$ and explain our numerical results. For binary systems, such a charge interpretation is natural by considering the classical positive and negative charges with Coulomb's law, see, e.g., \cite{choksi2017old}. For ternary systems, we have to work with three types of generalized charges, any two of which interact via the kernel $G$, like classical charges obeying Coulomb's law.

Suppose there are three types of charges distributed in $\Omega$, with $\rho_1$, $\rho_2$ and $\rho_3$ denoting their density fields, respectively. Then, similar to the classical Coulomb's law, we define the total potential energy between charges as
\begin{equation}
\label{total electrostatic potential energy}
U(\vec\rho)=\frac12\int_{\Omega}\int_{\Omega}\vec\rho(\vec x)^{\rm T}\,[f_{ij}]\,\vec\rho(\vec y)\,G(\vec x,\vec y)\dd{\vec x}\dd{\vec y},
\end{equation}
where the vectorized density $\vec\rho=[\rho_1,\rho_2,\rho_3]^{\rm T}$ is element-wise nonnegative and integrable over $\Omega$, and $[f_{ij}]\in\mathbb R^{3\times3}$ is the interaction strength matrix between charges (the potential energy due to the interaction between a unit amount of type $i$ charge at point $\vec x$ and a unit amount of type $j$ charge at point $\vec y$ is assumed to be $f_{ij}\,G(\vec x,\vec y)$, so \eqref{total electrostatic potential energy} is the pairwise "sum" of the potential energy between any two point charges).

\begin{definition}
\label{charge neutrality conditions}
 We say that $\vec\rho$ satisfies the overall charge neutrality condition if 
 $\int_\Omega\vec\rho(\vec x)\dd{\vec x}
 $
 is parallel to the vector $\vec 1=[1,1,1]^{\rm T}$. Moreover,
 $\vec\rho$ satisfies the pointwise charge neutrality condition if $\vec\rho(\vec x)$
  is parallel to the vector $\vec 1$
 for $\vec x\in\Omega$ almost everywhere.
\end{definition}

\begin{theorem}
\label{equivalence between long range term and electrostatic potential energy}
Under periodic or Neumann boundary conditions, the long range term in \eqref{definition of J_epsilon} or \eqref{definition of J} is twice the total potential energy $U(\vec\rho)$ defined in \eqref{total electrostatic potential energy}, with
\begin{equation*}
[f_{ij}]=
\begin{bmatrix}
\omega_1 &0 &0\\
0 &\omega_2 &0\\
0 &0 &\omega_3
\end{bmatrix}
[\gamma_{ij}]
\begin{bmatrix}
\omega_1 &0 &0\\
0 &\omega_2 &0\\
0 &0 &\omega_3
\end{bmatrix},
\end{equation*}
and
\begin{equation*}
\vec\rho=
\begin{bmatrix}
\omega_1 &0 &0\\
0 &\omega_2 &0\\
0 &0 &\omega_3
\end{bmatrix}
^{-1}
\vec u\quad\text{or}\quad\vec\rho=
\begin{bmatrix}
\omega_1 &0 &0\\
0 &\omega_2 &0\\
0 &0 &\omega_3
\end{bmatrix}
^{-1}
\begin{bmatrix}
\bm1_{\Omega_1}\\
\bm1_{\Omega_2}\\
\bm1_{\Omega_3}
\end{bmatrix}.
\end{equation*}
Such $\vec\rho$ satisfies the overall charge neutrality condition.
\end{theorem}

\begin{proof}
Under periodic or Neumann boundary conditions, we have $\int_{\Omega}G(\vec x,\vec y)\dd{\vec x}=0$ for any $\vec y\in\Omega$, therefore the long range term of \eqref{definition of J_epsilon} equals
\begin{equation}
\label{long range term}
\int_{\Omega}\int_{\Omega}\vec u(\vec x)^{\rm T}\,[\gamma_{ij}]\,\vec u(\vec y)\,G(\vec x,\vec y)\dd{\vec x}\dd{\vec y},
\end{equation}
and that of \eqref{definition of J} equals
\begin{equation}
\label{general framework}
\sum_{i=1}^3\sum_{j=1}^3\gamma_{ij}\int_{\Omega}\int_{\Omega}\bm1_{\Omega_i}(\vec x)\,G(\vec x,\vec y)\,\bm1_{\Omega_j}(\vec y)\dd{\vec x}\dd{\vec y}.
\end{equation}
Therefore the first statement is clearly correct. As for the second statement, we have
\begin{equation*}
\frac1{|\Omega|}\int_\Omega\vec\rho(\vec x)\dd{\vec x}=\vec 1,
\end{equation*}
so $\vec\rho$ satisfies the overall charge neutrality condition.
\end{proof}

\begin{remark}
Recall that the classical electrostatic potential energy can be expressed as $\epsilon_0\int_{\mathbb R^3}\big|\vec E\big|^2/2$, where $\epsilon_0$ is the vacuum permittivity, and $\vec E$ is the electric field. Analogously \eqref{total electrostatic potential energy}, the potential energy between our generalized charges, can be rewritten as $\int_{\Omega}\,[\vec E_i]\cdot[f_{ij}]\cdot[\vec E_i]/2$. To be more precise, since $G$ is the Green's function of $-\Delta$, we have
\begin{equation}
\label{L^2 norm of electrostatic field}
U(\vec\rho)=\frac12\sum_{i=1}^3\sum_{j=1}^3f_{ij}\int_{\Omega}\vec E_i(\vec x)\cdot\vec E_j(\vec x)\dd{\vec x},
\end{equation}
where
\begin{equation*}
\vec E_i(\vec x)=\nabla\raisebox{-3pt}{\fontsize{11pt}{0pt}\selectfont\text{$\vec x$}}\int_{\Omega}G(\vec x,\vec y)\,\rho_i(\vec y)\dd{\vec y}.
\end{equation*}
Note that the form \eqref{L^2 norm of electrostatic field} is equivalent to the commonly used expression of the long range term in terms of $H^{-1}$ norm or $(-\Delta)^{-1/2}$ operator, see e.g., \cite[Definition 1]{van2008copolymer}, \cite[Page 1338]{choksi2010small} and \cite[Page 168]{choksi2003derivation}.
\end{remark}

Theorem \ref{equivalence between long range term and electrostatic potential energy} shows the connection between the long range term of the ternary O\textendash K free energy and the potential energy between three types of generalized charges. The choice of $\vec\rho$ in Theorem \ref{equivalence between long range term and electrostatic potential energy} can be understood as follows: each monomer of type $i$ hypothetically carries a fixed amount of type $i$ charge, so type $i$ charge is uniformly distributed in domain $\Omega_i$, and we fix the total amount of type $i$ charge to be the total volume $|\Omega|$. Therefore the local charge density (i.e., the amount of charge per unit volume near a single point) of type $i$ charge is $\rho_i = u_i/\omega_i$ or $\rho_i = \bm1_{\Omega_i}/\omega_i$. In this way we can interpret the phase fields as charge density fields. Although it seems like a simple rescaling, the incompressibility condition is usually imposed on phase fields, but not on charge density fields. In Section \ref{physical intuition of conditions}, the interpretation by charge densities (without the incompressibility condition) enables us to propose admissibility conditions on $[f_{ij}]$ and thus $[\gamma_{ij}]$.

\subsection{Analogy with uniformly charged immiscible fluids}

With the above charge interpretation, we can make an analogy between the ternary O\textendash K free energy and the system of uniformly charged immiscible fluids:
\begin{itemize}

\item Like the classical Van der Waals\textendash Cahn\textendash Hilliard free energy \cite{cahn1958free}, we have different types of immiscible fluids confined to $\Omega$ subject to the incompressibility condition. The immiscibility causes them to separate, which is reflected in the interfacial energy (i.e., short range term) of \eqref{definition of J_epsilon} or \eqref{definition of J}.

\item Like the liquid drop model \cite{gamow1930mass}, we also assume that different types of fluids uniformly carry different types of charges respectively, so that there is always a tendency for them to mix. This is captured by the long range term of \eqref{definition of J_epsilon} or \eqref{definition of J}.

\end{itemize}
The above analogy is what we call "uniformly charged immiscible fluids", but its applications are certainly not limited to hypothetically charged fluids. The same free energy can arise in completely different physical contexts. In fact, the binary version \eqref{definition of I_epsilon} has been used to describe a wide diversity of systems ranging from polymer systems to ferroelectric/ferromagnetic systems to quantum systems to reaction-diffusion systems \cite{muratov1998theory,muratov2002theory}, and its sharp interface limit \eqref{definition of I} has also been used in the astrophysical context to model the structure of nuclear matter in the crust of neutron stars \cite{knupfer2016low}. The underlying physics are very different: the short range term may arise out of actual interfacial tension, or ejection of holes from the antiferromagnet, or the nuclear force; the long range term may come from actual electrostatic repulsion, or diffusion of chemically reacting species, or entropy of the system. But all those systems have one thing in common: intricate patterns arise from the competition between the short range tendency towards phase separation and the long range tendency to mix.

In the literature (e.g., \cite{muratov2002theory,alberti2009uniform,muratov2010droplet}), the short range term is commonly regarded as "attractive", and the long range term is "repulsive", mainly because the former causes phases of the same type to congregate and form large domains, while the latter suppresses or inhibits large domains. However according to our analogy, it also seems plausible that the short range term is "repulsive" and the long range term is "attractive", because the former causes immiscible fluids of different types to separate into different regions, while the latter tries to mix them back to form a neutrally charged mixture.

Note that our analogy can be generalized to multiphase systems of any number of phases, and the discussion in this section can be translated to those more general cases. We only present the ternary case because it is the simplest case (apart from the classical binary case), and most existing mathematical works only studied ternary systems, except that \cite[Chapter 4]{wang2018analysis} studied a planar triple bubble in a quaternary system.

In the block copolymer setting, the pointwise charge neutrality condition in Definition \ref{charge neutrality conditions} corresponds to the completely mixed state or disordered phase, i.e., $u_i(\vec x)=\omega_i$ for any $\vec x\in\Omega$ in \eqref{definition of J_epsilon}. In the binary case, such a uniform distribution has been proved in \cite{alberti2009uniform,nunzio2009uniform} to be preferred by the minimizers on a large length scale asymptotically. We expect analogous results to hold for the ternary case with certain choices of $[\gamma_{ij}]$ (see Conjecture \ref{uniform distribution conjecture}). In Section \ref{physical intuition of conditions}, we discuss in detail what conditions should be imposed on $[\gamma_{ij}]$ in order to ensure the driving force towards pointwise charge neutrality. In this regard, we can interpret the fine structures formed by block copolymers as the outcome of the competition between the interfacial tension and the principle of charge neutrality.

\section{Coefficient matrix of the long range term}
\label{conditions on [gamma_{ij}]}

The long range term of \eqref{definition of J_epsilon} or \eqref{definition of J} is given in various forms in the literature. Ren et al. formulated it using a $3\times3$ matrix $[\gamma_{ij}]$ in their earlier works \cite{ren2003triblock1,ren2003triblock2}, and then reduced it to a $2\times2$ matrix $[\tilde\gamma_{ij}]$ in their later works (e.g., \cite{ren2013double}) using the incompressibility condition $u_1+u_2+u_3=1$ or $\bm1_{\Omega_1}+\bm1_{\Omega_2}+\bm1_{\Omega_3}=1$. Although sometimes it is indeed more convenient to use the $2\times2$ matrix, in this work we stick to the $3\times3$ matrix, of which each entry has a clear meaning or intuitive interpretation, as we have seen from the charge analogy drawn in Section \ref{charge interpretation}. (This $3\times3$ version also inspired our proof of Proposition \ref{arrangement in the all nonpositive case}.) In Section \ref{existing choices of [gamma_ij]} we list the existing choices of $[\gamma_{ij}]$. In Section \ref{physical intuition of conditions} we impose some admissibility conditions on $[\gamma_{ij}]$, and prove that there exists a one-to-one correspondence between $[\tilde\gamma_{ij}]$ and those $[\gamma_{ij}]$ satisfying the admissibility conditions.

\subsection{Existing choices}
\label{existing choices of [gamma_ij]}

In the literature there exist several different choices of $[\gamma_{ij}]$ for \eqref{definition of J_epsilon} or \eqref{definition of J}. For each choice of $[\gamma_{ij}]$, there is a corresponding interaction strength matrix $[f_{ij}]$ (defined via the relation in Theorem \ref{equivalence between long range term and electrostatic potential energy}). In the following we let $a=\omega_1$, $b=\omega_2$ and $c=\omega_3$ for convenience.

\begin{itemize}
\item
In the original work \cite[Equations (2.23) and (A.7)]{nakazawa1993microphase}, Ohta et al. derived from mean field theory the free energy of \monomer{A}\monomer{B}\monomer{C} triblock copolymers with the following matrix
\begin{equation}
\label{Ohta's matrix}
[\gamma_{ij}]=\frac{3\gamma}{3-2(a\!+\!c)-(a\!-\!c)^2}
\begin{bmatrix}
 \frac{2b+2c}{a^2} & -\frac{2c+3b}{a b} & \frac{b}{a c} \\
 -\frac{2c+3b}{a b} & \frac{2+4b}{b^2} & -\frac{2a+3b}{b c} \\
 \frac{b}{a c} & -\frac{2a+3b}{b c} & \frac{2a+2b}{c^2} \\
\end{bmatrix},
\end{equation}
where $\gamma$ is a positive parameter related to the degree of polymerization (the total chain length). The corresponding interaction strength matrix is given by
\begin{equation}
\label{Ohta's matrix f_ij}
[f_{ij}]\sim
\begin{bmatrix}
 {2b\!+\!2c} & {-2c\!-\!3b} & {b} \\
 {-2c\!-\!3b} & {2\!+\!4b} & {-2a\!-\!3b} \\
 {b} & {-2a\!-\!3b} & {2a\!+\!2b} \\
\end{bmatrix},
\end{equation}
where the symbol $\sim$ denotes direct proportionality with a positive coefficient.

\item
In \cite[Equation (4.21)]{ren2003triblock1}, Ren and Wei re-derived the ternary O\textendash K free energy. The matrix that they obtained is of a symmetric form and given by
\begin{equation}
\label{Ren's matrix}
[\gamma_{ij}]=\frac{3\gamma}{4 (a b+a c+b c)}
\begin{bmatrix}
 \frac{b+c}{a^2} & -\frac{c}{a b} & -\frac{b}{a c} \\
 -\frac{c}{a b} & \frac{a+c}{b^2} & -\frac{a}{b c} \\
 -\frac{b}{a c} & -\frac{a}{b c} & \frac{a+b}{c^2} \\
\end{bmatrix},
\end{equation}
where $\gamma$ is a positive parameter controlling the length scale of microdomains. By symmetry we mean that the free energy is invariant to the permutation of the monomer types, i.e., it is the same for all three kinds of triblock copolymers (\monomer{A}\monomer{B}\monomer{C}, \monomer{A}\monomer{C}\monomer{B} and \monomer{B}\monomer{A}\monomer{C}), independent of the block sequence. (By contrast, in \eqref{Ohta's matrix} the parameter $b$ plays a different role from $a$ and $c$.) Accordingly, the $[f_{ij}]$ corresponding to \eqref{Ren's matrix} is also of a symmetric form:
\begin{equation}
\label{Ren's matrix f_ij}
[f_{ij}]\sim
\begin{bmatrix}
 {b\!+\!c} & -{c} & -{b} \\
 -{c} & {a\!+\!c} & -{a} \\
 -{b} & -{a} & {a\!+\!b} \\
\end{bmatrix}.
\end{equation}

\item
For the mixture of \monomer{A}\monomer{B} diblock copolymers and \monomer{C} homopolymers, Ohta et al. \cite[Equation (2.6)]{ohta1995dynamics} and Ren et al. \cite[Equation (3.27)]{choksi2005diblock} both proposed the following matrix,
\begin{equation}
\label{Blend's matrix}
[\gamma_{ij}]=\frac34\frac{\gamma}{a+b}
\begin{bmatrix}
 \frac{1}{a^2} & -\frac{1}{a b} & 0 \\
 -\frac{1}{a b} & \frac{1}{b^2} & 0 \\
 0 & 0 & 0 \\
\end{bmatrix},
\end{equation}
where $\gamma$ is a positive parameter. The corresponding $[f_{ij}]$ is an extension (by zero) of the classical Coulomb's law:
\begin{equation}
\label{Blend's matrix f_ij}
[f_{ij}]\sim
\begin{bmatrix}
 1 & -1 & 0 \\
 -1 & 1 & 0 \\
 0 & 0 & 0 \\
\end{bmatrix}.
\end{equation}

\item
In their later works on 2-D cases (e.g., \cite[Equation (1.1)]{ren2015double}), instead of using the derived matrix \eqref{Ren's matrix}, Ren et al. chose among the following general matrices:
\begin{equation}
\label{General matrix}
[\gamma_{ij}]=
\begin{bmatrix}
 1-a & -b \\
 -a & 1-b \\
 -a & -b
\end{bmatrix}
[\tilde\gamma_{ij}]
\begin{bmatrix}
 1-a & -a & -a \\
 -b & 1-b & -b
\end{bmatrix},
\end{equation}
where $[\tilde\gamma_{ij}]\in\mathbb R^{2\times2}$ is either positive definite (for triblock copolymers) or positive semi-definite with 0-eigenvector being $[a,b]^{\rm T}$ (for the mixture of diblock copolymers and homopolymers). In Section \ref{physical intuition of conditions}, we are able to show that the class \eqref{General matrix} contains \eqref{Ohta's matrix}, \eqref{Ren's matrix} and \eqref{Blend's matrix} as special cases, and actually consists of all the matrices that are admissible. This establishes the connection between Ren et al.'s choices \eqref{General matrix} and our framework of admissible matrices. Some other ranges of $[\tilde\gamma_{ij}]$ have also been considered in the literature. For example, \cite{ren2019stationary} requires $[\tilde\gamma_{ij}]$ to be a positive (but not necessarily positive definite) matrix. Note that \eqref{General matrix} is not in a symmetric form with respect to the permutation of $a$, $b$ and $c$, because it was obtained by eliminating $u_3$ using $u_3=1-u_1-u_2$. Corresponding to \eqref{General matrix}, we have
\begin{equation*}
[f_{ij}]\sim
\begin{bmatrix}
 1\!-\!a & -a \\
 -b & 1\!-\!b \\
 -c & -c
\end{bmatrix}
\begin{bmatrix}
 a & 0 \\
 0 & b
\end{bmatrix}
[\tilde\gamma_{ij}]
\begin{bmatrix}
 a & 0 \\
 0 & b
\end{bmatrix}
\begin{bmatrix}
 1\!-\!a & -b & -c \\
 -a & 1\!-\!b & -c
\end{bmatrix}.
\end{equation*}

\end{itemize}


\subsection{Admissibility conditions}
\label{physical intuition of conditions}

\begin{definition}
\label{facilitation of charge neutrality}
In terms of facilitating charge neutrality, $[f_{ij}]$ is said to be admissible if it satisfies the following three conditions
\begin{enumerate}[label=\protect\CircleAroundChar{\arabic*}]
\item
$\vec 1^{\rm T}[f_{ij}]\,\vec 1=0$,
\item
$\vec q^{\rm T}\,[f_{ij}]\,\vec q\geqslant0$, for any $\vec q$,
\item $[f_{ij}]$ is symmetric.
\end{enumerate}
\end{definition}

\begin{theorem}
\label{equivalent conditions for facilitation of charge neutrality}
The three conditions in Definition \ref{facilitation of charge neutrality} for $[f_{ij}]$ (or $[\gamma_{ij}]$, via the relation in Theorem \ref{equivalence between long range term and electrostatic potential energy}) are equivalent to the following two conditions:
\begin{align}
[f_{ij}]\,\vec 1=\vec 0\quad&\big(\ \text{or}\ \ 
[\gamma_{ij}]\,\vec\omega=\vec 0\ \big),\label{0-eigenvector condition}\\
[f_{ij}]\succcurlyeq0\quad&\big(\ \text{or}\ \ [\gamma_{ij}]\succcurlyeq0\ \big),\label{positive semi-definite condition}
\end{align}
where $\vec\omega=[\omega_1,\omega_2,\omega_3]^{\rm T}$. 
\end{theorem}

The proof is straightforward.

\begin{remark}
We provide some justification for the three conditions in Definition \ref{facilitation of charge neutrality}.
\begin{enumerate}[label=(\roman*)]
\item Condition \CircleAroundChar{1} means that within neutrally charged objects, there should be no net interaction, that is, any charge density field satisfying the pointwise charge neutrality condition is in equilibrium and thus has the same potential energy. See Lemma \ref{mathematical details of the conditions}-(i) for a mathematical description.

\item Condition \CircleAroundChar{2} is to ensure that neutrally charged objects are energetically most favorable, that is, any charge density field satisfying the pointwise charge neutrality condition has the lowest possible potential energy. See Lemma \ref{mathematical details of the conditions}-(ii) for mathematical details.

\item Condition \CircleAroundChar{3} means that the interaction strength matrix should be symmetric, in line with Newton's third law.

\end{enumerate}
\end{remark}

\begin{lemma}
\label{mathematical details of the conditions}
For the following statements we assume $\int_\Omega\vec\rho(\vec x)\dd{\vec x}=\vec 1$ in \eqref{total electrostatic potential energy}.
\begin{enumerate}[label=(\roman*)]

\item If $U(\vec\rho)$ is a constant for any $\vec\rho$ satisfying the pointwise charge neutrality condition, then $[f_{ij}]$ must satisfy Condition \CircleAroundChar{1} in Definition \ref{facilitation of charge neutrality}.

\item Further, if the above constant is not greater than $U(\vec\rho)$ for any $\vec\rho$, then $[f_{ij}]$ must satisfy Condition \CircleAroundChar{2}, provided that it satisfies Condition \CircleAroundChar{3}.

\end{enumerate}
\end{lemma}

\begin{proof}
\begin{enumerate}[label=(\roman*)]

\item Let $\vec\rho(\vec x)=\vec 1\,\psi(\vec x)$ with $\psi\geqslant0$ and $\int_\Omega\psi(\vec x)\dd{\vec x}=1$, then we have
\begin{equation*}
U(\vec\rho)=\frac{\vec 1^{\rm T}[f_{ij}]\,\vec 1}2
\int_{\Omega}\int_{\Omega}\psi(\vec x)\,G(\vec x,\vec y)\,\psi(\vec y)\dd{\vec x}\dd{\vec y}.
\end{equation*}
To ensure that $U(\vec\rho)$ is constant, Condition \CircleAroundChar{1} in Definition \ref{facilitation of charge neutrality} must be satisfied. In this way $U(\vec\rho)=0$.

\item Consider the superposition of two charge density fields deviating from charge neutrality:
\begin{equation*}
\vec\rho(\vec x)=(\vec 1+\vec q)\,\eta(\vec x)+(\vec 1-\vec q)\,\phi(\vec x),
\end{equation*}
where $\eta,\,\phi\geqslant0$ with $\int_\Omega\eta(\vec x)\dd{\vec x}=\int_\Omega\phi(\vec x)\dd{\vec x}=1/2$ and $\|\vec q\|_{\infty}\leqslant1$. For such $\vec\rho$ we have
\begin{equation*}
\begin{aligned}
U(\vec\rho)&=
\begin{aligned}[t]
&\frac{(\vec 1+\vec q)^{\rm T}[f_{ij}]\,(\vec 1+\vec q)}2
\int_{\Omega}\int_{\Omega}\eta(\vec x)\,G(\vec x,\vec y)\,\eta(\vec y)\dd{\vec x}\dd{\vec y}+\\
&(\vec 1+\vec q)^{\rm T}[f_{ij}]\,(\vec 1-\vec q)
\int_{\Omega}\int_{\Omega}\eta(\vec x)\,G(\vec x,\vec y)\,\phi(\vec y)\dd{\vec x}\dd{\vec y}+\\
&\frac{(\vec 1-\vec q)^{\rm T}[f_{ij}]\,(\vec 1-\vec q)}2
\int_{\Omega}\int_{\Omega}\phi(\vec x)\,G(\vec x,\vec y)\,\phi(\vec y)\dd{\vec x}\dd{\vec y}
\end{aligned}\\
&=
\begin{aligned}[t]
&\frac{\vec q^{\rm T}[f_{ij}]\,\vec q}2\int_{\Omega}\int_{\Omega}(\eta\!-\!\phi)(\vec x)\,G(\vec x,\vec y)\,(\eta\!-\!\phi)(\vec y)\dd{\vec x}\dd{\vec y}+\\
&\vec q^{\rm T}[f_{ij}]\,\vec 1\int_{\Omega}\int_{\Omega}\!\big(\eta(\vec x)\eta(\vec y)\!-\!\phi(\vec x)\phi(\vec y)\big)G(\vec x,\vec y)\dd{\vec x}\dd{\vec y}.
\end{aligned}
\end{aligned}
\end{equation*}
Since $G$ is positive semi-definite, we can find $\eta$ and $\phi$ such that the double integral in the second to last summand is positive. If Condition \CircleAroundChar{2} in Definition \ref{facilitation of charge neutrality} is not satisfied, then we can choose $\vec q$ such that $\vec q^{\rm T}[f_{ij}]\,\vec q<0$ with the last summand being nonpositive (otherwise replace $\vec q$ by $-\vec q$), so $U(\vec\rho)$ would be negative which is undesired.
\end{enumerate}
\end{proof}

\begin{remark}$ $
\label{remark on the incompressibility condition}
\begin{enumerate}[label=(\roman*)]

\item Note that our derivation in Lemma \ref{mathematical details of the conditions} does not incorporate the incompressibility condition $\vec u^{\rm T}\vec 1=1$ (or equivalently $\vec\rho^{\rm T}\vec\omega=1$). Although we can derive the same results under the incompressibility condition, the derivation is less intuitive, as shown in Appendix \ref{Alternative derivation}. Note that the proof of Lemma \ref{mathematical details of the conditions} only requires $G$ to be positive semi-definite, and therefore can be generalized to other kernels.

\item Our discussion here may help us understand the numerical results presented in \cite[Figure 4.4 (e)]{wang2019bubble} and \cite[Figures 5 and 8]{ren2019stationary} for indefinite $[\gamma_{ij}]$, where droplets formed by fluids of two different minority types are fully separated into macroscopic domains. As we can see from the proof of Lemma \ref{mathematical details of the conditions}-(ii), when $[f_{ij}]$ is indefinite we can take a nonzero $\vec q$ such that $\vec q^{\rm T}[f_{ij}]\,\vec q<0$ (with $\vec q^{\rm T}\vec\omega=0$ due to incompressibility), and take $\eta$ and $\phi$ to be the indicator functions of the macroscopic domains so that $\int_{\Omega}\int_{\Omega}(\eta\!-\!\phi)(\vec x)\,G(\vec x,\vec y)\,(\eta\!-\!\phi)(\vec y)\dd{\vec x}\dd{\vec y}$ is large and thus the free energy is low. Such macroscopic segregation may be of some interest per se, but is undesirable at least in the triblock copolymer context, where \monomer{A}, \monomer{B} and \monomer{C} subchains are connected by covalent bonds. Therefore we require $[f_{ij}]\succcurlyeq0$ to ensure that charge neutrality is preferred by the long range term.
\end{enumerate}
\end{remark}

\begin{remark}$ $
\label{one-to-one correspondence under the conditions}
\begin{enumerate}[label=(\roman*)]
\item Under the incompressibility condition $\vec u^{\rm T}\vec 1=1$ or $\sum_i\bm1_{\Omega_i}=1$, different choices of $[\gamma_{ij}]$ may yield the same long range term \eqref{long range term} or \eqref{general framework}. However, according to Proposition \ref{one choice satisfies the conditions}, among those equivalent choices, only one satisfies \eqref{0-eigenvector condition}. Therefore the choices of $[\gamma_{ij}]$ are unique under the conditions in Theorem \ref{equivalent conditions for facilitation of charge neutrality}.

\item The choice \eqref{General matrix} comprises all the admissible matrices via a one-to-one correspondence between $[\gamma_{ij}]$ and $[\tilde\gamma_{ij}]$. See Proposition \ref{one choice satisfies the conditions} and its proof for details.
\end{enumerate}
\end{remark}

In the binary case, the minimizer of \eqref{definition of I_epsilon} or \eqref{definition of I} asymptotically has uniform energy and density distribution on the macroscopic scale \cite{alberti2009uniform,nunzio2009uniform}, with the characteristic domain size being $\gamma^{-1/3}$. This is because of the scaling properties of the short and long range terms \cite[Equation (14)]{muratov2002theory}, with the latter prevailing over the former on the macroscopic level and favoring charge neutrality. We believe that one can prove analogous results for non-degenerate ternary (also quaternary, quinary, etc.) systems. To be more precise, we present the following conjecture.

\begin{conjecture}
\label{uniform distribution conjecture}
 If the matrix $[\gamma_{ij}]$ satisfies the conditions in Theorem \ref{equivalent conditions for facilitation of charge neutrality} and is of nullity 1, with $\gamma>0$ being its overall factor (i.e., we fix $[\gamma_{ij}]/\gamma$), then there exist $C,\gamma^*>0$ such that for any $\gamma\geqslant\gamma^*$ and $l\geqslant\sqrt[3]{\gamma^*}$, we have
\begin{equation*}
\bigg\|\fint_{Q\big(\vec x,\,l/\!\sqrt[\raisebox{1.5pt}{\fontsize{7pt}{0pt}\selectfont\text{$3$}}]\gamma\big)}\vec u(\vec y;\gamma)\dd{\vec y}-\vec\omega\bigg\|\leqslant\frac Cl,\;\;\text{for any}\;\;Q\big(\vec x,\,l/\!\sqrt[\raisebox{1.5pt}{\fontsize{7pt}{0pt}\selectfont\text{$3$}}]\gamma\big)\subseteq\Omega,
\end{equation*}
where $\vec u(\;\cdot\;\,;\gamma)$ is the minimizer of \eqref{definition of J_epsilon}, and $Q(\vec x,l)$ denotes the cube centered at $\vec x$ with edge length $l$. Analogous conclusions hold for the sharp interface limit \eqref{definition of J}.
\end{conjecture}

\subsection{Decomposition of the interaction strength matrix}
\label{A possible underlying mechanism}
\begin{proposition}
\label{existing choices are positive semi-definite}
The matrices in \eqref{Ohta's matrix}, \eqref{Ren's matrix} and \eqref{Blend's matrix} are admissible, and are special cases of \eqref{General matrix}.
\end{proposition}
\begin{proof}
We can decompose the right-hand side of \eqref{Ren's matrix f_ij} into the sum of three simple matrices
\begin{equation*}
\begin{bmatrix}
  {c} & -{c} & 0 \\
 -{c} & {c} & 0 \\
    0 & 0 & 0 \\
\end{bmatrix}
+
\begin{bmatrix}
 {b} & 0 & -{b} \\
 0 & 0 & 0 \\
 -{b} & 0 & {b} \\
\end{bmatrix}
+
\begin{bmatrix}
 {0} & 0 & {0} \\
 0 & {a} & -{a} \\
 {0} & -{a} & {a} \\
\end{bmatrix},
\end{equation*}
from which it is clear that $[f_{ij}]$ is positive semi-definite with 0-eigenvector being $[1,1,1]^{\rm T}$. Analogously the right-hand side of \eqref{Ohta's matrix f_ij} equals
\begin{equation*}
(2c\!+\!3b)
\begin{bmatrix}
 {1} & {-1} & 0 \\
 {-1} & {1} & 0 \\
 0 & 0 & 0 \\
\end{bmatrix}
+
\begin{bmatrix}
 {-b} & 0 & {b} \\
 0 & 0 & 0 \\
 {b} & 0 & {-b} \\
\end{bmatrix}
+(2a\!+\!3b)
\begin{bmatrix}
 0 & 0 & 0 \\
 0 & {1} & {-1} \\
 0 & {-1} & {1} \\
\end{bmatrix},
\end{equation*}
indicating that
\begin{equation*}
\begin{aligned}
&\;[\zeta, \theta, \mu]\,[f_{ij}]\,[\zeta, \theta, \mu]^{\rm T}\\
\sim&\;(2c\!+\!3b)(\zeta\!-\!\theta)^2-b(\zeta\!-\!\mu)^2+(2a\!+\!3b)(\theta\!-\!\mu)^2\\
=&\;(2c\!+\!3b)(\zeta\!-\!\theta)^2-b(\zeta\!-\!\theta\!+\!\theta\!-\!\mu)^2+(2a\!+\!3b)(\theta\!-\!\mu)^2\\
\geqslant&\;(2c\!+\!3b)(\zeta\!-\!\theta)^2-2b(\zeta\!-\!\theta)^2-2b(\theta\!-\!\mu)^2+(2a\!+\!3b)(\theta\!-\!\mu)^2\\
\geqslant&\;0,
\end{aligned}
\end{equation*}
where the inequalities become equalities when $\zeta=\theta=\mu$. Therefore \eqref{Ohta's matrix f_ij} is also positive semi-definite with 0-eigenvector being $[1,1,1]^{\rm T}$. By Theorem \ref{equivalent conditions for facilitation of charge neutrality}, we know that \eqref{Ohta's matrix}, \eqref{Ren's matrix} and \eqref{Blend's matrix} are all admissible. According to Remark \ref{one-to-one correspondence under the conditions}-(ii), they are all special cases of \eqref{General matrix}.
\end{proof}

\begin{remark}
\label{microscopic physical picture}

As we can see from Proposition \ref{sufficient and necessary condition for [f_{ij}] to be positive semi-definite}, admissible $[f_{ij}]$ has 3 degrees of freedom. However, if we are only concerned with the relative interaction strengths, then its degrees of freedom can be reduced to 2 (e.g., by imposing $f_{12}^2\!+\!f_{13}^2\!+\!f_{23}^2=1$ on nonzero $[f_{ij}]$), corresponding to a spherical cap shown in Figure \ref{f_ijChoices.png}.
\end{remark}

\begin{proposition}
\label{sufficient and necessary condition for [f_{ij}] to be positive semi-definite}
Any $[f_{ij}]$ satisfying \eqref{0-eigenvector condition} can be written as
\begin{equation}
\label{decomposition of [f_ij]}
f_{12}
\begin{bmatrix}
 -1 & 1 & 0 \\
 1 & -1 & 0 \\
    0 & 0 & 0 \\
\end{bmatrix}
+f_{13}
\begin{bmatrix}
 -1 & 0 & 1 \\
 0 & 0 & 0 \\
 1 & 0 & -1 \\
\end{bmatrix}
+f_{23}
\begin{bmatrix}
 {0} & 0 & {0} \\
 0 & -1 & 1 \\
 {0} & 1 & -1 \\
\end{bmatrix},
\end{equation}
which is positive semi-definite if and only if $f_{12}\!+\!f_{13}\!+\!f_{23}\leqslant-\sqrt{f_{12}^2\!+\!f_{13}^2\!+\!f_{23}^2}$.
\end{proposition}

\begin{proof}
For $[f_{ij}]$ given by \eqref{decomposition of [f_ij]}, we have
\begin{equation*}
\begin{aligned}
\null[\zeta, \theta, \mu]\,[f_{ij}]\,[\zeta, \theta, \mu]^{\rm T}=&-f_{12}(\zeta\!-\!\theta)^2-f_{13}(\zeta\!-\!\mu)^2-f_{23}(\theta\!-\!\mu)^2\\
=&-f_{12}\alpha^2-f_{13}\beta^2-f_{23}(\alpha\!+\!\beta)^2\\
=&-
\begin{bmatrix}
\alpha & \beta
\end{bmatrix}
\begin{bmatrix}
f_{12}+f_{23} & f_{23}\\
f_{23} & f_{13}+f_{23}
\end{bmatrix}
\begin{bmatrix}
\alpha\\
\beta
\end{bmatrix}
,
\end{aligned}
\end{equation*}
where $\alpha=\theta\!-\!\zeta$ and $\beta=\zeta\!-\!\mu$. For $[f_{ij}]$ to be positive semi-definite, we require the $2\times2$ matrix in the last line to be negative semi-definite, whose determinant is
\begin{equation*}
f_{12}f_{13}\!+\!f_{12}f_{23}\!+\!f_{13}f_{23}=\frac{(f_{12}\!+\!f_{13}\!+\!f_{23})^2\!-\!f_{12}^2\!-\!f_{13}^2\!-\!f_{23}^2}2,
\end{equation*}
which is required to be nonnegative. In other words, we require $|f_{12}\!+\!f_{13}\!+\!f_{23}|\geqslant\sqrt{f_{12}^2\!+\!f_{13}^2\!+\!f_{23}^2}$. Its trace $f_{12}\!+\!f_{13}\!+\!2f_{23}$ is required to be nonpositive, therefore we require $f_{12}\!+\!f_{13}\!+\!f_{23}\leqslant-\sqrt{f_{12}^2\!+\!f_{13}^2\!+\!f_{23}^2}$ since $|f_{23}|\leqslant\sqrt{f_{12}^2\!+\!f_{13}^2\!+\!f_{23}^2}$.
\end{proof}

\begin{corollary}
If the matrix \eqref{decomposition of [f_ij]} is positive semi-definite, then no more than one of $f_{12}$, $f_{13}$ and $f_{23}$ is positive.
\end{corollary}

\begin{proof}
If they are not all zero, we can assume $f_{12}^2\!+\!f_{13}^2\!+\!f_{23}^2=1$. By Proposition \ref{sufficient and necessary condition for [f_{ij}] to be positive semi-definite}, $[f_{12},f_{13},f_{23}]^{\rm T}$ lies on the cap of the unit sphere below the plane $f_{12}\!+\!f_{13}\!+\!f_{23}=-1$ and thus not inside the first octant or its nearest three neighbors.
\end{proof}

\section{Phase diagrams of 1-D global minimizers}
\label{Phase diagrams: most favored repetends}

In this section we compare the free energy of several candidates that are representative of the global minimizers obtained in our numerical experiments. We present the candidates of the lowest free energy using phase diagrams. With the charge analogy drawn in Section \ref{physical analogy}, we give some explanations of the computed phase diagrams.

\subsection{Computational results}

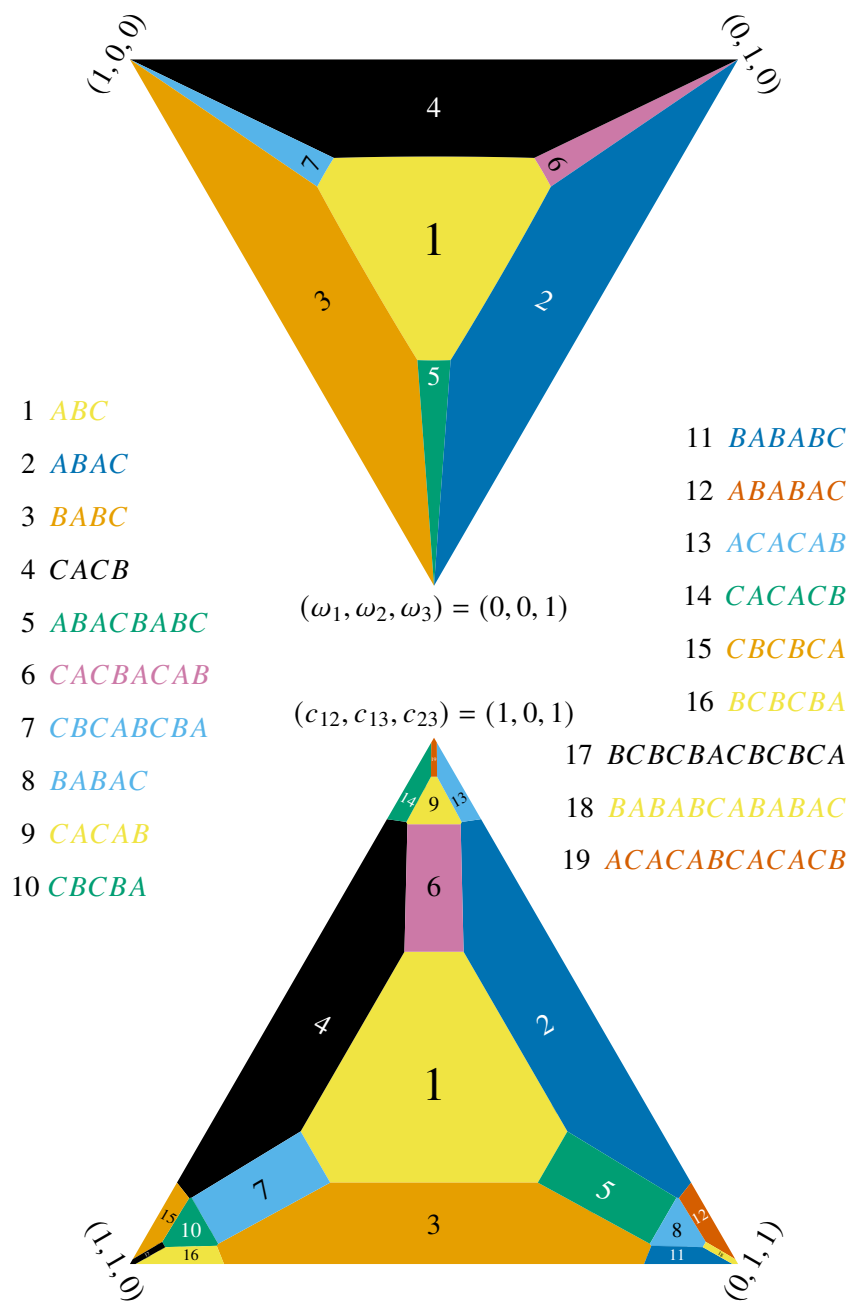
\begin{figure}
\centering
\begin{tikzpicture}
\draw (0pt,0pt) node[inner sep=0pt] {\includegraphics[width=270pt]{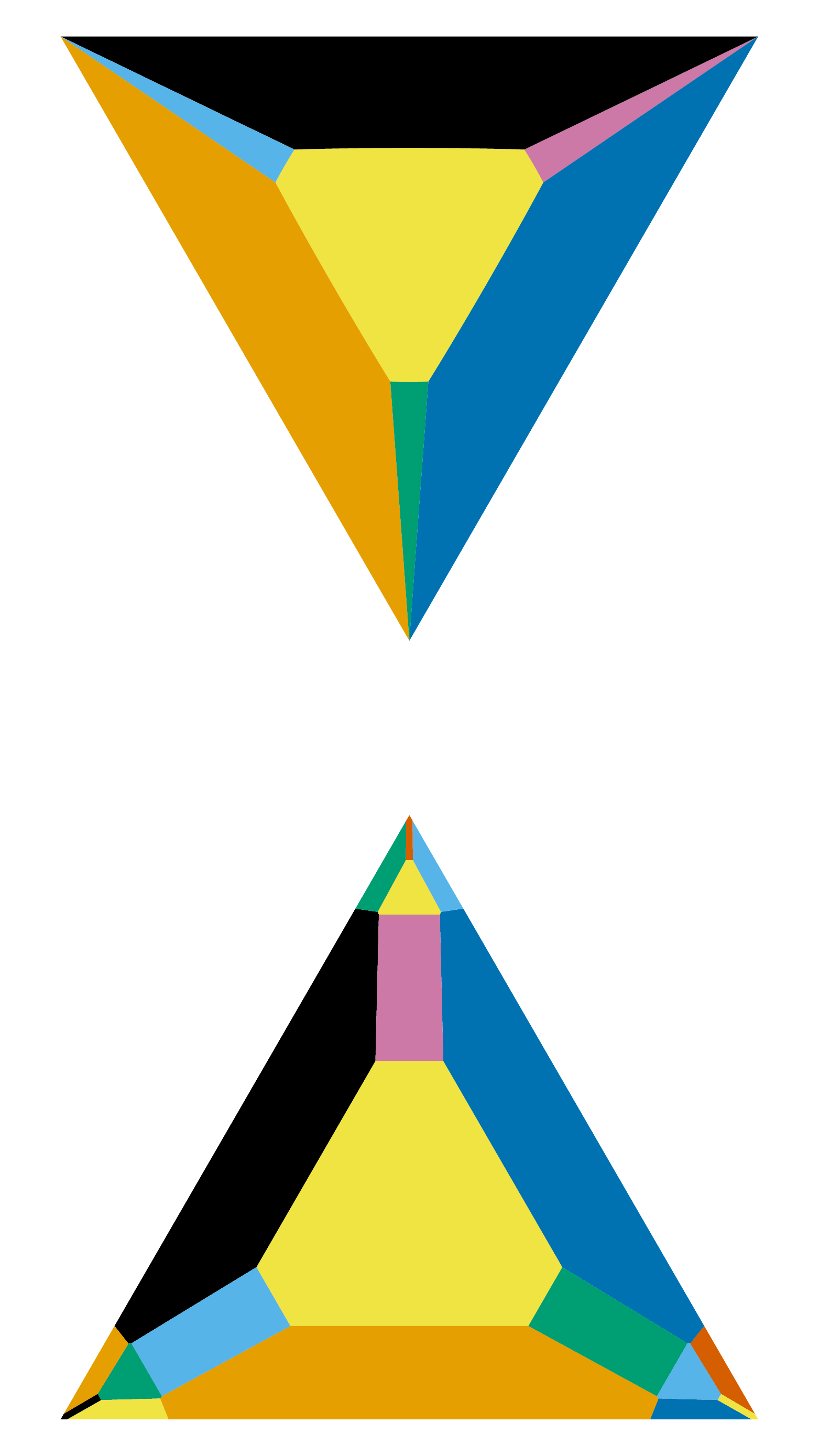}};
\node at (0pt,20pt) {$(\omega_1,\omega_2,\omega_3)=(0,0,1)$};
\node [rotate=60] at (-120pt,230pt) {$(1,0,0)$};
\node [rotate=-60] at (120pt,230pt) {$(0,1,0)$};
\node at (0pt,-20pt) {$(c_{12},c_{13},c_{23})=(1,0,1)$};
\node [rotate=-60] at (-121pt,-227pt) {$(1,1,0)$};
\node [rotate=60] at (121pt,-227pt) {$(0,1,1)$};
\node [anchor=west] at (-160pt,95pt) {1\hspace{5pt}\color{BW yellow}$ABC$};
\node [anchor=west] at (-160pt,75pt) {2\hspace{5pt}\color{BW blue}$ABAC$};
\node [anchor=west] at (-160pt,55pt) {3\hspace{5pt}\color{BW orange}$BABC$};
\node [anchor=west] at (-160pt,35pt) {4\hspace{5pt}\color{BW black}$CACB$};
\node [anchor=west] at (-160pt,15pt) {5\hspace{5pt}\color{BW bluish green}$ABACBABC$};
\node [anchor=west] at (-160pt,-5pt) {6\hspace{5pt}\color{BW reddish purple}$CACBACAB$};
\node [anchor=west] at (-160pt,-25pt) {7\hspace{5pt}\color{BW sky blue}$CBCABCBA$};
\node [anchor=west] at (-160pt,-45pt) {8\hspace{5pt}\color{BW sky blue}$BABAC$};
\node [anchor=west] at (-160pt,-65pt) {9\hspace{5pt}\color{BW yellow}$CACAB$};
\node [anchor=west] at (-160pt,-85pt) {\hspace{-4pt}10\hspace{3pt}\color{BW bluish green}$CBCBA$};
\node [anchor=east] at (160pt,85pt) {11\hspace{5pt}\color{BW blue}$BABABC$};
\node [anchor=east] at (160pt,65pt) {12\hspace{5pt}\color{BW vermillion}$ABABAC$};
\node [anchor=east] at (160pt,45pt) {13\hspace{5pt}\color{BW sky blue}$ACACAB$};
\node [anchor=east] at (160pt,25pt) {14\hspace{5pt}\color{BW bluish green}$CACACB$};
\node [anchor=east] at (160pt,5pt) {15\hspace{5pt}\color{BW orange}$CBCBCA$};
\node [anchor=east] at (160pt,-15pt) {16\hspace{5pt}\color{BW yellow}$BCBCBA$};
\node [anchor=east] at (160pt,-35pt) {17\hspace{5pt}\color{BW black}$BCBCBACBCBCA$};
\node [anchor=east] at (160pt,-55pt) {18\hspace{5pt}\color{BW yellow}$BABABCABABAC$};
\node [anchor=east] at (160pt,-75pt) {19\hspace{5pt}\color{BW vermillion}$ACACABCACACB$};
\node at (0pt,160pt) {\fontsize{18pt}{0pt}\selectfont 1};
\node [rotate=-30] at (42pt,137pt) {\color{white}\normalsize 2};
\node [rotate=30] at (-42pt,137pt) {\normalsize 3};
\node at (0pt,210pt) {\color{white}\normalsize 4};
\node at (0pt,108pt) {\color{white}\fontsize{10pt}{0pt}\selectfont 5};
\node [rotate=-60] at (46.5pt,188.5pt) {\fontsize{10pt}{0pt}\selectfont 6};
\node [rotate=60] at (-46.5pt,188.5pt) {\fontsize{10pt}{0pt}\selectfont 7};
\node at (0pt,-160pt) {\fontsize{18pt}{0pt}\selectfont 1};
\node [rotate=30] at (42pt,-137pt) {\color{white}\normalsize 2};
\node at (0pt,-213pt) {\normalsize 3};
\node [rotate=-30] at (-42pt,-137pt) {\color{white}\normalsize 4};
\node [rotate=-30] at (65pt,-200pt) {\color{white}\normalsize 5};
\node at (0pt,-85pt) {\normalsize 6};
\node [rotate=30] at (-65pt,-200pt) {\normalsize 7};
\node at (92pt,-215pt) {\fontsize{8pt}{0pt}\selectfont 8};
\node at (0pt,-54pt) {\fontsize{8pt}{0pt}\selectfont 9};
\node at (-92pt,-215pt) {\color{white}\fontsize{8pt}{0pt}\selectfont 10};
\node at (92pt,-224.5pt) {\color{white}\fontsize{6pt}{0pt}\selectfont 11};
\node [rotate=30] at (100.5pt,-210pt) {\color{white}\fontsize{6pt}{0pt}\selectfont 12};
\node [rotate=30] at (9.5pt,-52pt) {\fontsize{6pt}{0pt}\selectfont 13};
\node [rotate=-30] at (-9.7pt,-52pt) {\color{white}\fontsize{6pt}{0pt}\selectfont 14};
\node [rotate=-30] at (-100.5pt,-210pt) {\fontsize{6pt}{0pt}\selectfont 15};
\node at (-92pt,-224.5pt) {\fontsize{6pt}{0pt}\selectfont 16};
\node [rotate=30] at (-108.5pt,-224.2pt) {\color{white}\fontsize{2pt}{0pt}\selectfont 17};
\node [rotate=-30] at (108.5pt,-224.2pt) {\fontsize{2pt}{0pt}\selectfont 18};
\node at (-0.1pt,-37pt) {\color{white}\fontsize{2pt}{0pt}\selectfont 19};
\end{tikzpicture}
\caption{Plausible phase diagrams of 1-D global minimizers obtained by comparing the free energy of 19 candidates for $[\gamma_{ij}]$ given in \eqref{Ren's matrix} as $\gamma\rightarrow\infty$. Top: cross section along $\omega_1,\omega_2,\omega_3$ for $c_{12}=c_{13}=c_{23}$; the three vertices indicate that one of $\{\omega_i\}$ is $1$ and that the others are $0$. Bottom: cross section along $c_{12},c_{13},c_{23}$ for $\omega_1=\omega_2=\omega_3$; the three vertices indicate that one of $\{c_{ij}\}$ is $0$ and that the others are $1$. The interior points are convex combinations of the vertices.}
\label{PhaseDiagramsRen}
\end{figure}

\begin{figure}
\centering
\begin{tikzpicture}
\draw (0pt,0pt) node[inner sep=0pt] {\includegraphics[width=270pt]{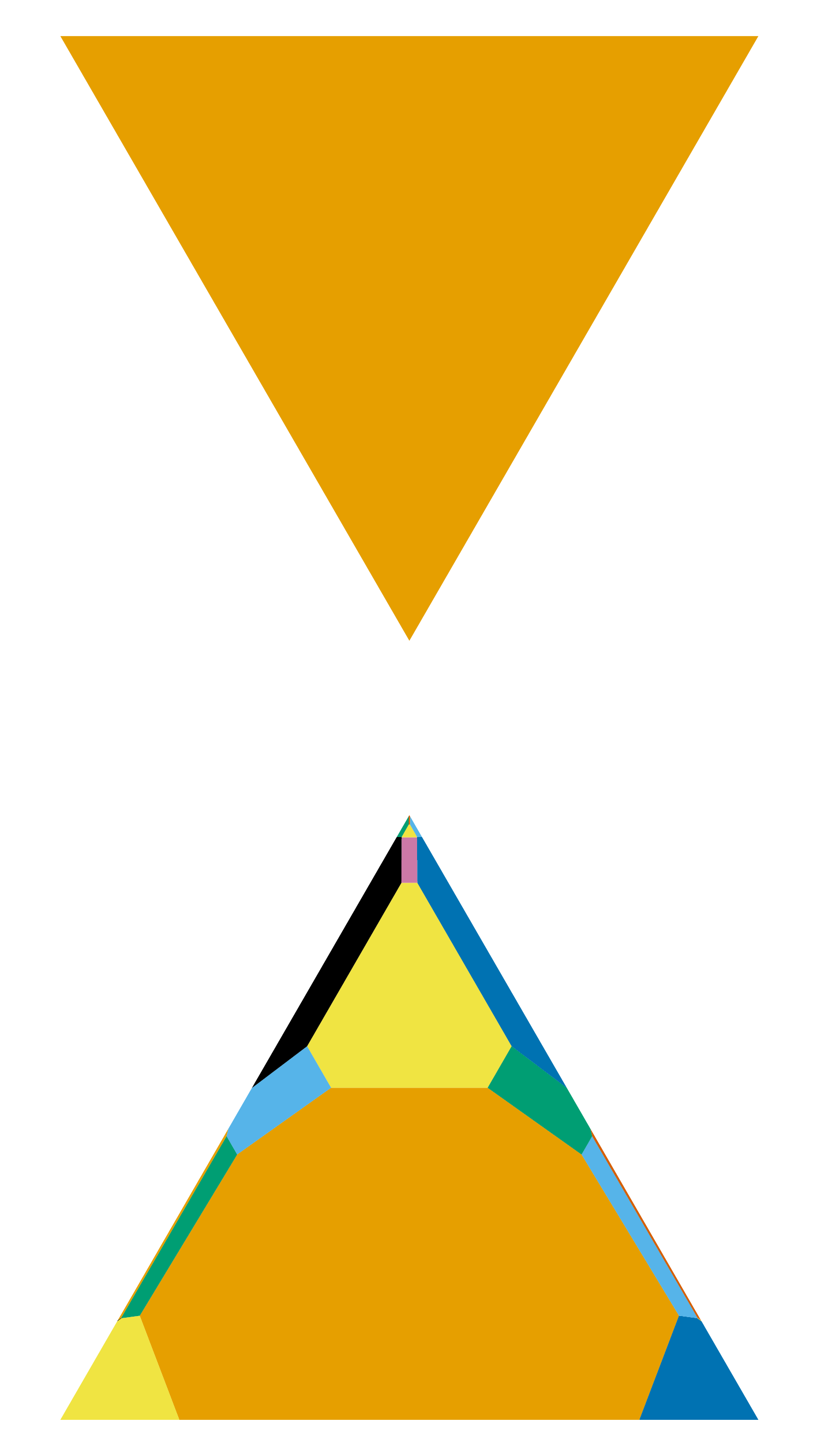}};
\node at (0pt,20pt) {$(\omega_1,\omega_2,\omega_3)=(0,0,1)$};
\node [rotate=60] at (-120pt,230pt) {$(1,0,0)$};
\node [rotate=-60] at (120pt,230pt) {$(0,1,0)$};
\node at (0pt,-16pt) {$(c_{12},c_{13},c_{23})=(1,0,1)$};
\node [rotate=-60] at (-121pt,-227pt) {$(1,1,0)$};
\node [rotate=60] at (121pt,-227pt) {$(0,1,1)$};
\node [anchor=west] at (-160pt,95pt) {1\hspace{5pt}\color{BW yellow}$ABC$};
\node [anchor=west] at (-160pt,75pt) {2\hspace{5pt}\color{BW blue}$ABAC$};
\node [anchor=west] at (-160pt,55pt) {3\hspace{5pt}\color{BW orange}$BABC$};
\node [anchor=west] at (-160pt,35pt) {4\hspace{5pt}\color{BW black}$CACB$};
\node [anchor=west] at (-160pt,15pt) {5\hspace{5pt}\color{BW bluish green}$ABACBABC$};
\node [anchor=west] at (-160pt,-5pt) {6\hspace{5pt}\color{BW reddish purple}$CACBACAB$};
\node [anchor=west] at (-160pt,-25pt) {7\hspace{5pt}\color{BW sky blue}$CBCABCBA$};
\node [anchor=west] at (-160pt,-45pt) {8\hspace{5pt}\color{BW sky blue}$BABAC$};
\node [anchor=west] at (-160pt,-65pt) {9\hspace{5pt}\color{BW yellow}$CACAB$};
\node [anchor=west] at (-160pt,-85pt) {\hspace{-4pt}10\hspace{3pt}\color{BW bluish green}$CBCBA$};
\node [anchor=east] at (160pt,85pt) {11\hspace{5pt}\color{BW blue}$BABABC$};
\node [anchor=east] at (160pt,65pt) {12\hspace{5pt}\color{BW vermillion}$ABABAC$};
\node [anchor=east] at (160pt,45pt) {13\hspace{5pt}\color{BW sky blue}$ACACAB$};
\node [anchor=east] at (160pt,25pt) {14\hspace{5pt}\color{BW bluish green}$CACACB$};
\node [anchor=east] at (160pt,5pt) {15\hspace{5pt}\color{BW orange}$CBCBCA$};
\node [anchor=east] at (160pt,-15pt) {16\hspace{5pt}\color{BW yellow}$BCBCBA$};
\node [anchor=east] at (160pt,-35pt) {17\hspace{5pt}\color{BW black}$BCBCBACBCBCA$};
\node [anchor=east] at (160pt,-55pt) {18\hspace{5pt}\color{BW yellow}$BABABCABABAC$};
\node [anchor=east] at (160pt,-75pt) {19\hspace{5pt}\color{BW vermillion}$ACACABCACACB$};
\node at (0pt,160pt) {\fontsize{21pt}{0pt}\selectfont 3};
\node at (0pt,-95pt) {\fontsize{18pt}{0pt}\selectfont 1};
\node [rotate=30] at (22pt,-75pt) {\color{white}\normalsize 2};
\node at (0pt,-180pt) {\fontsize{21pt}{0pt}\selectfont 3};
\node [rotate=-30] at (-21pt,-75pt) {\color{white}\normalsize 4};
\node [rotate=-25] at (40pt,-120pt) {\color{white}\normalsize 5};
\node at (0pt,-43pt) {\fontsize{9pt}{0pt}\selectfont 6};
\node [rotate=25] at (-40pt,-120pt) {\normalsize 7};
\node [rotate=30] at (76pt,-167pt) {\fontsize{8pt}{0pt}\selectfont 8};
\node at (0.1pt,-34.4pt) {\fontsize{4pt}{0pt}\selectfont 9};
\node [rotate=60] at (-76pt,-167pt) {\color{white}\fontsize{7pt}{0pt}\selectfont 10};
\node at (94pt,-215pt) {\color{white}\fontsize{10pt}{0pt}\selectfont 11};
\node [rotate=30] at (88pt,-175pt) {\fontsize{4pt}{0pt}\selectfont 12};
\node [rotate=-60] at (1.68pt,-33pt) {\fontsize{1.8pt}{0pt}\selectfont 13};
\node [rotate=60] at (-1.65pt,-33pt) {\color{white}\fontsize{1.8pt}{0pt}\selectfont 14};
\node [rotate=-30] at (-88pt,-175pt) {\fontsize{4pt}{0pt}\selectfont 15};
\node at (-94pt,-215pt) {\fontsize{10pt}{0pt}\selectfont 16};
\node [rotate=-30] at (-97.5pt,-195pt) {\fontsize{2pt}{0pt}\selectfont 17};
\node [rotate=30] at (97.5pt,-195pt) {\fontsize{2pt}{0pt}\selectfont 18};
\node at (0pt,-28pt) {\fontsize{2pt}{0pt}\selectfont 19};
\end{tikzpicture}
\caption{Plausible phase diagrams of 1-D global minimizers for $[\gamma_{ij}]$ given in \eqref{Ohta's matrix}. Top: cross section along $\omega_1,\omega_2,\omega_3$ for $c_{12}=c_{13}=c_{23}$. Bottom: cross section along $c_{12},c_{13},c_{23}$ for $\omega_1=\omega_2=\omega_3$. Colors adopted from \cite[Figure 2]{wong2011color}.}
\label{PhaseDiagramsOhta}
\end{figure}

For $\Omega=[0,1]$ with periodic boundary conditions, we obtain the phase diagrams by comparing 19 candidates for the minimizer of \eqref{definition of J}, each of which is a repetition of a certain repetend. Note that there are infinitely many repetends to consider, e.g., \monomer{A}\monomer{B}\monomer{A}\monomer{B} $\cdots\;$\monomer{A}\monomer{B}\monomer{C}, but we expect that they only matter near the boundaries of the phase diagrams. We retain 19 candidates for each of which \monomer{A}, \monomer{B} and \monomer{C} appear at least once in 6 consecutive layers. For our illustrative purposes, those candidates should be adequate to provide us a rough picture, although they might still be incomplete, e.g., the actual global minimizer might be the hybrid of two candidates, and the numerical experiments we have carried our so far are limited. The free energy of each candidate is derived in a similar manner to \eqref{free energy of ABAC} with the help of WOLFRAM MATHEMATICA. For convenience, we make some symmetry assumptions based on numerical results. For example, in the repetend \monomer{A}\monomer{B}\monomer{A}\monomer{B}\monomer{A}\monomer{C}, we assume the first and third \monomer{A} layers to have the same width (which is the distance between neighboring interfaces), and numerically optimize the width of the second \monomer{A} layer. We also assume the two \monomer{B} layers to have the same width.

For any choice of $[\gamma_{ij}]$, the functional \eqref{definition of J} has 5 degrees of freedom, i.e., $c_{12},c_{13},c_{23}$ and $\omega_1$, $\omega_2$ with $\omega_3=1\!-\!\omega_1\!-\!\omega_2$. However, since we are interested in the case where more and more microdomains emerge as the long range term becomes more and more dominant, we can require that $c_{12}\!+\!c_{13}\!+\!c_{23}=2$, and let the factor $\gamma$ (e.g., in \eqref{Ohta's matrix}) go to infinity. In this way, the degrees of freedom is reduced to 4.

To visualize such a 4-D phase diagram, we draw two cross sections: one is along $\omega_1,\omega_2,\omega_3$ with $c_{12}=c_{13}=c_{23}$, the other is along $c_{12},c_{13},c_{23}$ with $\omega_1=\omega_2=\omega_3$. Noticing that $[\omega_1,\omega_2,\omega_3]^{\rm T}$ with $\omega_i>0$ and $\sum_i\omega_i=1$ lies on an equilateral triangle as shown in Figure \ref{omega_iChoices.png}, we can draw the first cross section as an equilateral triangle, whose vertices represent one of $\{\omega_i\}$ being $1$ and the others being $0$. The second cross section can be analogously drawn as an equilateral triangle, whose vertices represent $(c_{12},c_{13},c_{23})$ being $(1,1,0)$, $(1,0,1)$ and $(0,1,1)$, respectively. Any $(c_{12},c_{13},c_{23})$ satisfying $c_{12}\!+\!c_{13}\!+\!c_{23}=2$ and triangle inequalities can therefore be written as a convex combination of the vertices, with coefficients being $(c_{12}\!+\!c_{13}\!-\!c_{23})/2$, $(c_{12}\!+\!c_{23}\!-\!c_{13})/2$ and $(c_{13}\!+\!c_{23}\!-\!c_{12})/2$, respectively. It is noteworthy that the interfacial tensions $\{c_{ij}\}$ are regarded here as independent of $\{\omega_i\}$ for convenience, although they can be derived from \eqref{definition of J_epsilon} as $\epsilon\rightarrow0$ and therefore depend on $W$ and $\{\omega_i\}$ \cite[Definition 3.3]{ren2003triblock2}.

For different choices of $[\gamma_{ij}]$, the comparison results are different. The phase diagrams in Figures \ref{PhaseDiagramsRen} and \ref{PhaseDiagramsOhta} correspond to \eqref{Ren's matrix} and \eqref{Ohta's matrix}, respectively. The two figures look similar to some extent, but noticeably in the latter the repetend \monomer{A}\monomer{B}\monomer{C}\monomer{B} is more dominant, and in particular occupies the entire top cross section. Since the lamellar \monomer{A}\monomer{B}\monomer{C}\monomer{B} phase is commonly observed in physical experiments on triblock copolymers (see Figure \ref{Ren vs. Mogi}), it is possible that the original choice \eqref{Ohta's matrix} by Ohta et al. better captures the self-assembly physics of certain types of triblock copolymers. However, from a mathematical point of view, it makes sense to choose among all the admissible matrices. As pointed out in Remark \ref{microscopic physical picture}, if we discount the overall factor and a trivial case (i.e., the zero matrix), such admissible matrices have 2 degrees of freedom and parametrize a spherical cap in Figure \ref{f_ijChoices.png}. It will be interesting to explore the entire range of admissible $[\gamma_{ij}]$ and the entire 6-D phase diagram. As mentioned before, a degenerate case \eqref{Blend's matrix} has already been studied in the context of the mixture of \monomer{A}\monomer{B} diblock copolymers and \monomer{C} homopolymers. This degenerate case is on the rim of the spherical cap and not covered by Figure \ref{PhaseDiagramsRen} or \ref{PhaseDiagramsOhta}, and the global minimizers are of patterns \monomer{A}\monomer{B}\monomer{A}\monomer{B} $\cdots\;$\monomer{C}, with \monomer{C} appearing only once \cite[Theorems 4 and 5]{van2008copolymer}. In Figures \ref{PhaseDiagramsRen} and \ref{PhaseDiagramsOhta} there are similar patterns such as \monomer{A}\monomer{B}\monomer{A}\monomer{B}\monomer{A}\monomer{C} $\cdots\,$\monomer{A}\monomer{B}\monomer{A}\monomer{B}\monomer{A}\monomer{C} for non-degenerate choices of $[\gamma_{ij}]$, but they are different in that \monomer{C} appears more than once, and they are of a different origin: nearly degenerate $\{c_{ij}\}$. For example, when $c_{12}\ll\min\{c_{13},c_{23}\}$ (i.e., the interfaces between \monomer{A} and \monomer{B} are barely penalized), $\{c_{ij}\}$ is said to be nearly degenerate, since triangle inequalities are almost violated.

\subsection{Explanations of phase diagrams}
\label{sec:explanation}

We now present some intuitive explanations of the computed phase diagrams, by illustrating how several repetends might achieve low free energy depending on the parameters. In essence and as expected, it is the competition between the potential energy (due to the nonlocal interactions between charges) and the interfacial energy weighted by various coefficients that gives rise to a variety of patterns.

\subsubsection{Repetend \texorpdfstring{\monomer{A}\monomer{B}\monomer{C}}{ABC}}
\label{Repetend ABC}

In the regions labelled by 1 in Figures \ref{PhaseDiagramsRen} and \ref{PhaseDiagramsOhta}, the repetend \monomer{A}\monomer{B}\monomer{C} has the lowest free energy among our candidates, suggesting that Ren\textendash Wei conjecture is likely to hold for the corresponding parameters. As a first step towards understanding this phenomenon, let us compare two patterns \monomer{A}\monomer{B}\monomer{C}\monomer{B}\monomer{A}\monomer{C} and \monomer{A}\monomer{B}\monomer{C}\monomer{A}\monomer{B}\monomer{C} for $c_{12}=c_{13}=c_{23}$. With periodic boundary conditions, we can visualize $[0,1]$ as a circle, and the Green's function $G(x,y)$ given in \eqref{expression of G on 1-D periodic cell} attains its minimum when $|x\!-\!y|=1/2$, indicating that in a 1-D torus, the Coulombic repulsion drives two point charges to opposite ends of a diameter. From the left side of Figure \ref{ABCBAC and ABCABC and balls with point charges} we can see that both patterns have the same number of interfaces, and thus the same short range term of the free energy. To compare their long range terms, for convenience we assume that the layer widths are uniform, i.e., all the layers of the same type have the same width. We know that \monomer{A} repels \monomer{A}, and \monomer{B} repels \monomer{B}, but \monomer{A} attracts \monomer{B} (i.e., in \eqref{Ren's matrix f_ij} and \eqref{Ohta's matrix f_ij}, we have $f_{11},f_{22}>0$ and $f_{12}<0$). In this way, the pattern \monomer{A}\monomer{B}\monomer{C}\monomer{B}\monomer{A}\monomer{C} is energetically not so favorable as \monomer{A}\monomer{B}\monomer{C}\monomer{A}\monomer{B}\monomer{C}, because in the former, the two \monomer{A} layers are closer, and the two \monomer{B} layers are closer, but each \monomer{A} layer is farther from a \monomer{B} layer.

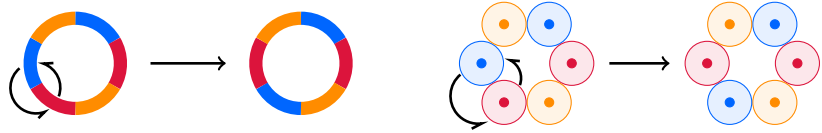
\begin{figure}[H]
\centering
\begin{tikzpicture}
\draw[color=blue A, line width = 5] ([shift={(0,0)}]30:0.6) arc (30:90:0.6);
\draw[color=red B, line width = 5] ([shift={(0,0)}]-30:0.6) arc (-30:30:0.6);
\draw[color=yellow C, line width = 5] ([shift={(0,0)}]-90:0.6) arc (-90:-30:0.6);
\draw[color=red B, line width = 5] ([shift={(0,0)}]-150:0.6) arc (-150:-90:0.6);
\draw[color=blue A, line width = 5] ([shift={(0,0)}]-210:0.6) arc (-210:-150:0.6);
\draw[color=yellow C, line width = 5] ([shift={(0,0)}]-270:0.6) arc (-270:-210:0.6);

\draw[line width = 1] ([shift={(0.08,0.08)}]240:0.6) arc (-20:70:0.34) -- ([shift={(0.3,-0.1)}]170:0.6);
\draw[line width = 1] ([shift={(-0.15,-0.17)}]170:0.6) arc (130:290:0.34) -- ([shift={(-0.19,-0.21)}]240:0.6);
\draw[-to, line width = 1](1,0) -- (2,0);

\draw[color=blue A, line width = 5] ([shift={(3,0)}]30:0.6) arc (30:90:0.6);
\draw[color=red B, line width = 5] ([shift={(3,0)}]-30:0.6) arc (-30:30:0.6);
\draw[color=yellow C, line width = 5] ([shift={(3,0)}]-90:0.6) arc (-90:-30:0.6);
\draw[color=blue A, line width = 5] ([shift={(3,0)}]-150:0.6) arc (-150:-90:0.6);
\draw[color=red B, line width = 5] ([shift={(3,0)}]-210:0.6) arc (-210:-150:0.6);
\draw[color=yellow C, line width = 5] ([shift={(3,0)}]-270:0.6) arc (-270:-210:0.6);

\draw[color=blue A, fill=blue A!10, line width = 0.4] ([shift={(6,0)}]60:0.6) circle (0.293);
\filldraw[color=blue A] ([shift={(6,0)}]60:0.6) circle (0.06);
\draw[color=red B, fill=red B!10, line width = 0.4] ([shift={(6,0)}]0:0.6) circle (0.293);
\filldraw[color=red B] ([shift={(6,0)}]0:0.6) circle (0.06);
\draw[color=yellow C, fill=yellow C!10, line width = 0.4] ([shift={(6,0)}]-60:0.6) circle (0.293);
\filldraw[color=yellow C] ([shift={(6,0)}]-60:0.6) circle (0.06);
\draw[color=red B, fill=red B!10, line width = 0.4] ([shift={(6,0)}]-120:0.6) circle (0.293);
\filldraw[color=red B] ([shift={(6,0)}]-120:0.6) circle (0.06);
\draw[color=blue A, fill=blue A!10, line width = 0.4] ([shift={(6,0)}]-180:0.6) circle (0.293);
\filldraw[color=blue A] ([shift={(6,0)}]-180:0.6) circle (0.06);
\draw[color=yellow C, fill=yellow C!10, line width = 0.4] ([shift={(6,0)}]-240:0.6) circle (0.293);
\filldraw[color=yellow C] ([shift={(6,0)}]-240:0.6) circle (0.06);

\draw[line width = 1] (6-0.09,-0.29) arc (-20:62:0.27) -- (6-0.09,0.03);
\draw[line width = 1] (6-0.88,-0.15) arc (130:277:0.37) -- (6-0.73,-0.87);
\draw[-to, line width = 1](7.1,0) -- (7.9,0);

\draw[color=blue A, fill=blue A!10, line width = 0.4] ([shift={(9,0)}]60:0.6) circle (0.293);
\filldraw[color=blue A] ([shift={(9,0)}]60:0.6) circle (0.06);
\draw[color=red B, fill=red B!10, line width = 0.4] ([shift={(9,0)}]0:0.6) circle (0.293);
\filldraw[color=red B] ([shift={(9,0)}]0:0.6) circle (0.06);
\draw[color=yellow C, fill=yellow C!10, line width = 0.4] ([shift={(9,0)}]-60:0.6) circle (0.293);
\filldraw[color=yellow C] ([shift={(9,0)}]-60:0.6) circle (0.06);
\draw[color=blue A, fill=blue A!10, line width = 0.4] ([shift={(9,0)}]-120:0.6) circle (0.293);
\filldraw[color=blue A] ([shift={(9,0)}]-120:0.6) circle (0.06);
\draw[color=red B, fill=red B!10, line width = 0.4] ([shift={(9,0)}]-180:0.6) circle (0.293);
\filldraw[color=red B] ([shift={(9,0)}]-180:0.6) circle (0.06);
\draw[color=yellow C, fill=yellow C!10, line width = 0.4] ([shift={(9,0)}]-240:0.6) circle (0.293);
\filldraw[color=yellow C] ([shift={(9,0)}]-240:0.6) circle (0.06);
\end{tikzpicture}
\caption{Left: patterns \monomer{A}\monomer{B}\monomer{C}\monomer{B}\monomer{A}\monomer{C} and \monomer{A}\monomer{B}\monomer{C}\monomer{A}\monomer{B}\monomer{C}. Right: simplification into point charges. {\color{blue A}Blue}, {\color{red B}red} and {\color{yellow C}orange} represent \monomer{A}, \monomer{B} and \monomer{C}, respectively.}
\label{ABCBAC and ABCABC and balls with point charges}
\end{figure}

For the pattern \monomer{A}\monomer{B}\monomer{C} $\cdots\,$\monomer{A}\monomer{B}\monomer{C} to be a local minimizer, Ren and Wei proved that having uniform layer widths is a sufficient condition \cite[Proposition 4.7]{ren2003triblock2}. But it is unknown if it is also a necessary condition, because it is unclear if the solutions to \cite[Equation (4.22)]{ren2003triblock2} are unique up to translation and reflection. However, if we only consider patterns with uniform layer widths and ignore the interfacial energy, then according to Proposition \ref{connection between OK and Ising}, each layer can be regarded as a point (generalized) charge at its midpoint. So the question becomes how to arrange tightly packed balls with charges at their centers, in order to minimize the potential energy between the charges, as shown on the right side of Figure \ref{ABCBAC and ABCABC and balls with point charges} (we assume that the diameters of the balls are the corresponding layer widths so that the charges are separated by proper distances). Note that this question also arises in Figure \ref{charges consist of quarks} naturally and its precise formulation can be found in Appendix \ref{Optimal arrangement of charged balls in 1-D}. We believe that the optimal arrangement is \monomer{A}\monomer{B}\monomer{C} $\cdots\,$\monomer{A}\monomer{B}\monomer{C} (see Conjecture \ref{ABC Ising model}), which seems to maximize the overall distances between balls of the same type, and minimize those of different types. In Proposition \ref{AB Ising model} we prove the binary analogue, i.e., for the interaction strength matrix given by the classical Coulomb's law,
\begin{equation*}
[f_{ij}]=\begin{bmatrix}1 & -1\\-1 & 1\end{bmatrix},
\end{equation*}
the optimal arrangement is \monomer{A}\monomer{B} $\cdots\,$\monomer{A}\monomer{B}. There may not be a straightforward quaternary analogue in 1-D, because \monomer{A}\monomer{B}\monomer{C}\monomer{D} $\cdots\,$\monomer{A}\monomer{B}\monomer{C}\monomer{D} lacks symmetry (\monomer{A} and \monomer{B} are in contact while \monomer{A} and \monomer{C} are not).

\begin{proposition}
\label{connection between OK and Ising}
Given a positive integer $n$, if each $\Omega_i$ is a union of $n$ intervals whose widths are all $\omega_i/n$, then the long range term of \eqref{definition of J} equals $2/n^2$ times the electrostatic potential energy $U$ defined in \eqref{electrostatic potential energy of A, B and C}, up to addition by a constant.
\end{proposition}
\begin{proof}
Let $\{I_k\}_{k=1}^{3n}$ denote those intervals (which cannot be nonoverlapping), and define $i_k$ so that $I_k\subseteq\Omega_{i_k}$ for each $k$. By Theorem \ref{equivalence between long range term and electrostatic potential energy}, the long range term of \eqref{definition of J} equals two times
\begin{equation*}
\frac12\sum_{k=1}^{3n}\sum_{m=1}^{3n}q_k\,f_{\,i_k\,j_m}\,q_m\fint_{I_m}\fint_{I_k}G(x,y)\dd{x}\dd{y},
\end{equation*}
where $j_m=i_m$, and $q_k$ is the $i_k$-th component of $\int_{I_k}\vec\rho(x)\dd{x}$. From the relation between $\vec\rho$ and $\bm1_{\Omega_i}$ shown in Theorem \ref{equivalence between long range term and electrostatic potential energy}, we can see $q_k=1/n$. Using the explicit form of $G$ given in \eqref{expression of G on 1-D periodic cell}, for $x_1,x_2,y_1,y_2\in[0,1]$ and $(x_1,x_2)\cap(y_1,y_2)=\varnothing$ we have
\begin{equation*}
\fint_{y_1}^{y_2}\hspace{-6pt}\fint_{x_1}^{x_2}\hspace{-5pt}G(x,y)\dd{x}\dd{y}=G\Big(\frac{x_1\!+\!x_2}2,\frac{y_1\!+\!y_2}2\Big)+\frac{(x_2\!-\!x_1)^2+(y_2\!-\!y_1)^2}{24},
\end{equation*}
where the first summand only involves the midpoints, and the second summand only yields constant terms because the width of each $I_k$ is fixed.
\end{proof}

\subsubsection{Repetend \texorpdfstring{\monomer{A}\monomer{B}\monomer{C}\monomer{B}}{ABCB}}

If we incorporate the interfacial energy (i.e., penalize adjacent balls of different types), then \monomer{A}\monomer{B}\monomer{C} $\cdots\,$\monomer{A}\monomer{B}\monomer{C} may no longer be the optimal arrangement. As shown on the left side of Figure \ref{ABCABC and ABCB}, after the swap, the two \monomer{A} layers merge into one, and the two \monomer{C} layers merge into one, so the short range term of the free energy decreases by $2c_{13}$. Depending on the parameters, if such a decrease outweighs the increase in long range term, then the pattern \monomer{A}\monomer{B}\monomer{C}\monomer{B} is better than \monomer{A}\monomer{B}\monomer{C}\monomer{A}\monomer{B}\monomer{C} with uniform layer widths. For example, when $c_{13}$ is large, the decrease in the short range term is large. This case corresponds to the regions labelled by 3 at the bottom of Figures \ref{PhaseDiagramsRen} and \ref{PhaseDiagramsOhta}. Another example is when $f_{13}$ becomes larger, \monomer{A} and \monomer{C} balls become less attractive (while $f_{13}$ is still negative) or even become repulsive (after $f_{13}$ becomes positive), so the increase in long range term becomes smaller. This qualitatively explains why the regions labelled by 3 are small in Figure \ref{PhaseDiagramsRen} (whose $f_{13}$ is negative), but large in Figure \ref{PhaseDiagramsOhta} (whose $f_{13}$ is positive). It also explains the existence of the region labelled by 3 in the upper half of Figure \ref{PhaseDiagramsRen}, which corresponds to relatively small $\omega_2$, and thus relatively large $f_{13}$ according to \eqref{Ren's matrix f_ij}.

\begin{figure}[H]
\centering
\begin{tikzpicture}
\draw[color=blue A, line width = 5] ([shift={(0,0)}]30:0.6) arc (30:90:0.6);
\draw[color=red B, line width = 5] ([shift={(0,0)}]-30:0.6) arc (-30:30:0.6);
\draw[color=yellow C, line width = 5] ([shift={(0,0)}]-90:0.6) arc (-90:-30:0.6);
\draw[color=blue A, line width = 5] ([shift={(0,0)}]-150:0.6) arc (-150:-90:0.6);
\draw[color=red B, line width = 5] ([shift={(0,0)}]-210:0.6) arc (-210:-150:0.6);
\draw[color=yellow C, line width = 5] ([shift={(0,0)}]-270:0.6) arc (-270:-210:0.6);

\draw ([shift={(0,0)}]90:0.69) -- ([shift={(0,0)}]90:0.8);
\draw ([shift={(0,0)}]-90:0.69) -- ([shift={(0,0)}]-90:0.8);
\draw[line width = 1] (-0.15,-0.45) arc (-20:18:1.3) -- (-0.04,0.3);
\draw[line width = 1] (-0.25,0.41) arc (20:-17:1.2) -- (-0.26,-0.2);
\draw[-to, line width = 1](1,0) -- (2,0);

\draw[color=blue A, line width = 5] ([shift={(3,0)}]30:0.6) arc (30:150:0.6);
\draw[color=red B, line width = 5] ([shift={(3,0)}]-30:0.6) arc (-30:30:0.6);
\draw[color=yellow C, line width = 5] ([shift={(3,0)}]-150:0.6) arc (-150:-30:0.6);
\draw[color=red B, line width = 5] ([shift={(3,0)}]-210:0.6) arc (-210:-150:0.6);

\draw ([shift={(3,0)}]90:0.69) -- ([shift={(3,0)}]90:0.8);
\draw ([shift={(3,0)}]-90:0.69) -- ([shift={(3,0)}]-90:0.8);

\draw[color=blue A, fill=blue A!10, line width = 0.4] ([shift={(6,0)}]60:0.6) circle (0.293);
\filldraw[color=blue A] ([shift={(6,0)}]60:0.6) circle (0.06);
\draw[color=red B, fill=red B!10, line width = 0.4] ([shift={(6,0)}]0:0.6) circle (0.293);
\filldraw[color=red B] ([shift={(6,0)}]0:0.6) circle (0.06);
\draw[color=yellow C, fill=yellow C!10, line width = 0.4] ([shift={(6,0)}]-60:0.6) circle (0.293);
\filldraw[color=yellow C] ([shift={(6,0)}]-60:0.6) circle (0.06);
\draw[color=blue A, fill=blue A!10, line width = 0.4] ([shift={(6,0)}]-120:0.6) circle (0.293);
\filldraw[color=blue A] ([shift={(6,0)}]-120:0.6) circle (0.06);
\draw[color=red B, fill=red B!10, line width = 0.4] ([shift={(6,0)}]-180:0.6) circle (0.293);
\filldraw[color=red B] ([shift={(6,0)}]-180:0.6) circle (0.06);
\draw[color=yellow C, fill=yellow C!10, line width = 0.4] ([shift={(6,0)}]-240:0.6) circle (0.293);
\filldraw[color=yellow C] ([shift={(6,0)}]-240:0.6) circle (0.06);

\draw[line width = 1] (6-0.07,-0.31) arc (-30:27:0.55) -- (6+0.06,0.12);
\draw[line width = 1] (6-0.18,0.23) arc (30:-19:0.52) -- (6-0.18,-0.05);
\draw[-to, line width = 1](7.1,0) -- (7.9,0);

\draw[color=blue A, fill=blue A!10, line width = 0.4] ([shift={(9,0)}]60:0.6) circle (0.293);
\filldraw[color=blue A] ([shift={(9,0)}]60:0.6) circle (0.06);
\draw[color=red B, fill=red B!10, line width = 0.4] ([shift={(9,0)}]0:0.6) circle (0.293);
\filldraw[color=red B] ([shift={(9,0)}]0:0.6) circle (0.06);
\draw[color=yellow C, fill=yellow C!10, line width = 0.4] ([shift={(9,0)}]-60:0.6) circle (0.293);
\filldraw[color=yellow C] ([shift={(9,0)}]-60:0.6) circle (0.06);
\draw[color=yellow C, fill=yellow C!10, line width = 0.4] ([shift={(9,0)}]-120:0.6) circle (0.293);
\filldraw[color=yellow C] ([shift={(9,0)}]-120:0.6) circle (0.06);
\draw[color=red B, fill=red B!10, line width = 0.4] ([shift={(9,0)}]-180:0.6) circle (0.293);
\filldraw[color=red B] ([shift={(9,0)}]-180:0.6) circle (0.06);
\draw[color=blue A, fill=blue A!10, line width = 0.4] ([shift={(9,0)}]-240:0.6) circle (0.293);
\filldraw[color=blue A] ([shift={(9,0)}]-240:0.6) circle (0.06);
\end{tikzpicture}
\caption{Left: patterns \monomer{A}\monomer{B}\monomer{C}\monomer{A}\monomer{B}\monomer{C} and \monomer{A}\monomer{B}\monomer{C}\monomer{B}. Right: simplification into point charges. {\color{blue A}Blue}, {\color{red B}red} and {\color{yellow C}orange} represent \monomer{A}, \monomer{B} and \monomer{C}, respectively.}
\label{ABCABC and ABCB}
\end{figure}
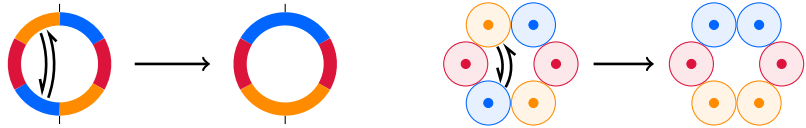

\subsubsection{Repetend \texorpdfstring{\monomer{A}\monomer{B}\monomer{A}\monomer{B}\monomer{C}}{ABABC}}

When $c_{12}$ is small, the pattern \monomer{A}\monomer{B}\monomer{A}\monomer{B}\monomer{C} can have lower free energy than \monomer{A}\monomer{B}\monomer{C}. As shown on the left side of Figure \ref{ABC and ABABC and balls with point charges}, by swapping a portion of \monomer{A} layer and a portion of \monomer{B} layer, while increasing the short range term of the free energy by $2c_{12}$, we can decrease the long range term. In fact, according to Proof of Proposition \ref{connection between OK and Ising}, this is equivalent to swapping two balls which carry charges at their centers, as shown on the right side of Figure \ref{ABC and ABABC and balls with point charges}. After the swap, the two \monomer{A} balls become farther apart, and the two \monomer{B} balls become farther apart, but \monomer{A} and \monomer{B} balls become closer. As mentioned before, \monomer{A} repels \monomer{A}, and \monomer{B} repels \monomer{B}, but \monomer{A} attracts \monomer{B}, so the potential energy between charges decreases. This argument can also be generalized to longer patterns like \monomer{A}\monomer{B}\monomer{A}\monomer{B}\monomer{A}\monomer{C}. In this way we can explain why patterns featuring well mixed \monomer{A} and \monomer{B} occupy the bottom right corners of Figures \ref{PhaseDiagramsRen} and \ref{PhaseDiagramsOhta}, where $c_{12}$ is relatively small. Under the assumption $\omega_1=\omega_2$, when $f_{13}=f_{23}$, we have $f_{11}=f_{22}$ by \eqref{decomposition of [f_ij]}, so \monomer{A}\monomer{B}\monomer{A}\monomer{B}\monomer{A}\monomer{C} has lower free energy than \monomer{B}\monomer{A}\monomer{B}\monomer{A}\monomer{B}\monomer{C} if and only if $c_{13}<c_{23}$. This scenario corresponds to the regions labelled by 11 and 12 in Figure \ref{PhaseDiagramsRen}. However, if $f_{11}<f_{22}$, then the long range term of \monomer{A}\monomer{B}\monomer{A}\monomer{B}\monomer{A}\monomer{C} is larger than that of \monomer{B}\monomer{A}\monomer{B}\monomer{A}\monomer{B}\monomer{C}, because \monomer{B} is farther apart in the latter. This explains why region 12 is smaller than region 11 in Figure \ref{PhaseDiagramsOhta}, and why the watershed between those two regions is shifted from $c_{13}=c_{23}$ towards $c_{13}<c_{23}$.

\begin{figure}[H]
\centering
\begin{tikzpicture}
\draw[color=blue A, line width = 5] ([shift={(0,0)}]-30:0.6) arc (-30:90:0.6);
\draw[color=yellow C, line width = 5] ([shift={(0,0)}]-150:0.6) arc (-150:-30:0.6);
\draw[color=red B, line width = 5] ([shift={(0,0)}]-270:0.6) arc (-270:-150:0.6);

\draw ([shift={(0,0)}]-30:0.69) -- ([shift={(0,0)}]-30:0.8);
\draw ([shift={(0,0)}]16.847:0.69) -- ([shift={(0,0)}]16.847:0.8);
\draw ([shift={(0,0)}]163.153:0.69) -- ([shift={(0,0)}]163.153:0.8);
\draw ([shift={(0,0)}]210:0.69) -- ([shift={(0,0)}]210:0.8);
\draw[line width = 1] (0.44,0.01) -- (-0.37, 0.01) -- (-0.25,0.08);
\draw[line width = 1] (-0.44,-0.1) -- (0.37, -0.1) -- (0.25,-0.17);
\draw[-to, line width = 1](1,0) -- (2,0);

\draw[color=blue A, line width = 5] ([shift={(3,0)}]16.847:0.6) arc (16.847:90:0.6);
\draw[color=red B, line width = 5] ([shift={(3,0)}]-30:0.6) arc (-30:16.847:0.6);
\draw[color=yellow C, line width = 5] ([shift={(3,0)}]-150:0.6) arc (-150:-30:0.6);
\draw[color=red B, line width = 5] ([shift={(3,0)}]90:0.6) arc (90:163.153:0.6);
\draw[color=blue A, line width = 5] ([shift={(3,0)}]210:0.6) arc (210:163.153:0.6);

\draw ([shift={(3,0)}]-30:0.69) -- ([shift={(3,0)}]-30:0.8);
\draw ([shift={(3,0)}]16.847:0.69) -- ([shift={(3,0)}]16.847:0.8);
\draw ([shift={(3,0)}]163.153:0.69) -- ([shift={(3,0)}]163.153:0.8);
\draw ([shift={(3,0)}]210:0.69) -- ([shift={(3,0)}]210:0.8);

\draw[color=blue A, fill=blue A!10, line width = 0.4] ([shift={(6,0)}]65.8698+77.3995-90:0.6) circle (0.32*1.1);
\filldraw[color=blue A] ([shift={(6,0)}]65.8698+77.3995-90:0.6) circle (0.066*1.1);
\draw[color=blue A, fill=blue A!10, line width = 0.4] ([shift={(6,0)}]77.3995-90:0.6) circle (0.32*0.9);
\filldraw[color=blue A] ([shift={(6,0)}]77.3995-90:0.6) circle (0.066*0.9);
\draw[color=yellow C, fill=yellow C!10, line width = 0.4] ([shift={(6,0)}]-90:0.6) circle (0.32*1.4);
\filldraw[color=yellow C] ([shift={(6,0)}]-90:0.6) circle (0.066*1.4);
\draw[color=red B, fill=red B!10, line width = 0.4] ([shift={(6,0)}]270-77.3995:0.6) circle (0.32*0.9);
\filldraw[color=red B] ([shift={(6,0)}]270-77.3995:0.6) circle (0.066*0.9);
\draw[color=red B, fill=red B!10, line width = 0.4] ([shift={(6,0)}]270-65.8698-77.3995:0.6) circle (0.32*1.1);
\filldraw[color=red B] ([shift={(6,0)}]270-65.8698-77.3995:0.6) circle (0.066*1.1);

\draw[line width = 1] (6+0.31,0.05) -- (6-0.25, 0.05) -- (6-0.13,0.12);
\draw[line width = 1] (6-0.28,-0.06) -- (6+0.22, -0.06) -- (6+0.10,-0.13);
\draw[-to, line width = 1](7.1,0) -- (7.9,0);

\draw[color=blue A, fill=blue A!10, line width = 0.4] ([shift={(9,0)}]65.8698+77.3995-90:0.6) circle (0.32*1.1);
\filldraw[color=blue A] ([shift={(9,0)}]65.8698+77.3995-90:0.6) circle (0.066*1.1);
\draw[color=red B, fill=red B!10, line width = 0.4] ([shift={(9,0)}]77.3995-90:0.6) circle (0.32*0.9);
\filldraw[color=red B] ([shift={(9,0)}]77.3995-90:0.6) circle (0.066*0.9);
\draw[color=yellow C, fill=yellow C!10, line width = 0.4] ([shift={(9,0)}]-90:0.6) circle (0.32*1.4);
\filldraw[color=yellow C] ([shift={(9,0)}]-90:0.6) circle (0.066*1.4);
\draw[color=blue A, fill=blue A!10, line width = 0.4] ([shift={(9,0)}]270-77.3995:0.6) circle (0.32*0.9);
\filldraw[color=blue A] ([shift={(9,0)}]270-77.3995:0.6) circle (0.066*0.9);
\draw[color=red B, fill=red B!10, line width = 0.4] ([shift={(9,0)}]270-65.8698-77.3995:0.6) circle (0.32*1.1);
\filldraw[color=red B] ([shift={(9,0)}]270-65.8698-77.3995:0.6) circle (0.066*1.1);
\end{tikzpicture}
\caption{Left: patterns \monomer{A}\monomer{B}\monomer{C} and \monomer{A}\monomer{B}\monomer{A}\monomer{B}\monomer{C}. Right: simplification into point charges. {\color{blue A}Blue}, {\color{red B}red} and {\color{yellow C}orange} represent \monomer{A}, \monomer{B} and \monomer{C}, respectively.}
\label{ABC and ABABC and balls with point charges}
\end{figure}
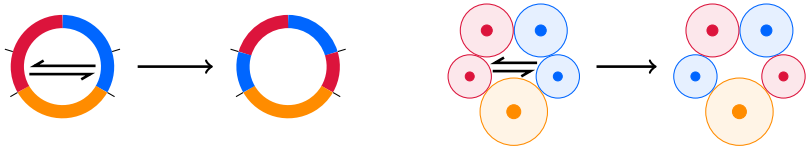

\subsubsection{Repetend \texorpdfstring{\monomer{A}\monomer{B}\monomer{A}\monomer{C}\monomer{B}\monomer{A}\monomer{B}\monomer{C}}{ABACBABC}}

Under the assumptions $\omega_1=\omega_2$, $c_{13}=c_{23}$ and $f_{13}=f_{23}$, the pattern \monomer{A}\monomer{B}\monomer{A}\monomer{C}\monomer{A}\monomer{B}\monomer{A}\monomer{C} with uniform layer widths has the same free energy as \monomer{B}\monomer{A}\monomer{B}\monomer{C}\monomer{B}\monomer{A}\monomer{B}\monomer{C} by symmetry. The pattern \monomer{A}\monomer{B}\monomer{A}\monomer{C}\monomer{B}\monomer{A}\monomer{B}\monomer{C}, which can be viewed as a hybrid or transitional or intermediate stage between the above two, can have lower free energy. As shown in the top of Figure \ref{BABCBABC and BABCABAC}, this intermediate stage can be reached in two steps. In the first step the layer types are changed, but the interfaces do not move. In the second step the wide \monomer{A} and \monomer{B} layers (located within 11$\sim$13 and 5$\sim$7 o'clock directions, respectively) are relaxed to shrink, while the narrow ones (1$\sim$2, 4$\sim$5, 7$\sim$8, and 10$\sim$11 o'clock) are relaxed to expand. The first step is equivalent to reversing the orientations of two dipoles, as shown in the bottom left of Figure \ref{BABCBABC and BABCABAC}. According to Lemma \ref{BABCBABC and BABCABAC have the same energy}, the potential energy between charges does not change after the first step. But it decreases after the second step, which is equivalent to swapping four tiny balls as shown in the bottom right of Figure \ref{BABCBABC and BABCABAC}. In fact, after the swapping, \monomer{A} balls become farther apart, and \monomer{B} balls become farther apart, but \monomer{A} balls become closer to \monomer{B} balls. As a result, the regions labelled by 5 arise as watersheds between regions 2 and 3 in Figure \ref{PhaseDiagramsRen}.

\begin{figure}[H]
\centering
\begin{tikzpicture}
\draw[color=red B, line width = 5] ([shift={(0,0)}]60:0.6) arc (60:120:0.6);
\draw[color=blue A, line width = 5] ([shift={(0,0)}]30:0.6) arc (30:60:0.6);
\draw[color=yellow C, line width = 5] ([shift={(0,0)}]-30:0.6) arc (-30:30:0.6);
\draw[color=blue A, line width = 5] ([shift={(0,0)}]-60:0.6) arc (-60:-30:0.6);
\draw[color=red B, line width = 5] ([shift={(0,0)}]-120:0.6) arc (-120:-60:0.6);
\draw[color=blue A, line width = 5] ([shift={(0,0)}]-150:0.6) arc (-150:-120:0.6);
\draw[color=yellow C, line width = 5] ([shift={(0,0)}]-210:0.6) arc (-210:-150:0.6);
\draw[color=blue A, line width = 5] ([shift={(0,0)}]-240:0.6) arc (-240:-210:0.6);

\draw[line width = 1] (-0.05,0.48) arc (0:-110:0.18) -- (-0.23,0.24);
\draw[line width = 1] (0.05,0.48) arc (180:290:0.18) -- (0.23,0.24);
\draw[line width = 1] (-0.45,0.55) arc (200:50:0.18) -- (-0.21,0.84);
\draw[line width = 1] (0.45,0.55) arc (-20:130:0.18) -- (0.21,0.84);
\draw[-to, line width = 1](1,0) -- (2,0);

\draw[color=blue A, line width = 5] ([shift={(3,0)}]60:0.6) arc (60:120:0.6);
\draw[color=red B, line width = 5] ([shift={(3,0)}]30:0.6) arc (30:60:0.6);
\draw[color=yellow C, line width = 5] ([shift={(3,0)}]-30:0.6) arc (-30:30:0.6);
\draw[color=blue A, line width = 5] ([shift={(3,0)}]-60:0.6) arc (-60:-30:0.6);
\draw[color=red B, line width = 5] ([shift={(3,0)}]-120:0.6) arc (-120:-60:0.6);
\draw[color=blue A, line width = 5] ([shift={(3,0)}]-150:0.6) arc (-150:-120:0.6);
\draw[color=yellow C, line width = 5] ([shift={(3,0)}]-210:0.6) arc (-210:-150:0.6);
\draw[color=red B, line width = 5] ([shift={(3,0)}]-240:0.6) arc (-240:-210:0.6);

\draw[line width = 1] (3-0.24,0.41) arc (160:195:1.32) -- (3-0.37,-0.28);
\draw[line width = 1] (3-0.21,-0.46) arc (-20:18:1.3) -- (3-0.1,0.29);
\draw[line width = 1] (3+0.24,0.41) arc (20:-15:1.32) -- (3+0.37,-0.28);
\draw[line width = 1] (3+0.21,-0.46) arc (200:162:1.3) -- (3+0.1,0.29);
\draw[-to, line width = 1](4,0) -- (5,0);

\draw ([shift={(3,0)}]60:0.69) -- ([shift={(3,0)}]60:0.8);
\draw ([shift={(3,0)}]64.641:0.69) -- ([shift={(3,0)}]64.641:0.8);
\draw ([shift={(3,0)}]115.359:0.69) -- ([shift={(3,0)}]115.359:0.8);
\draw ([shift={(3,0)}]120:0.69) -- ([shift={(3,0)}]120:0.8);
\draw ([shift={(3,0)}]-60:0.69) -- ([shift={(3,0)}]-60:0.8);
\draw ([shift={(3,0)}]-64.641:0.69) -- ([shift={(3,0)}]-64.641:0.8);
\draw ([shift={(3,0)}]-115.359:0.69) -- ([shift={(3,0)}]-115.359:0.8);
\draw ([shift={(3,0)}]-120:0.69) -- ([shift={(3,0)}]-120:0.8);

\draw[color=blue A, line width = 5] ([shift={(6,0)}]64.641:0.6) arc (64.641:115.359:0.6);
\draw[color=red B, line width = 5] ([shift={(6,0)}]30:0.6) arc (30:64.641:0.6);
\draw[color=yellow C, line width = 5] ([shift={(6,0)}]-30:0.6) arc (-30:30:0.6);
\draw[color=blue A, line width = 5] ([shift={(6,0)}]-64.641:0.6) arc (-64.641:-30:0.6);
\draw[color=red B, line width = 5] ([shift={(6,0)}]-115.359:0.6) arc (-115.359:-64.641:0.6);
\draw[color=blue A, line width = 5] ([shift={(6,0)}]-150:0.6) arc (-150:-115.359:0.6);
\draw[color=yellow C, line width = 5] ([shift={(6,0)}]-210:0.6) arc (-210:-150:0.6);
\draw[color=red B, line width = 5] ([shift={(6,0)}]-244.641:0.6) arc (-244.641:-210:0.6);

\draw ([shift={(6,0)}]60:0.69) -- ([shift={(6,0)}]60:0.8);
\draw ([shift={(6,0)}]64.641:0.69) -- ([shift={(6,0)}]64.641:0.8);
\draw ([shift={(6,0)}]115.359:0.69) -- ([shift={(6,0)}]115.359:0.8);
\draw ([shift={(6,0)}]120:0.69) -- ([shift={(6,0)}]120:0.8);
\draw ([shift={(6,0)}]-60:0.69) -- ([shift={(6,0)}]-60:0.8);
\draw ([shift={(6,0)}]-64.641:0.69) -- ([shift={(6,0)}]-64.641:0.8);
\draw ([shift={(6,0)}]-115.359:0.69) -- ([shift={(6,0)}]-115.359:0.8);
\draw ([shift={(6,0)}]-120:0.69) -- ([shift={(6,0)}]-120:0.8);
\end{tikzpicture}\\
\begin{tikzpicture}
\draw[color=yellow C, fill=yellow C!10, line width = 0.4] ([shift={(0,0)}]0:0.6) circle (0.165*1.41);
\filldraw[color=yellow C] ([shift={(0,0)}]0:0.6) circle (0.035*1.41);
\draw[color=blue A, fill=blue A!10, line width = 0.4] ([shift={(0,0)}]40.144:0.6) circle (0.165);
\filldraw[color=blue A] ([shift={(0,0)}]40.144:0.6) circle (0.035);
\draw[color=red B, fill=red B!10, line width = 0.4] ([shift={(0,0)}]73.381:0.6) circle (0.165);
\filldraw[color=red B] ([shift={(0,0)}]73.381:0.6) circle (0.035);
\draw[color=red B, fill=red B!10, line width = 0.4] ([shift={(0,0)}]180-73.381:0.6) circle (0.165);
\filldraw[color=red B] ([shift={(0,0)}]180-73.381:0.6) circle (0.035);
\draw[color=blue A, fill=blue A!10, line width = 0.4] ([shift={(0,0)}]180-40.144:0.6) circle (0.165);
\filldraw[color=blue A] ([shift={(0,0)}]180-40.144:0.6) circle (0.035);
\draw[color=yellow C, fill=yellow C!10, line width = 0.4] ([shift={(0,0)}]180:0.6) circle (0.165*1.41);
\filldraw[color=yellow C] ([shift={(0,0)}]180:0.6) circle (0.035*1.41);
\draw[color=blue A, fill=blue A!10, line width = 0.4] ([shift={(0,0)}]-40.144:0.6) circle (0.165);
\filldraw[color=blue A] ([shift={(0,0)}]-40.144:0.6) circle (0.035);
\draw[color=red B, fill=red B!10, line width = 0.4] ([shift={(0,0)}]-73.381:0.6) circle (0.165);
\filldraw[color=red B] ([shift={(0,0)}]-73.381:0.6) circle (0.035);
\draw[color=red B, fill=red B!10, line width = 0.4] ([shift={(0,0)}]73.381-180:0.6) circle (0.165);
\filldraw[color=red B] ([shift={(0,0)}]73.381-180:0.6) circle (0.035);
\draw[color=blue A, fill=blue A!10, line width = 0.4] ([shift={(0,0)}]40.144-180:0.6) circle (0.165);
\filldraw[color=blue A] ([shift={(0,0)}]40.144-180:0.6) circle (0.035);

\draw[line width = 1] (-0.1,0.4) arc (0:-90:0.17) -- (-0.19,0.19);
\draw[line width = 1] (0.1,0.4) arc (180:270:0.17) -- (0.19,0.19);
\draw[line width = 1] (-0.56,0.55) arc (200:65:0.18) -- (-0.37,0.86);
\draw[line width = 1] (0.56,0.55) arc (-20:115:0.18) -- (0.37,0.86);
\draw[-to, line width = 1](1.1,0) -- (1.9,0);

\draw[color=yellow C, fill=yellow C!10, line width = 0.4] ([shift={(3,0)}]0:0.6) circle (0.165*1.41);
\filldraw[color=yellow C] ([shift={(3,0)}]0:0.6) circle (0.035*1.41);
\draw[color=red B, fill=red B!10, line width = 0.4] ([shift={(3,0)}]40.144:0.6) circle (0.165);
\filldraw[color=red B] ([shift={(3,0)}]40.144:0.6) circle (0.035);
\draw[color=blue A, fill=blue A!10, line width = 0.4] ([shift={(3,0)}]73.381:0.6) circle (0.165);
\filldraw[color=blue A] ([shift={(3,0)}]73.381:0.6) circle (0.035);
\draw[color=blue A, fill=blue A!10, line width = 0.4] ([shift={(3,0)}]180-73.381:0.6) circle (0.165);
\filldraw[color=blue A] ([shift={(3,0)}]180-73.381:0.6) circle (0.035);
\draw[color=red B, fill=red B!10, line width = 0.4] ([shift={(3,0)}]180-40.144:0.6) circle (0.165);
\filldraw[color=red B] ([shift={(3,0)}]180-40.144:0.6) circle (0.035);
\draw[color=yellow C, fill=yellow C!10, line width = 0.4] ([shift={(3,0)}]180:0.6) circle (0.165*1.41);
\filldraw[color=yellow C] ([shift={(3,0)}]180:0.6) circle (0.035*1.41);
\draw[color=blue A, fill=blue A!10, line width = 0.4] ([shift={(3,0)}]-40.144:0.6) circle (0.165);
\filldraw[color=blue A] ([shift={(3,0)}]-40.144:0.6) circle (0.035);
\draw[color=red B, fill=red B!10, line width = 0.4] ([shift={(3,0)}]-73.381:0.6) circle (0.165);
\filldraw[color=red B] ([shift={(3,0)}]-73.381:0.6) circle (0.035);
\draw[color=red B, fill=red B!10, line width = 0.4] ([shift={(3,0)}]73.381-180:0.6) circle (0.165);
\filldraw[color=red B] ([shift={(3,0)}]73.381-180:0.6) circle (0.035);
\draw[color=blue A, fill=blue A!10, line width = 0.4] ([shift={(3,0)}]40.144-180:0.6) circle (0.165);
\filldraw[color=blue A] ([shift={(3,0)}]40.144-180:0.6) circle (0.035);

\draw (3,0.63) -- (3,0.77);
\node at (3,0.9) {\fontsize{7pt}{0pt}\selectfont0};
\draw (3.84,0) -- (3.9,0);
\node at (4.1,0) {\fontsize{9pt}{0pt}\selectfont$\frac14$};
\draw ([shift={(3,0)}]73.381:0.773) -- ([shift={(3,0)}]73.381:0.83);
\node at (3.32,0.87) {\fontsize{8pt}{0pt}\selectfont$x_1$};
\draw ([shift={(3,0)}]40.144:0.773) -- ([shift={(3,0)}]40.144:0.83);
\node at (3.73,0.62) {\fontsize{8pt}{0pt}\selectfont$x_2$};
\draw ([shift={(3,0)}]-40.144:0.773) -- ([shift={(3,0)}]-40.144:0.83);
\node at (3.8,-0.53) {\fontsize{8pt}{0pt}\selectfont$y_2$};
\draw ([shift={(3,0)}]-73.381:0.773) -- ([shift={(3,0)}]-73.381:0.83);
\node at (3.41,-0.83) {\fontsize{8pt}{0pt}\selectfont$y_1$};
\draw ([shift={(3,0)}]73.381-180:0.773) -- ([shift={(3,0)}]73.381-180:0.83);
\node at (2.64,-0.83) {\fontsize{8pt}{0pt}\selectfont$y_3$};
\draw ([shift={(3,0)}]40.144-180:0.773) -- ([shift={(3,0)}]40.144-180:0.83);
\node at (2.22,-0.53) {\fontsize{8pt}{0pt}\selectfont$y_4$};
\draw ([shift={(3,0)}]180-40.144:0.773) -- ([shift={(3,0)}]180-40.144:0.83);
\node at (2.25,0.63) {\fontsize{8pt}{0pt}\selectfont$x_4$};
\draw ([shift={(3,0)}]180-73.381:0.773) -- ([shift={(3,0)}]180-73.381:0.83);
\node at (2.6,0.87) {\fontsize{8pt}{0pt}\selectfont$x_3$};

\draw[color=yellow C, fill=yellow C!10, line width = 0.4] ([shift={(6,0)}]0:0.6) circle (0.153*1.43);
\filldraw[color=yellow C] ([shift={(6,0)}]0:0.6) circle (0.032*1.43);
\draw[color=red B, fill=red B!10, line width = 0.4] ([shift={(6,0)}]37.5:0.6) circle (0.153);
\filldraw[color=red B] ([shift={(6,0)}]37.5:0.6) circle (0.032);
\draw[color=blue A, fill=blue A!10, line width = 0.15] ([shift={(6,0)}]57.52:0.6) circle (0.153*0.3);
\filldraw[color=blue A] ([shift={(6,0)}]57.52:0.6) circle (0.009);
\draw[color=blue A, fill=blue A!10, line width = 0.4] ([shift={(6,0)}]76.05:0.6) circle (0.153*0.9);
\filldraw[color=blue A] ([shift={(6,0)}]76.05:0.6) circle (0.032);
\draw[color=blue A, fill=blue A!10, line width = 0.4] ([shift={(6,0)}]180-76.05:0.6) circle (0.153*0.9);
\filldraw[color=blue A] ([shift={(6,0)}]180-76.05:0.6) circle (0.032);
\draw[color=blue A, fill=blue A!10, line width = 0.15] ([shift={(6,0)}]180-57.52:0.6) circle (0.153*0.3);
\filldraw[color=blue A] ([shift={(6,0)}]180-57.52:0.6) circle (0.009);
\draw[color=red B, fill=red B!10, line width = 0.4] ([shift={(6,0)}]180-37.5:0.6) circle (0.153);
\filldraw[color=red B] ([shift={(6,0)}]180-37.5:0.6) circle (0.032);
\draw[color=yellow C, fill=yellow C!10, line width = 0.4] ([shift={(6,0)}]180:0.6) circle (0.153*1.43);
\filldraw[color=yellow C] ([shift={(6,0)}]180:0.6) circle (0.032*1.43);
\draw[color=blue A, fill=blue A!10, line width = 0.4] ([shift={(6,0)}]180+37.5:0.6) circle (0.153);
\filldraw[color=blue A] ([shift={(6,0)}]180+37.5:0.6) circle (0.032);
\draw[color=red B, fill=red B!10, line width = 0.15] ([shift={(6,0)}]180+57.52:0.6) circle (0.153*0.3);
\filldraw[color=red B] ([shift={(6,0)}]180+57.52:0.6) circle (0.009);
\draw[color=red B, fill=red B!10, line width = 0.4] ([shift={(6,0)}]180+76.05:0.6) circle (0.153*0.9);
\filldraw[color=red B] ([shift={(6,0)}]180+76.05:0.6) circle (0.032);
\draw[color=red B, fill=red B!10, line width = 0.4] ([shift={(6,0)}]360-76.05:0.6) circle (0.153*0.9);
\filldraw[color=red B] ([shift={(6,0)}]360-76.05:0.6) circle (0.032);
\draw[color=red B, fill=red B!10, line width = 0.15] ([shift={(6,0)}]360-57.52:0.6) circle (0.153*0.3);
\filldraw[color=red B] ([shift={(6,0)}]360-57.52:0.6) circle (0.009);
\draw[color=blue A, fill=blue A!10, line width = 0.4] ([shift={(6,0)}]360-37.5:0.6) circle (0.153);
\filldraw[color=blue A] ([shift={(6,0)}]360-37.5:0.6) circle (0.032);

\draw[line width = 1] (6-0.37,0.57) arc (85:270:0.57) -- (6-0.52,-0.63);
\draw[line width = 1] (6-0.29,-0.46) arc (-20:17:1.34) -- (6-0.17,0.29);
\draw[line width = 1] (6+0.37,0.57) arc (95:-90:0.57) -- (6+0.52,-0.63);
\draw[line width = 1] (6+0.29,-0.46) arc (200:163:1.34) -- (6+0.17,0.29);
\draw[-to, line width = 1](7.1,0) -- (7.9,0);

\draw[color=yellow C, fill=yellow C!10, line width = 0.4] ([shift={(9,0)}]0:0.6) circle (0.153*1.43);
\filldraw[color=yellow C] ([shift={(9,0)}]0:0.6) circle (0.032*1.43);
\draw[color=red B, fill=red B!10, line width = 0.4] ([shift={(9,0)}]37.5:0.6) circle (0.153);
\filldraw[color=red B] ([shift={(9,0)}]37.5:0.6) circle (0.032);
\draw[color=red B, fill=red B!10, line width = 0.15] ([shift={(9,0)}]57.52:0.6) circle (0.153*0.3);
\filldraw[color=red B] ([shift={(9,0)}]57.52:0.6) circle (0.009);
\draw[color=blue A, fill=blue A!10, line width = 0.4] ([shift={(9,0)}]76.05:0.6) circle (0.153*0.9);
\filldraw[color=blue A] ([shift={(9,0)}]76.05:0.6) circle (0.032);
\draw[color=blue A, fill=blue A!10, line width = 0.4] ([shift={(9,0)}]180-76.05:0.6) circle (0.153*0.9);
\filldraw[color=blue A] ([shift={(9,0)}]180-76.05:0.6) circle (0.032);
\draw[color=red B, fill=red B!10, line width = 0.15] ([shift={(9,0)}]180-57.52:0.6) circle (0.153*0.3);
\filldraw[color=red B] ([shift={(9,0)}]180-57.52:0.6) circle (0.009);
\draw[color=red B, fill=red B!10, line width = 0.4] ([shift={(9,0)}]180-37.5:0.6) circle (0.153);
\filldraw[color=red B] ([shift={(9,0)}]180-37.5:0.6) circle (0.032);
\draw[color=yellow C, fill=yellow C!10, line width = 0.4] ([shift={(9,0)}]180:0.6) circle (0.153*1.43);
\filldraw[color=yellow C] ([shift={(9,0)}]180:0.6) circle (0.032*1.43);
\draw[color=blue A, fill=blue A!10, line width = 0.4] ([shift={(9,0)}]180+37.5:0.6) circle (0.153);
\filldraw[color=blue A] ([shift={(9,0)}]180+37.5:0.6) circle (0.032);
\draw[color=blue A, fill=blue A!10, line width = 0.15] ([shift={(9,0)}]180+57.52:0.6) circle (0.153*0.3);
\filldraw[color=blue A] ([shift={(9,0)}]180+57.52:0.6) circle (0.009);
\draw[color=red B, fill=red B!10, line width = 0.4] ([shift={(9,0)}]180+76.05:0.6) circle (0.153*0.9);
\filldraw[color=red B] ([shift={(9,0)}]180+76.05:0.6) circle (0.032);
\draw[color=red B, fill=red B!10, line width = 0.4] ([shift={(9,0)}]360-76.05:0.6) circle (0.153*0.9);
\filldraw[color=red B] ([shift={(9,0)}]360-76.05:0.6) circle (0.032);
\draw[color=blue A, fill=blue A!10, line width = 0.15] ([shift={(9,0)}]360-57.52:0.6) circle (0.153*0.3);
\filldraw[color=blue A] ([shift={(9,0)}]360-57.52:0.6) circle (0.009);
\draw[color=blue A, fill=blue A!10, line width = 0.4] ([shift={(9,0)}]360-37.5:0.6) circle (0.153);
\filldraw[color=blue A] ([shift={(9,0)}]360-37.5:0.6) circle (0.032);
\end{tikzpicture}
\caption{Top: patterns \monomer{A}\monomer{B}\monomer{A}\monomer{C}\monomer{A}\monomer{B}\monomer{A}\monomer{C} and \monomer{A}\monomer{B}\monomer{A}\monomer{C}\monomer{B}\monomer{A}\monomer{B}\monomer{C}. Bottom: simplification into point charges. {\color{blue A}Blue}, {\color{red B}red} and {\color{yellow C}orange} represent \monomer{A}, \monomer{B} and \monomer{C}, respectively.}
\label{BABCBABC and BABCABAC}
\end{figure}
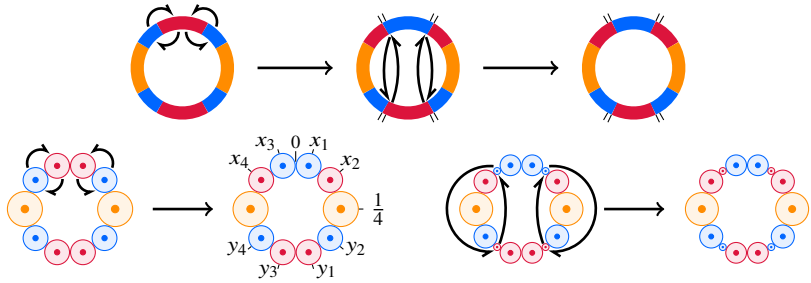

\begin{lemma}
\label{BABCBABC and BABCABAC have the same energy}
Under the assumption $f_{13}=f_{23}$, the two arrangements shown in the bottom left of Figure \ref{BABCBABC and BABCABAC} have the same potential energy (defined by \eqref{electrostatic potential energy of A, B and C}) if all the \monomer{A} and \monomer{B} balls have the same size.
\end{lemma}

\begin{proof}
Because of the symmetric assumption about \monomer{A} and \monomer{B}, the transformation from the first arrangement to the second one only changes the potential energy between the top four point charges and the bottom four. Before the transformation, the total potential energy is
\begin{equation}
\label{point charge ABACABAC}
U=\sum_{k=1}^4\sum_{m=1}^4f_{\,i_k\,j_m}\,G(x_k,y_m)+\text{constant},
\end{equation}
where $x_k$ and $y_m$ are labelled in the bottom left of Figure \ref{BABCBABC and BABCABAC}, and $i_k=j_k=\big(3-(-1)^k\big)/2$. After the transformation, the $i_k$ in \eqref{point charge ABACABAC} should be replaced by $3\!-\!i_k$, and the increase in $U$ is
\begin{equation*}
\begin{aligned}
&\sum_{k=1}^4\sum_{m=1}^4(-1)^{k+m}\,(f_{12}\!-\!f_{11})\,G(x_k,y_m)\\
=\;&(f_{12}\!-\!f_{11})\!\sum_{k=1}^4(-1)^{k}\sum_{m=1}^4(-1)^{m}\,G(x_k,y_m),
\end{aligned}
\end{equation*}
because we have $f_{\;3-i_k\;j_m}-f_{\,i_k\,j_m}=(-1)^{k+m}\,(f_{12}\!-\!f_{11})$ under the assumption $f_{13}=f_{23}$ and thus $f_{11}=f_{22}$ by \eqref{decomposition of [f_ij]}. By \eqref{expression of G on 1-D periodic cell}, the function $\sum_{m=1}^4(-1)^{m}\,G(x,y_m)$ is a quadratic function on $[0,1]\backslash(y_2,y_4)$ satisfying periodic boundary conditions, with the coefficient of $x^2$ being 0, and it is symmetric with respect to $x=1/2$ because of the vertical symmetry of the arrangement, so it must be constant for $x\in[0,1]\backslash(y_2,y_4)$. Therefore $U$ increases by 0.
\end{proof}

\begin{remark}$ $
\begin{enumerate}[label=(\roman*)]

\item If we assume $f_{13}>f_{23}$ instead of $f_{13}=f_{23}$, then we have $f_{11}<f_{22}$ by \eqref{decomposition of [f_ij]}, and the long range term is larger for the repetend \monomer{A}\monomer{B}\monomer{A}\monomer{C} than \monomer{B}\monomer{A}\monomer{B}\monomer{C} (because in the later, \monomer{B} is farther apart, and \monomer{A} is farther from \monomer{C}, but \monomer{B} is closer to \monomer{C}). To compensate for this and equalize their free energy, we need the condition $c_{13}<c_{23}$ instead of $c_{13}=c_{23}$, which is why the region labelled by 5 in Figure \ref{PhaseDiagramsOhta} is shifted towards $c_{13}<c_{23}$.

\item Note that our arguments can be readily adapted for longer patterns like \monomer{A}\monomer{B}\monomer{A}\monomer{B}\monomer{A}\monomer{C}\monomer{B}\monomer{A}\monomer{B}\monomer{A}\monomer{B}\monomer{C}, because each layer may be finely discretized into many balls, so that all the balls have (almost) the same size, and Lemma \ref{BABCBABC and BABCABAC have the same energy} can be generalized to an arbitrary number of \monomer{A} and \monomer{B} balls. Alternatively, Lemma \ref{BABCBABC and BABCABAC have the same energy} can be generalized to the continuum level. In this way we can explain the regions labelled by 17, 18 and 19 in Figures \ref{PhaseDiagramsRen} and \ref{PhaseDiagramsOhta}.

\end{enumerate}
\end{remark}

\section{Discussion}
\label{discussion}

The results in Section \ref{Phase diagrams: most favored repetends} are limited to 1-D where the geometry is restricted, but it is of natural interest to explore higher dimensions. For approximately equal $\{\omega_i\}$, the 2-D and 3-D minimizers are expected to be lamellar in certain parameter regions. However, when at least one of $\{\omega_i\}$ is small, the 2-D and 3-D minimizers are expected to take droplet-like shapes which have lower interfacial energy than lamellae (see, e.g., \cite[Figures 3 and 5]{bates2000block}). Therefore, it is unclear whether the 2-D and 3-D minimizers can have lamellar patterns like \monomer{A}\monomer{B}\monomer{C}\monomer{B} and \monomer{A}\monomer{B}\monomer{A}\monomer{C}\monomer{B}\monomer{A}\monomer{B}\monomer{C} which occupy the fringes of the top phase diagram in Figure \ref{PhaseDiagramsRen}.

To examine whether our 1-D results can be extended to higher dimensions, we may compare our lamellar candidates to some 2-D or 3-D droplet-like candidates, and determine the threshold of $\omega_i$ above which lamellae are preferred over droplets. To this end, the first step would be to figure out the corresponding 2-D or 3-D candidates. In 2-D, some stationary points have been found under various parameters. Among them, it is plausible that the repetend \monomer{A}\monomer{B}\monomer{C} corresponds to double bubbles \cite{ren2015double}, that \monomer{A}\monomer{B}\monomer{A}\monomer{C} corresponds to single bubbles \cite{ren2019stationary} or core-shells \cite{ren2017stationary} (with \monomer{A} being the shells and \monomer{B} being the cores), and that \monomer{A}\monomer{B}\monomer{A}\monomer{C}\monomer{B}\monomer{A}\monomer{B}\monomer{C} corresponds to core-shells with \monomer{A} and \monomer{B} taking turns as the shells and cores (although such alternating core-shells have not been found in the literature yet). Also in 2-D, for certain parameters the global minimizers have been found in \cite{alama2019periodic} to be coexisting single and double bubbles, which might correspond to patterns like \monomer{C}\monomer{A}\monomer{C}\monomer{A}\monomer{B} (where \monomer{A} coexists with \monomer{A}\monomer{B}) in our 1-D setting. However, the results in \cite{alama2019periodic} cannot be directly generalized from 2-D to 1-D, because their proofs rely on the singularity of the 2-D Green's function of $-\Delta$ in order to extract the leading order terms. Numerical studies indicate that \monomer{A}\monomer{B} double bubbles can also coexist with both \monomer{A} and \monomer{B} single bubbles simultaneously in 2-D \cite[Figure 4.4-(b)]{wang2019bubble}, but there is no 1-D analogue among the candidates considered in Figures \ref{PhaseDiagramsRen} and \ref{PhaseDiagramsOhta} (a possible 1-D analogue can be \monomer{C}\monomer{A}\monomer{C}\monomer{B}\monomer{A}\monomer{C}\monomer{B} with \monomer{C} as the background). In 2-D, core-shells have not yet been proved to be global minimizers, and in fact they are not minimizers without the long range term (i.e., $\gamma=0$) \cite{lawlor2014double}. It is unclear how large $\gamma$ needs to be for them to be global minimizers.

Currently we do not have any 1-D candidate corresponding to \cite[Figure 4.4-(b)]{wang2019bubble}, where \monomer{A}\monomer{B} double bubbles coexist with both \monomer{A} and \monomer{B} single bubbles.

In the context of block copolymers, there are possible physical interpretations of some repetends found in Section \ref{Phase diagrams: most favored repetends}:
\begin{itemize}

\item The repetend \monomer{A}\monomer{B}\monomer{C}\monomer{B} can be seen as the head-to-head and tail-to-tail arrangement of the triblock copolymers \monomer{A}\monomer{B}\monomer{C}, i.e., \monomer{A}\monomer{B}\monomer{C} \monomer{C}\monomer{B}\monomer{A} \monomer{A}\monomer{B}\monomer{C} \monomer{C}\monomer{B}\monomer{A} $\cdots$

\item The repetend \monomer{A}\monomer{C}\monomer{A}\monomer{B}\monomer{C}\monomer{A}\monomer{C}\monomer{B} can be seen as the head-to-head and tail-to-tail arrangement of the pentablock terpolymers \monomer{A}\monomer{C}\monomer{B}\monomer{A}\monomer{C}, i.e., \monomer{A}\monomer{C}\monomer{B}\monomer{A}\monomer{C} \monomer{C}\monomer{A}\monomer{B}\monomer{C}\monomer{A} \monomer{A}\monomer{C}\monomer{B}\monomer{A}\monomer{C} \monomer{C}\monomer{A}\monomer{B}\monomer{C}\monomer{A} $\cdots$

\end{itemize}
The finding of the latter repetend, along with some other long repetends, is unexpected. Note that the ternary O\textendash K free energy was proposed to model the simplest triblock copolymers. Thus, it remains unclear whether it can also describe other multiblock terpolymers, nor do we know how the block sequence is related to $[\gamma_{ij}]$. However, it seems plausible that for star and cyclic architectures (which are symmetric, see \cite[Figure 1]{feng2017block}), $[\gamma_{ij}]$ should be of a symmetric form like \eqref{Ren's matrix}. For linear \monomer{A}\monomer{B}\monomer{C} and \monomer{A}\monomer{C}\monomer{B}\monomer{A}\monomer{C} architectures, $[\gamma_{ij}]$ should depend on $b$ in a different way from $a$ and $c$, like \eqref{Ohta's matrix}. It will be interesting to find out the relation between $[\gamma_{ij}]$ and the molecular architecture.

There are surely many other interesting directions to explore. For example, in this work, only $(-\Delta)^{-1}$ is considered as the kernel of the long range term. One may also consider other positive definite kernels like $(-\Delta)^{-s}$ \cite[Appendix]{chan2019lamellar}, $(-\mathcal{L}_\delta)^{-1}$ for some nonlocal diffusion operator $\mathcal{L}_\delta$  \cite{luo2022nonlocal,du2012sirev,du2019nonlocal}, and the screened Coulomb kernel $(\kappa^2I-\Delta)^{-1}$ \cite{muratov2010droplet}. It is also interesting to study quaternary systems \cite[Equation (4.1)]{wang2018analysis}, and our auxiliary results in Sections \ref{physical analogy} and \ref{conditions on [gamma_{ij}]} can be readily generalized to a system of arbitrarily many phases. The diffuse interface version \eqref{definition of J_epsilon} is of interest as well, but when $\epsilon\ll1$, the results should be parallel to the sharp interface limit \eqref{definition of J} via $\Gamma$-convergence \cite[Section 3]{ren2003triblock2}. As a final remark, our discussions here can be viewed as attempts towards a broad topic: competitions between short and long range interactions in multicomponent systems \cite{burchard2015nonlocal,muratov2019nonlocal,cicalese2016ground,giuliani2009pattern,mossa2004ground,liu2008self}.

\section*{Acknowledgments}
The authors would like to thank Professors Chong Wang, Xiaofeng Ren, An-Chang Shi, Juncheng Wei and Yanxiang Zhao for helpful discussions.

\appendix

\section{Underlying mechanism of interactions between generalized charges}
\label{underlying mechanism}

The decomposition \eqref{decomposition of [f_ij]} of $[f_{ij}]$ suggests a possible way to interpret the interactions between the generalized charges:
\begin{itemize}
\item Each charge consists of two sub-charges, 
as shown in Figure \ref{charges consist of quarks}. For example, a charge of type 1 consists of a sub-charge of type \quark{$12$} and a sub-charge of type \quark{$13$}.
\item For $i\neq j$, the potential energy due to the interaction between the pair of sub-charges \quark{$ij$} at $\vec x$ and \quark{$ji$} at $\vec y$ is $f_{ij}\,G(\vec x,\vec y)$, and that between \quark{$ij$} at $\vec x$ and \quark{$ij$} at $\vec y$ is $-f_{ij}\,G(\vec x,\vec y)$. There is no interaction within other pairs.
\end{itemize}

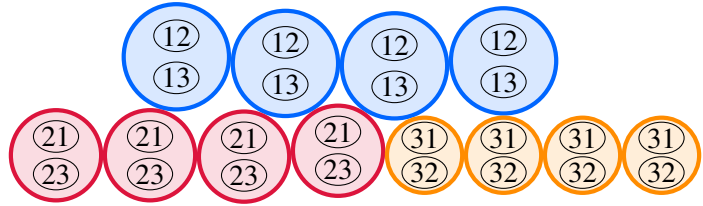
\begin{figure}[H]
\centering
\begin{tikzpicture}
\filldraw[color=blue A, fill=blue A!15, ultra thick](0+7.3,1.31) circle (0.7);
\draw (7.3+0,0.28+1.31) ellipse (0.3 and 0.21) node {12};
\draw (7.3+0,-0.28+1.31) ellipse (0.3 and 0.21) node {13};
\filldraw[color=blue A, fill=blue A!15, ultra thick](1.45+7.3,1.222) circle (0.7);
\draw (7.3+1.45,0.28+1.222) ellipse (0.3 and 0.21) node {12};
\draw (7.3+1.45,-0.28+1.222) ellipse (0.3 and 0.21) node {13};
\filldraw[color=blue A, fill=blue A!15, ultra thick](2.90+7.3,1.19) circle (0.7);
\draw (7.3+2.90,0.28+1.19) ellipse (0.3 and 0.21) node {12};
\draw (7.3+2.90,-0.28+1.19) ellipse (0.3 and 0.21) node {13};
\filldraw[color=blue A, fill=blue A!15, ultra thick](4.35+7.3,1.256) circle (0.7);
\draw (7.3+4.35,0.28+1.256) ellipse (0.3 and 0.21) node {12};
\draw (7.3+4.35,-0.28+1.256) ellipse (0.3 and 0.21) node {13};
\filldraw[color=red B, fill=red B!15, ultra thick](5.70,0) circle (0.6);
\draw (5.70,0.25) ellipse (0.3 and 0.21) node {21};
\draw (5.70,-0.25) ellipse (0.3 and 0.21) node {23};
\filldraw[color=red B, fill=red B!15, ultra thick](6.95,0) circle (0.6);
\draw (6.95,0.25) ellipse (0.3 and 0.21) node {21};
\draw (6.95,-0.25) ellipse (0.3 and 0.21) node {23};
\filldraw[color=red B, fill=red B!15, ultra thick](8.20,-0.018) circle (0.6);
\draw (8.20,0.25-0.018) ellipse (0.3 and 0.21) node {21};
\draw (8.20,-0.25-0.018) ellipse (0.3 and 0.21) node {23};
\filldraw[color=red B, fill=red B!15, ultra thick](9.45,0.06) circle (0.6);
\draw (9.45,0.25+0.06) ellipse (0.3 and 0.21) node {21};
\draw (9.45,-0.25+0.06) ellipse (0.3 and 0.21) node {23};
\filldraw[color=yellow C, fill=yellow C!15, ultra thick](10.60,0) circle (0.5);
\draw (10.60,0.23) ellipse (0.3 and 0.21) node {31};
\draw (10.60,-0.23) ellipse (0.3 and 0.21) node {32};
\filldraw[color=yellow C, fill=yellow C!15, ultra thick](11.65,0) circle (0.5);
\draw (11.65,0.23) ellipse (0.3 and 0.21) node {31};
\draw (11.65,-0.23) ellipse (0.3 and 0.21) node {32};
\filldraw[color=yellow C, fill=yellow C!15, ultra thick](12.70,0) circle (0.5);
\draw (12.70,0.23) ellipse (0.3 and 0.21) node {31};
\draw (12.70,-0.23) ellipse (0.3 and 0.21) node {32};
\filldraw[color=yellow C, fill=yellow C!15, ultra thick](13.75,0) circle (0.5);
\draw (13.75,0.23) ellipse (0.3 and 0.21) node {31};
\draw (13.75,-0.23) ellipse (0.3 and 0.21) node {32};
\end{tikzpicture}
\caption{Decomposition of charges into sub-charges. {\color{blue A}Blue}, {\color{red B}red} and {\color{yellow C}orange} balls represent charges of types 1, 2 and 3, respectively. Each charge consists of two (out of six) types of sub-charges.}
\label{charges consist of quarks}
\end{figure}

We further assume that there are van der Waals forces between charges, that the cohesive forces between charges of the same type are stronger than the adhesive forces between charges of different types, so that the charges behave like immiscible fluids in the thermodynamic limit at certain temperature, with the interfacial tensions being $\{c_{ij}\}$. (For a general account of the interfacial tension, see, e.g., \cite{birdi2015introduction}.) To ensure the overall charge neutrality, there must be the same number of charges of each type (Figure \ref{charges consist of quarks} shows four charges of each type). In accordance with the volume constraints, the volume ratio of charges of types 1, 2 and 3 is $\omega_1:\omega_2:\omega_3$. On the continuum level, such a discrete particle system can be described by the ternary O\textendash K free energy, and thus serves as an intuitive analogy. This analogy naturally raises a question on how to arrange balls in 1-D in order to minimize the potential energy between charges, a question further discussed in Section \ref{Repetend ABC} and Appendix \ref{Optimal arrangement of charged balls in 1-D}. The decomposition into simple interactions between sub-charges helps us answer this question for $f_{12},f_{13},f_{23}\leqslant0$ (see Proposition \ref{arrangement in the all nonpositive case}).

\section{Optimal arrangement of charged balls in 1-D}
\label{Optimal arrangement of charged balls in 1-D}

\subsection{Binary case}
\label{Binary case of long range Ising model in 1-D}
Given a positive integer $n$, consider a 1-D periodic cell $[0,1]$ packed with $n$ balls of type \monomer{A} and $n$ balls of type \monomer{B}, with unit amounts of positive and negative point charges at their centers, respectively. We assume that all the balls have the same radius $\frac1{4n}$, and that their centers are located at $\frac k{2n}$ for $k = 1,2,\cdots,2n$. Let $u:\big\{\frac k{2n}\big\}_{k=1}^{2n}\rightarrow\pm1$ represent the arrangement of the balls, with $1$ and $-1$ denoting \monomer{A} and \monomer{B}, respectively, then the total potential energy between charges can be written as
\begin{equation}
\label{electrostatic potential energy of A and B}
U(u)=\frac12\sum_{k=1}^{2n}\sum_{m=1}^{2n}u\Big(\frac k{2n}\Big)u\Big(\frac m{2n}\Big)G\Big(\frac k{2n},\frac m{2n}\Big),
\end{equation}
where $G$ is given by \eqref{expression of G on 1-D periodic cell}. By Proposition \ref{AB Ising model}, the alternating arrangement \monomer{A}\monomer{B} $\cdots\;$\monomer{A}\monomer{B} minimizes $U$, as shown in Figure \ref{binary alternating balls}. (Throughout Appendix \ref{Optimal arrangement of charged balls in 1-D}, "minimize" refers to "globally minimize".) This is a long range variant of the 1-D antiferromagnetic Ising model without external fields, subject to the zero overall spin constraint.
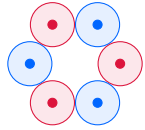
\begin{figure}[H]
\centering
\begin{tikzpicture}
\draw[color=blue A, fill=blue A!10, line width = 0.4] ([shift={(0,0)}]60:0.6) circle (0.293);
\filldraw[color=blue A] ([shift={(0,0)}]60:0.6) circle (0.06);
\draw[color=red B, fill=red B!10, line width = 0.4] ([shift={(0,0)}]0:0.6) circle (0.293);
\filldraw[color=red B] ([shift={(0,0)}]0:0.6) circle (0.06);
\draw[color=blue A, fill=blue A!10, line width = 0.4] ([shift={(0,0)}]-60:0.6) circle (0.293);
\filldraw[color=blue A] ([shift={(0,0)}]-60:0.6) circle (0.06);
\draw[color=red B, fill=red B!10, line width = 0.4] ([shift={(0,0)}]-120:0.6) circle (0.293);
\filldraw[color=red B] ([shift={(0,0)}]-120:0.6) circle (0.06);
\draw[color=blue A, fill=blue A!10, line width = 0.4] ([shift={(0,0)}]-180:0.6) circle (0.293);
\filldraw[color=blue A] ([shift={(0,0)}]-180:0.6) circle (0.06);
\draw[color=red B, fill=red B!10, line width = 0.4] ([shift={(0,0)}]-240:0.6) circle (0.293);
\filldraw[color=red B] ([shift={(0,0)}]-240:0.6) circle (0.06);
\end{tikzpicture}
\caption{An optimal arrangement for $n=3$. {\color{blue A}Blue} and {\color{red B}red} represent \monomer{A} and \monomer{B} balls with positive and negative point charges at their centers, respectively.}
\label{binary alternating balls}
\end{figure}

\begin{proposition}
\label{AB Ising model}
The minimizer of \eqref{electrostatic potential energy of A and B} is (up to translation) $u\big(\frac k{2n}\big)=(-1)^{k-1}$ for $k = 1,2,\cdots,2n$.
\end{proposition}

\begin{proof} \emph{(Inspired by {\normalfont\cite[Proof of Proposition 3.1]{ren2000multiplicity}})}
Within an optimal arrangement, let us prove that $l$ must be $1$ for any segment like the following
{\large\begin{equation*}
\cdots \overset{\frac1{2n}}{\text{\monomer{A}}}\underbrace{\overset{\frac2{2n}}{\text{\monomer{B}}} \cdots \overset{\hspace{-5pt}\frac{l+1}{2n}\hspace{-3pt}}{\text{\monomer{B}}}}_{\hspace{-20pt}l\;\text{consecutive \monomer{B}}\hspace{-20pt}}\overset{\hspace{-3pt}\frac{l+2}{2n}\hspace{-5pt}}{\text{\monomer{A}}} \cdots{\normalsize\text{,\quad where $l=1,2,\cdots,n.$}}
\end{equation*}}
\noindent By translational invariance, we can assume that the above segment occupies the first $l\!+\!2$ sites $\big\{\frac k{2n}\big\}_{k=1}^{l+2}$. By assumption, $U$ should not decrease if we swap the \monomer{A} and \monomer{B} at the first two sites. Let $u^*$ represent this optimal arrangement, then we have
\begin{equation}
\label{test: exchange A and B}
\sum_{m=3}^{2n}u^*\Big(\frac m{2n}\Big)G\Big(\frac 1{2n},\frac m{2n}\Big)\leqslant\sum_{m=3}^{2n}u^*\Big(\frac m{2n}\Big)G\Big(\frac 2{2n},\frac m{2n}\Big).
\end{equation}
Define the electrostatic potential
\begin{equation*}
V(x;u)=\sum_{m=1}^{2n}u\Big(\frac m{2n}\Big)G\Big(x,\frac m{2n}\Big),\quad x\in[0,1].
\end{equation*}
Then from \eqref{test: exchange A and B} we know
\begin{equation*}
\begin{aligned}
V\Big(\frac1{2n};u^*\Big)\!-\!V\Big(\frac2{2n};u^*\Big)&\leqslant G\Big(\frac 1{2n},\frac 1{2n}\Big)\!-\!2G\Big(\frac 1{2n},\frac 2{2n}\Big)\!+\!G\Big(\frac 2{2n},\frac 2{2n}\Big)\\
&=\frac1{2n}-\frac1{4n^2}.
\end{aligned}
\end{equation*}
From \eqref{expression of G on 1-D periodic cell} and $\sum_{m=1}^{2n}u\big(\frac m{2n}\big)=0$ we can see that the potential $V$ is piecewise quadratic in $x$ with the coefficient of $x^2$ being 0, and thus is linear on every subinterval $[\frac{k-1}{2n},\frac k{2n}]$ for $k=1,\cdots,2n$. Define the electrostatic field $E(x;u)=\dd V(x;u)/\dd x$, then $E$ is piecewise constant in $x$ and
\begin{equation*}
E(x;u^*)\big|_{x\in\big(\frac1{2n},\frac2{2n}\big)}\geqslant\frac1{2n}\!-\!1>-1.
\end{equation*}
Analogously we have
\begin{equation*}
E(x;u^*)\big|_{x\in\big(\frac{l+1}{2n},\frac{l+2}{2n}\big)}<1.
\end{equation*}
We also know for $k=1,\cdots,2n\!-\!1$,
\begin{equation*}
E(x;u)\big|_{x\in\big(\frac{k-1}{2n},\frac k{2n}\big)}-E(x;u)\big|_{x\in\big(\frac k{2n},\frac{k+1}{2n}\big)}=u\Big(\frac k{2n}\Big).
\end{equation*}
Consequently, we have $-\,l>(-1)-1$, that is, $l=1$.
\end{proof}

\begin{remark}
The above proof can be generalized to the case where different types of balls have different sizes, that is, the positions of the point charges are no longer uniform (i.e., $x_k=\frac k{2n}$ for $k=1,2,\cdots,2n$). Instead, we have
\begin{equation*}
x_k-x_{k-1}=\frac{1\!+\!u(x_k)\omega}{4n}+\frac{1\!+\!u(x_{k-1})\omega}{4n},\quad\text{for }k=1,2,\cdots,2n,
\end{equation*}
where $x_0=0$ is identified with $x_{2n}=1$ so that $u(x_0)=u(x_{2n})$, and the radii of \monomer{A} and \monomer{B} balls are $\frac{1+\omega}{4n}$ and $\frac{1-\omega}{4n}$, respectively, for some $\omega\in(-1,1)\textbackslash\{0\}$.
\end{remark}

\subsection{Ternary case}
\label{Ternary case of long range Ising model in 1-D}
Given a positive integer $n$, consider a 1-D periodic cell $[0,1]$ packed with $n$ balls of type \monomer{A}, $n$ balls of type \monomer{B}, and $n$ balls of type \monomer{C}, labelled by 1, 2 and 3, respectively. Balls of type $i$ are assumed to have the radius $\omega_i/(2n)$. We also assume $f_{ij}\,G(x,y)$ to be the potential energy between a ball of type $i$ centered at $x$ and a ball of type $j$ centered at $y$. The total potential energy $U$ is the sum of all the pairwise interactions:
\begin{equation}
\label{electrostatic potential energy of A, B and C}
U=\frac12\sum_{k=1}^{3n}\sum_{m=1}^{3n}f_{\,i_k\,j_m}G(x_k,y_m),
\end{equation}
where $G$ is given by \eqref{expression of G on 1-D periodic cell}. The $k$-th ball, which is of type $i_k$, is centered at $x_k$, and the $m$-th ball, which is of type $j_m$, is centered at $y_m$. Therefore we have the following relation
\begin{equation*}
x_k\!-\!x_{k-1}=\frac{\omega_{i_k}}{2n}\!+\!\frac{\omega_{i_{k-1}}}{2n},\;\;x_k=y_k,\;\;\text{and}\;\;i_k=j_k,
\end{equation*}
for $k=1,2,\cdots,3n$, with $x_0=y_0=0$ and $i_0=j_0=i_{3n}=j_{3n}$.

\begin{conjecture}
\label{ABC Ising model}
For any admissible interaction strength matrix $[f_{ij}]$ and any positive $\{\omega_i\}$ satisfying $\sum_i\omega_i=1$, the arrangement \monomer{A}\monomer{B}\monomer{C} $\cdots$ \monomer{A}\monomer{B}\monomer{C} (i.e., $i_k\equiv k\mod 3$) minimizes \eqref{electrostatic potential energy of A, B and C}, as illustrated on the right side of Figure \ref{ABCBAC and ABCABC and balls with point charges}.
\end{conjecture}
We have numerically verified Conjecture \ref{ABC Ising model} from $n=2$ to $8$ for the Cartesian product of 100 choices of $[f_{ij}]$ and 72 choices of $\{\omega_i\}$ (up to permutations we can assume $\omega_1\leqslant\omega_2\leqslant\omega_3$), as shown in Figures \ref{f_ijChoices.png} and \ref{omega_iChoices.png}. We also prove some special cases of Conjecture \ref{ABC Ising model} in Propositions \ref{arrangement in the degenerate case} and \ref{arrangement in the all nonpositive case}.

\begin{figure}[htbp]
\centering
\includegraphics[width=280pt]{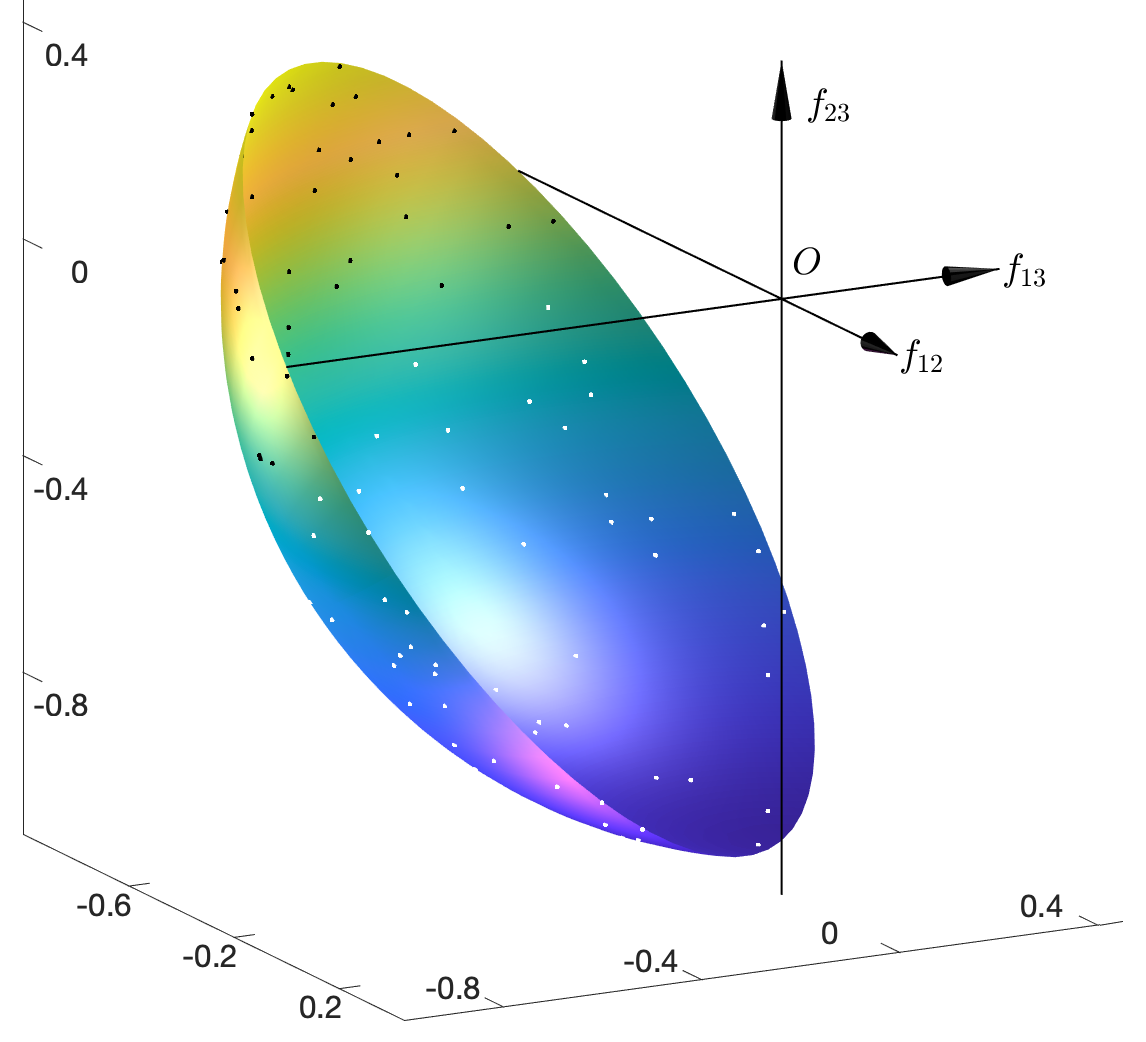}
\caption{Choices of $[f_{ij}]$ in the numerical verification. Colorful cap: the entire range of admissible $[f_{ij}]$ subject to $f_{12}^2+f_{13}^2+f_{23}^2=1$. Black and white dots: samples of $[f_{ij}]$ used in the numerical computation. Colors are only for visualization.}
\label{f_ijChoices.png}
\end{figure}

\begin{figure}[htbp]
\centering
\includegraphics[width=240pt]{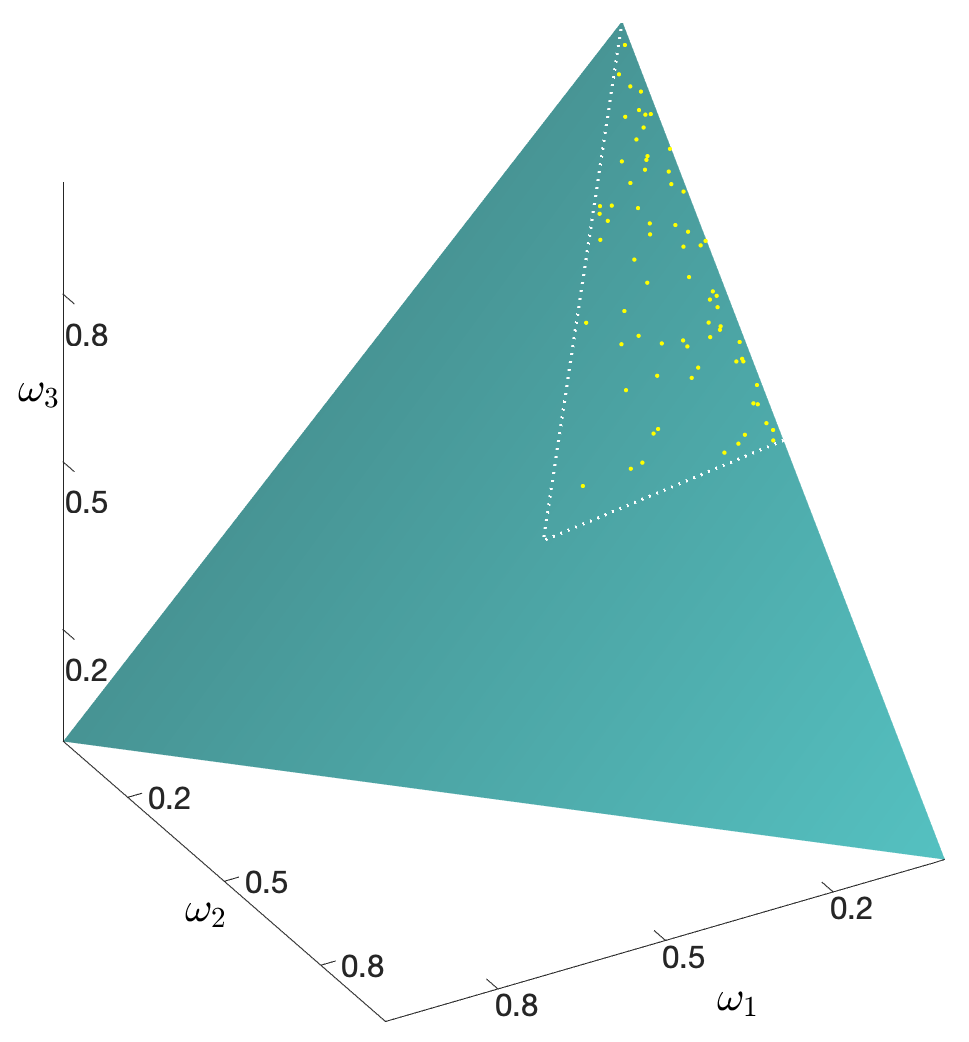}
\caption{Choices of $\{\omega_i\}$ in the numerical verification. {\color[RGB]{0,140,140}Teal} triangular plate: the entire range of $\{\omega_i\}$. {\color[RGB]{190,180,0}Yellow} dots: samples of $\{\omega_i\}$. Portion enclosed by dashed line segments: $\omega_1\leqslant\omega_2\leqslant\omega_3$.}
\label{omega_iChoices.png}
\end{figure}

\begin{proposition}
\label{arrangement in the degenerate case}
For $[f_{ij}]$ given by \eqref{Blend's matrix f_ij}, the minimizers of \eqref{electrostatic potential energy of A, B and C} are (up to translation and reflection) the following with any $l_k\geqslant0$ satisfying $\sum_{k=1}^nl_k=n$,
{\large\begin{equation*}
\text{\monomer{A}\monomer{B}}\underbrace{\text{\monomer{C}}\cdots\text{\monomer{C}}}_{\hspace{-20pt}l_1\;\text{consecutive \monomer{C}}\hspace{-20pt}}\text{\monomer{A}\monomer{B}}\underbrace{\text{\monomer{C}}\cdots\text{\monomer{C}}}_{l_2\hspace{47pt}\cdots\hspace{-55pt}}\text{\monomer{A}\monomer{B}}\cdots\text{\monomer{A}\monomer{B}}\underbrace{\text{\monomer{C}}\cdots\text{\monomer{C}}}_{l_n}{\normalsize\text{.}}
\end{equation*}}
\end{proposition}

\begin{proof}
For \eqref{Blend's matrix f_ij}, balls of type \monomer{C} do not engage in the interaction between charges. Similar to Proof of Proposition \ref{AB Ising model}, we can prove that within any optimal arrangement, \monomer{A} and \monomer{B} must be alternating if \monomer{C} is ignored. Now let us exclude the following segment from optimal arrangements:
{\large\begin{equation*}
\cdots \overset{x_1}{\text{\monomer{C}}}\,\overset{x_2}{\text{\monomer{A}}}\,\overset{x_3}{\text{\monomer{C}}} \cdots{\normalsize\text{.}}
\end{equation*}}
\noindent By translational invariance, we can assume that the above segment occupies the first 3 sites $\{x_k\}_{k=1}^3$. Define the potential created by charges at all other sites as
\begin{equation*}
V_2(x)=\sum_{m=4}^{3n}f_{\,i_2\,j_m}G(x,y_m),\quad x\in[0,1],\quad\text{where}\;i_2=1.
\end{equation*}
Among the remaining balls centered at $\{y_m\}_{m=4}^{3n}$, there are one more \monomer{B} than \monomer{A}, so $V_2$ is piecewise quadratic in $x$ with the coefficient of $x^2$ being $-\frac12$, and thus is strictly concave on $[x_0,x_4]$. Therefore by swapping the \monomer{A} at $x_2$ and the \monomer{C} at either $x_1$ or $x_3$, we can decrease $U$ by the amount
\begin{equation*}
V_2(x_2)-\min\Big\{V_2\Big(x_2\!-\!\frac{\omega_3}n\Big),\;V_2\Big(x_2\!+\!\frac{\omega_3}n\Big)\Big\}.
\end{equation*}
Analogously, we can rule out other segments like \monomer{C}\monomer{A}\monomer{B}\monomer{A}\monomer{C}, \monomer{C}\monomer{A}\monomer{B}\monomer{A}\monomer{B}\monomer{A}\monomer{C}, etc. Lastly, let us verify that different choices of $\{l_k\}$ yield the same $U$. Consider two \monomer{A}\monomer{B} dipoles separated by \monomer{C}
{\large\begin{equation*}
\cdots \overset{\hspace{2pt}x_1\hspace{-2pt}}{\text{\monomer{A}}}\,\text{\monomer{B}}\underbrace{\text{\monomer{C}}\cdots\text{\monomer{C}}}_{\hspace{-20pt}l\;\text{consecutive \monomer{C}}\hspace{-20pt}}\overset{\hspace{-3pt}x_{l+3}\hspace{-7pt}}{\text{\monomer{A}}}\,\text{\monomer{B}} \cdots{\normalsize\text{,}}
\end{equation*}}
between which the potential energy is
\begin{equation*}
\begin{aligned}
&G(x_1,x_{l+3})-G(x_1,x_{l+4})-G(x_2,x_{l+3})+G(x_2,x_{l+4})\\
=&-(x_2-x_1)\,(x_{l+4}-x_{l+3})\\
=&-\Big(\frac{\omega_1}{2n}+\frac{\omega_2}{2n}\Big)^2,
\end{aligned}
\end{equation*}
and is independent of $l$.
\end{proof}

\begin{remark}
In Proposition \ref{arrangement in the degenerate case}, if we penalize \monomer{A}\monomer{C} and \monomer{B}\monomer{C} interfaces with equal weights, then the minimizer is unique, i.e., \monomer{A}\monomer{B} $\cdots\;$\monomer{A}\monomer{B}\monomer{C} $\cdots\;$\monomer{C}. This is reminiscent of the results in \cite[Theorems 4 and 5]{van2008copolymer}, where \monomer{C} forms only one macrodomain.
\end{remark}

\begin{proposition}
\label{arrangement in the all nonpositive case}
For $[f_{ij}]$ given by \eqref{decomposition of [f_ij]} with nonpositive $f_{12}$, $f_{13}$ and $f_{23}$, the arrangement \monomer{A}\monomer{B}\monomer{C} $\cdots$ \monomer{A}\monomer{B}\monomer{C} minimizes \eqref{electrostatic potential energy of A, B and C}.
\end{proposition}

\begin{proof}
By \eqref{decomposition of [f_ij]}, $[f_{ij}]$ can be decomposed into three components, which are permutations of \eqref{Blend's matrix f_ij} and represent the interactions between sub-charges shown in Figure \ref{charges consist of quarks}. With the assumption $f_{12},f_{13},f_{23}\leqslant0$, all three components are simultaneously minimized by the cyclic arrangement \monomer{A}\monomer{B}\monomer{C} $\cdots$ \monomer{A}\monomer{B}\monomer{C}, according to Proposition \ref{arrangement in the degenerate case}.
\end{proof}

\section{Numerical computation of free energy in a 1-D periodic cell}
\label{Calculation of the free energy J in 1-D}

We now offer some details of the numerical computation to seek for the 1-D global minimizers of the free energy \eqref{definition of J}. For each pattern, we obtain the optimal layer widths numerically, with the initial guess (for optimization) having uniform layer widths. We use {\fontfamily{pcr}\selectfont fmincon}, a constrained local minimization function in MATLAB, with the constraints being the volume constraints and nonnegativity of layer widths. For the input argument {\fontfamily{pcr}\selectfont options}, we set {\fontfamily{pcr}\selectfont OptimalityTolerance}, {\fontfamily{pcr}\selectfont ConstraintTolerance} and {\fontfamily{pcr}\selectfont StepTolerance} to be $10^{-6}$, which should be sufficient for our purposes. Since patterns are defined modulo translation \cite[Definition 4.1]{ren2003triblock2}, we can avoid some redundant computation. For further acceleration, we use MATLAB's parallel tool {\fontfamily{pcr}\selectfont parfor} to work on multiple (e.g., 24) patterns simultaneously. 

We adopt a simple algorithm based on \eqref{general framework} to compute the long range term of \eqref{definition of J}. Noticing that the Green's function on $[0,1]$ with periodic boundary conditions is given by \cite[Equation (4.19)]{ren2003triblock2}
\begin{equation}
\label{expression of G on 1-D periodic cell}
G(x,y)=\frac{|x\!-\!y|^2}2-\frac{|x\!-\!y|}2+\frac1{12},\quad\text{for}\;x,y\in[0,1],
\end{equation}
we have $\int_{y_1}^{y_2}\int_{x_1}^{x_2}G(x,y)\dd{x}\dd{y}=F(x_2\!-\!y_1)-F(x_1\!-\!y_1)-F(x_2\!-\!y_2)+F(x_1\!-\!y_2)$, where
\begin{equation*}
F(x)=\frac{\big(1\!-\!|x|\big)^2x^2}{24}.
\end{equation*}
Now, given a pattern and the positions of interfaces, the following algorithm returns the free energy $J$.
\begin{lstlisting}[caption = {MATLAB code for computing $J$ in a 1-D periodic cell},mathescape]
function J = FreeEnergy(p, y, cij, gamij)
% p represents the pattern. For example, p = [1;2;3] for the shortest pattern $\text{\color[RGB]{34,139,33}\monomer{A}\monomer{B}\monomer{C}}$.
% y consists of the positions of interfaces. For example, y = [0;1/3;2/3;1].
% cij = [${\color[RGB]{34,139,33}c_{ij}}$] with ${\color[RGB]{34,139,33}c_{ii}}$ being 0.
% gamij = [${\color[RGB]{34,139,33}\gamma_{ij}}$].
ShortRangeTerm = sum( cij( sub2ind( size(cij), p, circshift(p,1) ) ) );
Fx24 = @(x) (1-abs(x)).^2.*x.^2; % An auxiliary function 24 times F.
LongRangeTerm = -sum(gamij(p,p).*diff(diff(Fx24(y-reshape(y,1,[])),1,1),1,2),'all')/24;
J = ShortRangeTerm + LongRangeTerm;
end
\end{lstlisting}
Although this algorithm is of complexity $O\big(${\fontfamily{pcr}\selectfont length(p)}\!\textasciicircum$2\big)$, it is not a bottleneck compared to the exhaustive search (among all the patterns) of complexity $O\big(2$\textasciicircum{\fontfamily{pcr}\selectfont length(p)}$\big)$, since the total complexity is the product of the above two and that of {\fontfamily{pcr}\selectfont fmincon} (for constrained optimization in {\fontfamily{pcr}\selectfont y}).

\section{Analytic calculation of free energy of 1-D periodic patterns}
\label{Calculating the free energy of periodic patterns}

For $\Omega=[0,1]$ with periodic boundary conditions, the long range term of \eqref{definition of J_epsilon} can be rewritten as
\begin{equation*}
\sum_{i=1}^3\sum_{j=1}^3\gamma_{ij}\int_0^1\big(u_i(x)\!-\!\omega_i\big)v_j(x)\dd{x}=\sum_{i=1}^3\sum_{j=1}^3\gamma_{ij}\int_0^1w_i(x)\,w_j(x)\dd{x},
\end{equation*}
where $v_j(x)=\int_0^1G(x,y)\,\big(u_j(y)\hspace{-0.5pt}-\hspace{-0.5pt}\omega_j\big)\dd{y}$ (so that $-v_j''=u_j\hspace{-0.5pt}-\hspace{-0.5pt}\omega_j$) and $w_j=-v_j'\,$. Denoting $\vec w=[w_1,w_2,w_3]^{\rm T}$, we can rewrite the above right-hand side as
\begin{equation*}
\int_0^1\vec w(x)^{\rm T}\,[\gamma_{ij}]\,\vec w(x)\dd{x}.
\end{equation*}
In \cite[Section 4]{ren2003triblock2}, Ren and Wei studied a local minimizer which is \monomer{A}\monomer{B}\monomer{C} identically repeating for $n$ times. In that case, $\vec w$ is periodic with period $1/n$, so one only needs to solve the following equation in one period in order to obtain the free energy
\begin{equation*}
\frac{\dd\vec w}{\dd x}=
\left\{
\begin{aligned}
&\vec e_1-\vec\omega,&&\;0<x<\frac{\omega_1}n,\\
&\vec e_2-\vec\omega,&&\;\frac{\omega_1}n<x<\frac{\omega_1\!+\!\omega_2}n,\\
&\vec e_3-\vec\omega,&&\;\frac{\omega_1\!+\!\omega_2}n<x<\frac 1n,
\end{aligned}
\right.
\quad\text{with}\;\int_0^1\vec w(x)\dd{x}=\vec0,
\end{equation*}
where $\vec\omega$ denotes $[\omega_1,\omega_2,\omega_3]^{\rm T}$, and $\{\vec e_1,\vec e_2,\vec e_3\}$ forms the standard basis. In this way, we can obtain \eqref{free energy of ABC}.

Analogously, for \monomer{A}\monomer{B}\monomer{A}\monomer{C} identically repeating for $n$ times (with all the \monomer{A} layers having the same width), we can obtain \eqref{free energy of ABAC} by solving the following equation
\begin{equation*}
\frac{\dd\vec w}{\dd x}=
\left\{
\begin{aligned}
&\vec e_1-\vec\omega,&&\;0<x<\frac{\omega_1}{2n},\\
&\vec e_2-\vec\omega,&&\;\frac{\omega_1}{2n}<x<\frac{\omega_1}{2n}\!+\!\frac{\omega_2}{n},\\
&\vec e_1-\vec\omega,&&\;\frac{\omega_1}{2n}\!+\!\frac{\omega_2}{n}<x<\frac{\omega_1\!+\!\omega_2}{n},\\
&\vec e_3-\vec\omega,&&\;\frac{\omega_1\!+\!\omega_2}n<x<\frac 1n,
\end{aligned}
\right.
\quad\text{with}\;\int_0^1\vec w(x)\dd{x}=\vec0.
\end{equation*}

\section{Alternative derivation of the admissibility conditions}
\label{Alternative derivation}

As mentioned in Remark \ref{remark on the incompressibility condition}-(i), there is an alternative derivation of the conditions in Theorem \ref{equivalent conditions for facilitation of charge neutrality} from the following three requirements:
\begin{itemize}
\item for $\vec u=\vec\omega$, the long range term \eqref{long range term} attains zero;
\item for any $\vec u$ satisfying the incompressibility condition $\vec u^{\rm T}\vec 1=1$, the long range term \eqref{long range term} is nonnegative;
\item $[\gamma_{ij}]$ is symmetric.
\end{itemize}
Under Neumann or periodic boundary conditions, we have $\int_{\Omega}G(\vec x,\vec y)\dd{\vec x}=0$ for any $\vec y\in\Omega$, so the first requirement is automatically satisfied and therefore does not lead to the condition $\vec\omega^{\rm T}[\gamma_{ij}]\vec\omega=0$. Moreover, the second requirement does not lead to the condition $[\gamma_{ij}]\succcurlyeq0$ because of the incompressibility condition. However, as explained in Proposition \ref{one choice satisfies the conditions}, there are many equivalent choices of $[\gamma_{ij}]$, and one of them satisfies $[\gamma_{ij}]\vec\omega=\vec0$ and $[\gamma_{ij}]\succcurlyeq0$ as desired.

\begin{proposition}
\label{one choice satisfies the conditions}
Under the incompressibility condition, among all the $[\gamma_{ij}]$ fulfilling the above three requirements and yielding the same long range term \eqref{long range term}, there is a unique one satisfying $[\gamma_{ij}]\,\vec\omega=\vec 0$. Meanwhile, it also satisfies $[\gamma_{ij}]\succcurlyeq0$.
\end{proposition}

\begin{proof}
Under the asssumption $u_1+u_2+u_3=1$, we have
$\vec u=[u_1,u_2,u_3]^{\rm T}=\mathcal A[u_1,u_2]^{\rm T}+[0,0,1]^{\rm T}$,
where $\mathcal A$ is given by \eqref{elimination by incompressibiblity}, and the second summand is constant. Since $\int_{\Omega}G(\vec x,\vec y)\dd{\vec x}=0$ for any $\vec y\in\Omega$, we can rewrite \eqref{long range term} as
\begin{equation*}
\int_{\Omega}\int_{\Omega}
\begin{bmatrix}
u_1(\vec x) & u_2(\vec x)
\end{bmatrix}
[\tilde\gamma_{ij}]
\begin{bmatrix}
u_1(\vec y)\\
u_2(\vec y)
\end{bmatrix}
G(\vec x,\vec y)\dd{\vec x}\dd{\vec y},
\end{equation*}
where $[\tilde\gamma_{ij}]=\mathcal A^{\rm T}[\gamma_{ij}]\mathcal A$. To ensure that the above integral is nonnegative, we need to impose the condition $[\tilde\gamma_{ij}]\succcurlyeq0$. (In fact, we can diagonalize $[\tilde\gamma_{ij}]$ into $Q^{\rm T}{\rm diag}(\lambda_1,\lambda_2)\,Q$, and rewrite the above integral as a quadratic form like \eqref{L^2 norm of electrostatic field}, from which it would be clear that $\lambda_1$ and $\lambda_2$ should be both nonnegative.)

By Lemma \ref{linear algebraic problem}, there is a class of equivalent choices of $[\gamma_{ij}]$, but only one of them satisfies $[\gamma_{ij}]\vec\omega=\vec0$. Such $[\gamma_{ij}]$ is given by \eqref{General matrix} and is positive semi-definite since we have $[\tilde\gamma_{ij}]\succcurlyeq0$.
\end{proof}

\begin{lemma}
\label{linear algebraic problem}
Define $T:S_3\rightarrow S_2$ to be $T(H)=\mathcal A^{\rm T}H\mathcal A$, where $S_m$ is the set of $m\times m$ symmetric real matrices, and
\begin{equation}
\label{elimination by incompressibiblity}
\mathcal A=
\begin{bmatrix}
1 &0\\
0 &1\\
-1 &-1
\end{bmatrix},
\end{equation}
then $T$ is surjective. The kernel of $T$ is $\big\{\vec1\vec p^{\rm T}+\vec p\vec1^{\rm T}\;\big|\;\vec p\in\mathbb R^3\big\}$. Given any $\vec w=[w_1,w_2,w_3]^{\rm T}\in\mathbb R^3$ with $\vec w^{\rm T}\vec 1=1$, in the quotient space $S_3/{\rm ker}(T)$, the equivalence class of any $H\in S_3$ has a unique representative $\tilde H$ satisfying $\tilde H\,\vec w=\vec 0$. This representative is given by $\tilde H=\mathcal B^{\rm T}T(H)\mathcal B$, where
\begin{equation}
\label{expand dimension}
\mathcal B=
\begin{bmatrix}
 1-w_1 & -w_1 & -w_1 \\
 -w_2 & 1-w_2 & -w_2
\end{bmatrix}.
\end{equation}
\end{lemma}

\begin{proof}
We can take $H=\begin{bmatrix}H_1 & 0\\0 & 0\end{bmatrix}$, where $H_1\in S_2$, then $T(H)=H_1$, so $T$ is surjective. Since $\vec1^{\rm T}\mathcal A=\vec0^{\rm T}$, based on the rank\textendash nullity theorem we know ${\rm ker}(T)=\big\{\vec1\vec p^{\rm T}+\vec p\vec1^{\rm T}\;\big|\;\vec p\in\mathbb R^3\big\}$. For any $H\in S_3$, we have
\begin{equation*}
(H+\vec1\vec p^{\rm T}+\vec p\vec1^{\rm T})\vec w=H\vec w+\vec1\vec w^{\rm T}\vec p+\vec p=H\vec w+(I+\vec1\vec w^{\rm T})\vec p,
\end{equation*}
where $I$ is an identity matrix. According to the Sherman\textendash Morrison formula, there is a unique $\vec p$ such that the above right-hand side vanishes. We can verify that $\tilde H=\mathcal B^{\rm T}T(H)\mathcal B$ is the corresponding representative within the equivalence class of $H$, by checking $\tilde H\vec w=\vec0$ and $T(\tilde H)=T(H)$, which are clear from $\mathcal B\vec w=\vec0$ and $\mathcal B\mathcal A=I$, respectively.
\end{proof}

\begin{remark}
The results in this section can be generalized from $\vec u\in\mathbb R^3$ to any dimension $\mathbb R^m$, by changing \eqref{elimination by incompressibiblity} into
\begin{equation*}
\mathcal A=\begin{bmatrix}
I_{m-1}\\
-\vec1^{\rm T}
\end{bmatrix},
\end{equation*}
and changing \eqref{expand dimension} into
\begin{equation*}
\mathcal B=
\begin{bmatrix}
I_{m-1} & \vec0
\end{bmatrix}
(I_{m}-\vec w\,\vec1^{\rm T}),
\end{equation*}
where $I_m$ is the $m\times m$ identity matrix, and $\vec1$ is a vector whose components are all 1.
\end{remark}

\end{document}